\documentclass[reqno, 11pt]{amsart}
\usepackage{amsmath,amssymb,mathrsfs,amsthm,amsfonts,mathtools}
\usepackage[inline]{enumitem} 				
\usepackage[usenames,dvipsnames]{xcolor}
\usepackage{graphicx} 
\usepackage[paper=letterpaper,margin=1in]{geometry}
\definecolor{hypercolor}{HTML}{003399}
\usepackage{hyperref}
 	
\hypersetup{colorlinks=true, linkcolor=hypercolor,citecolor=hypercolor}
\usepackage[margin=3pt, font=small]{caption}	
\newtheorem{thm}{Theorem}[section]
\newtheorem{lem}[thm]{Lemma}
\newtheorem{prop}[thm]{Proposition}
\newtheorem{cor}[thm]{Corollary}
\newtheorem{innercustomcor}{Theorem}
\newenvironment{customcor}[1]
  {\renewcommand\theinnercustomcor{#1}\innercustomcor}%
  {\endinnercustomcor}
\theoremstyle{definition}
\newtheorem{defn}[thm]{Definition}
\newtheorem{conj}[thm]{Conjecture}
\newtheorem{assu}[thm]{Assumption}
\newtheorem{rmk}[thm]{Remark}

\numberwithin{equation}{section}
\allowdisplaybreaks

\newcommand{\loc}{\mathrm{loc}}				
\newcommand{\ic}{\mathrm{ic}}				
\newcommand{\trace}{\mathrm{tr}}			
\newcommand{\HS}{\mathrm{HS}}				
\newcommand{\op}{\mathrm{op}}				

\newcommand{\e}{\varepsilon}

\newcommand{\C}{\mathbb{C}}
\newcommand{\N}{\mathbb{N}}

\newcommand{\R}{\mathbb{R}}
\newcommand{\Z}{\mathbb{Z}}
\newcommand{\soliton}{\mathcal{S}_{\star}}
\newcommand{\nosoliton}{\mathcal{N}\!\mathcal{S}_{\star}}

\newcommand{\Csp}{\mathcal{C}}
\newcommand{\Ccsp}{\mathcal{C}_\mathrm{c}}	
\newcommand{\Lsp}{\mathcal{L}}

\newcommand{\ind}{1}						
\newcommand{\hk}{\mathsf{k}}				
\newcommand{\hkkilled}{\mathsf{j}}			
\newcommand{\Kern}{\mathsf{K}}				
\newcommand{\maxfn}{\mathsf{Q}}				
\newcommand{\img}{\mathbf{i}}
\newcommand{\conserved}{\mathfrak{C}}
\DeclareMathOperator{\Li}{Li}
\DeclareMathOperator{\sech}{sech}

\newcommand{\w}{w}							
\newcommand{\q}{q}							
\newcommand{\p}{p}							
\newcommand{\qfn}{\mathsf{q}}				

\newcommand{\hkiter}[2]{\hk^{#2(*#1)}}		

\newcommand{\step}{\mathsf{Stp}}			
\newcommand{\plateau}{\mathsf{Plt}}			
\newcommand{\sdplateau}{\mathsf{Plt}^{\mathsf{sd}}}	
\newcommand{\puncher}{\mathsf{Pchr}}		
\newcommand{\actregion}{\Omega}				
\newcommand{\bdyterm}{\mathsf{Bdy}}			
\newcommand{\tcterm}{\mathsf{Tc}}			

\newcommand{\qic}{\q_\mathrm{ic}}			
\newcommand{\qicfn}{f_\mathrm{ic}}			
\newcommand{\rateic}{\beta_\mathrm{ic}}		
\newcommand{\ptc}{\p_\mathrm{tc}}			
\newcommand{\ii}{i}							
\newcommand{\jj}{j}	
\newcommand{\mm}{\mathfrak{m}}	
\newcommand{\nn}{\mathfrak{n}}	

\newcommand{\U}{\mathsf{U}}					
\newcommand{\V}{\mathsf{V}}					
\renewcommand{\S}{\mathsf{S}}				
\newcommand{\Sa}{\mathsf{a}}				
\newcommand{\Sb}{\mathsf{b}}				
\newcommand{\Sr}{\mathsf{r}}				
\newcommand{\Sat}{\til{\mathsf{a}}}		
\newcommand{\Sbt}{\til{\mathsf{b}}}		
\newcommand{\Srt}{\til{\mathsf{r}}}		

\newcommand{\jost}{\mathsf{J}}				
\newcommand{\X}{\mathsf{X}}					
\newcommand{\G}{\mathsf{G}}					
\newcommand{\up}{{\scriptscriptstyle\mathrm{up}}}				
\newcommand{\md}{{\scriptscriptstyle\mathrm{mid}}}		
\newcommand{\lw}{{\scriptscriptstyle\mathrm{lw}}}
\newcommand{\region}{D}						

\newcommand{\g}{\mathsf{g}}
\newcommand{\Za}{\mathcal{Z}_{\up}(\Sa)}
\newcommand{\Zat}{\mathcal{Z}_{\lw}(\Sat)}

\newcommand{\fourier}{\mathscr{F}}			
\newcommand{\Srho}{\rho}					
\newcommand{\Srhot}{\til{\rho}}			
\newcommand{\SXi}{\mathsf{\Xi}}				
\newcommand{\SGamma}{\mathsf{\Gamma}}		
\newcommand{\dressf}{g}						

\newcommand{\oid}{\mathbf{1}}				
\newcommand{\ozero}{\mathbf{0}}				
\newcommand{\smallp}{{\scriptscriptstyle+}}
\newcommand{\smallm}{{\scriptscriptstyle-}}
\newcommand{\smallpm}{{\scriptscriptstyle\pm}}
\newcommand{\oindp}{\oid_\smallp}
\newcommand{\oindm}{\oid_\smallm}
\newcommand{\oindpm}{\oid_\smallpm}
\newcommand{\orho}{{\boldsymbol\rho}}
\newcommand{\orhot}{\til{\boldsymbol\rho}}

\newcommand{\oXi}{\mathbf{\Xi}}				
\newcommand{\oGamma}{\mathbf{\Gamma}}		
\newcommand{\oI}{\mathbf{I}}				
\newcommand{\oIp}{\mathbf{I}_{\scriptscriptstyle+}}

\newcommand{\oIpm}{\mathbf{I}_{\scriptscriptstyle\pm}}

\newcommand{\rankone}{h}					
\newcommand{\projection}[1]{\boldsymbol{\pi[}#1\boldsymbol{]}}	
\newcommand{\reflection}{\oid_{\scriptscriptstyle\leftrightarrow}} 
\newcommand{\oremainder}{\mathbf{remainder}}
\newcommand{\frhoI}{\Srho_{\mathrm{I}}}
\newcommand{\frhoII}{\Srho_{\mathrm{II}}}

\newcommand{\rateone}{\psi_{\star}}
\newcommand{\ratenosoliton}{\psi_{\star,\mathrm{ns}}}
\newcommand{\ratesoliton}{\psi_{\star,\mathrm{s}}}
\newcommand{\salpha}{\alpha_{\scriptscriptstyle *}}	
\newcommand{\sgamma}{\gamma_{\scriptscriptstyle *}}	

\renewcommand{\P}{\mathbb{P}}				
\newcommand{\E}{\mathbb{E}}					
\newcommand{\cc}{{\scalebox{.45}{$\complement$}}}
\renewcommand{\d}{\mathrm{d}}				
\newcommand{\sign}{\mathrm{sgn}}			
\newcommand{\supp}{\mathrm{supp}}			
\newcommand{\im}{\mathrm{Im}}				

\newcommand{\norm}[1]{\Vert #1\Vert}
\newcommand{\Norm}[1]{\big\Vert #1\big\Vert}

\newcommand{\ip}[1]{\langle #1\rangle}

\newcommand{\IP}[1]{\Big\langle #1\Big\rangle}

\newcommand{\Matrix}[1]{\begin{pmatrix}#1\end{pmatrix}}
\newcommand{\tracen}{\mathrm{tr}}
\newcommand{\transp}{\mathrm{t}}

\renewcommand{\bar}{\overline}
\newcommand{\und}{\underline}
\newcommand{\til}{\widetilde}

\newcommand*{\Cdot}{{\raisebox{-0.5ex}{\scalebox{1.8}{$\cdot$}}}} 

\usepackage{newtxtext}

\title{Integrability in the weak noise theory}
\author{Li-Cheng Tsai}

\address[Li-Cheng Tsai]{\hspace{1.5pt} Department of Mathematics, University of Utah}
\email{lctsai.math@gmail.com}


\begin{document}
\begin{abstract}
We consider the variational problem associated with the Freidlin--Wentzell Large Deviation Principle (LDP) for the Stochastic Heat Equation (SHE).
For a general class of initial-terminal conditions, we show that a minimizer of this variational problem exists, and any minimizer solves a system of imaginary-time Nonlinear Schr\"{o}dinger equations. 
This system is integrable.
Utilizing the integrability, we prove that the formulas from the physics work~\cite{krajenbrink21} hold for every minimizer of the variational problem.
As an application, we consider the Freidlin--Wentzell LDP for the SHE with the delta initial condition.
Under a technical assumption on the poles of the reflection coefficients, we prove the explicit expression for the one-point rate function that was predicted in the physics works \cite{ledoussal16short,krajenbrink21}. 
Under the same assumption, we give detailed pointwise estimates of the most probable shape in the upper-tail limit.
\end{abstract}

\maketitle

\section{Introduction and results}
\label{s.intro}

The variational principle, or the least action principle, offers a framework for the study of the Large Deviation Principle (LDP) for a stochastic system~\cite{freidlin98,touchette09}. Often associated with such an LDP is a variational problem, along with its Euler--Lagrange and Hamilton equations. The minimum in the variational problem gives the rate function in the LDP, and the minimizers give the candidates for the most probable shape. By analyzing the variational problem, one can often extract information on the rate function and the most probable shape and gain an insight on how certain deviations are realized.

In this paper we study the variational problem associated with the Freidlin--Wentzell LDP for the Stochastic Heat Equation (SHE), or equivalently the Kardar--Parisi--Zhang (KPZ) equation. The KPZ equation \cite{kardar86} describes the evolution of a randomly growing interface; the SHE describes the partition function of a directed polymer in a continuum random environment. The two equations are equivalent in the sense that the Hopf--Cole transform maps the KPZ equation to the SHE. Both equations have been widely studied in mathematics and physics thanks in part to their connections to various physical systems and their integrability; we refer to \cite{quastel2011introduction,corwin2012kardar,quastel2015one,chandra2017stochastic,corwin2019} for reviews on the mathematical literature related to the KPZ equation.

The \emph{weak noise theory} in physics can be viewed as the study of the Freidlin--Wentzell LDP for the SHE and the KPZ equation through the associated variational problem. The theory has seen much progress. Behaviors of the one-point rate function for various initial conditions and boundary conditions have been predicted \cite{kolokolov07, kolokolov08, kolokolov09, meerson16, kamenev16, meerson17, meerson18, smith18exact, smith18finitesize, asida19, smith19}, some of which recently proven \cite{lin21,gaudreaulamarre21}; an intriguing symmetry breaking and second-order phase transition has been discovered in~\cite{janas16,smith18} via numerical means and analytically derived in \cite{krajenbrink17short,krajenbrink21flat}.

Some recent developments brought the weak noise theory and integrable PDEs together. As pointed out in \cite[Appendix~B]{janas16}, at the level of the SHE, the Hamilton equations associated with the variational problem form a system of imaginary-time Nonlinear Schr\"{o}dinger (NLS) equations, and the system is integrable. Based on the integrability, the physics works \cite{krajenbrink21, krajenbrink21flat} derived formulas for the minimizer(s), which can be used to infer the most probable shape. These new developments are fast growing in the physics literature (for example the recent works~\cite{bettelheim22,krajenbrink22,mallick22}) and they unlock much potential for further study. In particular, the formulas can give access to very precise information on the LDP.

Our main result puts the connection of the weak noise theory and integrable PDEs on a mathematically rigorous ground. The first part of the main result, Theorem~\ref{t.NS}, shows that a minimizer of the variational problem exists and that any minimizer solves the NLS equations. The second part of the main result, Theorem~\ref{t.formula}, gives explicit formulas for any minimizer in terms of its scattering coefficients. Together with the Freidlin--Wentzell LDP for the SHE proven in \cite{lin21}, the main result here gives a mathematically rigorous description of the LDP in terms of the NLS equations and the formulas.

We mention some challenges in the proof of these theorems. One challenge in the proof of Theorem~\ref{t.NS} is to show that the terminal condition of one of the NLS equations is a sum of delta functions. The proof of this property does not just follow from the calculus of variation alone and requires further analysis of the variational problem; see Section~\ref{s.results.tc}. The formulas in Theorem~\ref{t.formula} were derived in \cite{krajenbrink20det,krajenbrink21} at a physics level of rigor, along with a numerical scheme for evaluating the formulas. 
Here, we take the route through a Riemann--Hilbert problem. 
(See the paragraph after Remark~\ref{r.1-1.ratefn} for a more detailed description.)
In the presence of the poles of the reflection coefficients, rigorously solving the problems could require delicate analysis. Here we use Fourier transform to solve the problem (for a given minimizer of the variational problem) in one shot without requiring any additional information on the poles. The result holds for any super-exponentially-decaying initial condition and any finite-point terminal condition.
Other types of initial conditions (flat for example, which was treated in \cite{krajenbrink21flat}) may be considered, and we leave them to future work.

As an application, we consider what we call the 1-to-1 initial-terminal condition. It corresponds to having the delta initial condition for the SHE and conditioning the value of its solution at a later time at the origin. Under Assumption~\ref{assu.poles} (which is discussed in the next paragraph), we prove in Theorem~\ref{c.1-1} that the minimizer is unique and prove the explicit formulas from the physics works \cite{ledoussal16short,krajenbrink21}. The formulas describe the minimum and the minimizer of the variational problem, or equivalently the rate function and the most probable shape of the SHE. The rate function exhibits an intriguing `flip' phenomenon, which is discussed in Remark~\ref{r.1-1.ratefn}. An interesting phenomenon, discovered and explicitly described in the physics work~\cite{krajenbrink21}, is that the most probable shape transitions from a non-solitonic solution to a solitonic solution of the NLS equations; see the second paragraph in Section~\ref{s.results.1-1} for a brief review. Theorem~\ref{c.1-1} establishes this phenomenon on a mathematically rigorous ground. (See Remark~\ref{r.soliton}.)

The challenge in proving Theorem~\ref{c.1-1} is to rigorously obtain the reflection coefficients. 
In Section~\ref{s.1-1}, we formulate a \emph{scalar} Riemann--Hilbert problem for the reflection coefficients, which can be solved explicitly.
However, there are infinitely many such solutions, and each solution is a candidate for the reflection coefficients of the actual minimizer.
The challenge is hence to rule out all the `non-physical' candidates.
Our analysis rules out some but not all, and we impose Assumption~\ref{assu.poles} to exclude those that we did not rule out.

To showcase the applications of the explicit formulas, in Corollary~\ref{c.1-1.large} and still under Assumption~\ref{assu.poles}, we give detailed estimates of the minimizer under a scaling that corresponds to the so-called upper-tail limit. Under Assumption~\ref{assu.poles}, the estimates confirm the physics prediction that a soliton controls the most probable shape~\cite{kolokolov07,meerson16,krajenbrink21}. 
Further, from the perspective of hydrodynamic large deviations \cite{jensen00,varadhan04}, the soliton produces a non-entropy shock, and it is interesting to explore how such a production occurs at a greater generality; see Section~\ref{s.results.burgers} for the discussion about this.

Let us briefly compare the study of the NLS equations in this paper with the ones in the literature.
Solved by the inverse scattering transform in \cite{shabat72}, the real-time NLS equation and its variants are among the most studied integrable PDEs.
We refer to \cite{ablowitz04,ablowitz06,faddeev07} for reviews on these subjects.
Most of the algebraic properties (the Lax pair, the conserved quantities, etc.) in this paper are the same as or similar to that in the literature.
On the other hand, many analytic properties differ.
For example, the unique solvability of the equations does not hold here (discussed in Section~\ref{s.results.uniqueness}); the equations could exhibit the focusing, defocusing, or mixed behavior depending on the terminal condition (discussed in Remark~\ref{r.NS}\ref{r.NS.focusing.type}).
These differences mostly arise from how we pose the NLS equations: with one initial condition and one terminal condition.
Such conditions are natural in the LDP setting and are necessary from the perspective of analysis (explained in Remark~\ref{r.NS}\ref{r.NS.backward}). 
We note that our large-scale analysis in Corollary~\ref{c.1-1.large} is based on exact formulas, and it is interesting to investigate whether the method of nonlinear steepest descent~\cite{deift93,deift93long} can be applied to the imaginary-time equations here.

We conclude the introduction by mentioning some related works. 

Recently, there has been much interest in the LDPs of the KPZ equation in mathematics and physics.
Several strands of methods, based in part on exact formulas of observables of the KPZ equation, produce detailed information on the one-point tail probabilities and the one-point rate function.
This includes the physics works \cite{ledoussal16short,ledoussal16long,krajenbrink17short,sasorov17,corwin18,krajenbrink18half,krajenbrink18simple,krajenbrink18systematic,krajenbrink18simple,krajenbrink19thesis,krajenbrink19linear,ledoussal19kp}, the simulation works \cite{hartmann18,hartmann19,hartmann20,hartmann21}, and the mathematics works \cite{corwin20general,corwin20lower,bothner21,cafasso21airy,das21,das21iter,kim21,lin21half,cafasso22,ghosal20,tsai18}.

Incidentally, the NLS equations studied here appeared in the literature of mean-field games~\cite{huang06,lasry07}.
Central in that literature are the Hamilton--Jacobi-Fokker--Planck (HJ-FP) equations.
A particular instance of the HJ-FP equations is equivalent to the NLS equations through the Hopf--Cole transform; see \cite[Section~2.2]{gomes16} for example.
We refer to \cite{lions07,cardaliaguet10,gueant11} and the references therein for the literature on mean-field games and note that the physics works~\cite{swiecicki16,bonnemain21} studied the mean-field games from the perspective of the imaginary-time NLS equations.

Let us point out a major difference between our setting and the mean-field-game setting.
In our setting, what we call $ \w = \w(t,x) $ (see Section~\ref{s.results.variational}) may not have a definite sign; in the mean-field-game setting, the function $ \w $ is assumed to be non-positive.
This difference has significant implications on the behaviors of the NLS equations.
For example, the unique solvability of the equations does not hold here (see Section~\ref{s.results.uniqueness}), but it holds when $ \w $ is non-positive (see~\cite[Section~1.1.5]{gomes16} for example).
Also, even though the scaling in Corollary~\ref{c.1-1.large} can be understood as a vanishing viscosity limit, the limit is carried out with $ \w $ being positive. 
Such a limit differs from the one studied in the mean-field-game literature where $ \w $ is non-positive.
The former limit generally gives a non-entropy solution of the limiting equation (see Section~\ref{s.results.burgers}).

\subsection*{Outline.}
In Section~\ref{s.results}, we state the results and carry out some discussions about them.
Section~\ref{s.NS} makes up the proof of Theorem~\ref{t.NS}. 
In Section~\ref{s.solving}, we recall the relevant properties of the forward scatter transform and the Riemann--Hilbert problem of the NLS equations, and then solve the Riemann--Hilbert problem to prove Theorem~\ref{t.NS}.
In Section~\ref{s.1-1}, we consider the 1-to-1 initial-terminal condition and prove Theorem~\ref{c.1-1}.
In Section~\ref{s.1-1.large}, we perform the asymptotic analysis to prove Corollary~\ref{c.1-1.large}.
To streamline the presentation, we place some peripheral and technical parts of the proof in the appendices and make explicit references to them in the main text.
Finally, in Appendix~\ref{s.a.classical.field.theory}, we give a physics derivation of the NLS equations. 
The last appendix is not used elsewhere in the text, and is in place just to offer a perspective from classical field theory.

\subsection*{Acknowledgements.}
I thank William Feldman, Pei-Ken Hung, and Alexandre Krajenbrink for useful discussions, thank Alexandre Krajenbrink, Pierre Le Doussal, Yier Lin, and two anonymous reviewers for helping improve the presentation of this paper, and thank William Feldman and Yier Lin for pointing out the literature on mean-field games to me.
LCT was partially supported by the NSF through DMS-1953407 and DMS-2243112 and by the Sloan Research Fellowship.

\section{Results and discussions}
\label{s.results}

\subsection{The variational problem}
\label{s.results.variational}

We begin by formulating the variational problem of interest. 
Fix a time horizon $ T<\infty $.
Given a $ \w \in \Lsp^2([0,T]\times\R) $, consider the PDE
\begin{align*}
	\partial_t q = \tfrac12 \partial_{xx} q + \w q,
	\qquad
	(t,x) \in (0,T)\times\R,
\end{align*}
with a fixed initial condition $ \qic $ that will be specified later.
For each such $ \w $, this PDE has a unique solution (in the sense specified in Definition~\ref{d.duhamelsense}), so we view $ \q $ as being driven by $ \w $ and write $ \q = \qfn[\w] = \qfn[\w](t,x) $.
Fixing $ \xi_1 < \ldots < \xi_{\mm} \in \R $ and $ \alpha_1,\ldots,\alpha_{\mm} \in \R $, we look for those  $ \w $s such that $ \qfn[\w](T,\xi_\ii) = e^{\alpha_\ii} $, for $ \ii=1,\ldots,\mm $.
Among such $ \w $s, the variational problem seeks to minimize the squared $ \Lsp^2 $ norm $ (\norm{\w}_{2;[0,T]\times\R})^2 := \int_{[0,T]\times\R} \d t \d x\, \w^2 $, namely
\begin{align}
	\label{e.minimization}
	\inf\big\{ \tfrac{1}{2} (\norm{\w}_{2;[0,T]\times\R})^2 \, : \, \qfn[\w](T,\xi_{\ii}) = e^{\alpha_i}, i=1,\ldots,\mm \big\}.
\end{align}
We refer to $ (\xi_i,e^{\alpha_i})_{\ii=1}^{\mm} $ as the \textbf{terminal condition} of $ \q $.

This variational problem is associated with the Freidlin--Wentzell LDP for the SHE, or equivalently the KPZ equation.
Let $ \eta $ denote the spacetime white noise, let $ \e>0 $ denote a small parameter, and consider the SHE driven by a white noise
\begin{align}
	\label{e.she}
	&&&\partial_t Z_\e = \tfrac12 \partial_{xx} Z_\e + \sqrt{\e} \eta \, Z_\e ,
	&&
	Z_\e(0,\Cdot) = \qic,
	\
	(t,x) \in (0,T]\times\R.
\end{align}
It was proven in \cite[Proposition~1.7]{lin21} that $ Z_\e $ enjoys the LDP with speed $ \e^{-1} $ and the rate function
\begin{align*}
	I(\q) 
	:= 
	\inf\big\{ \tfrac{1}{2} (\norm{\w}_{2;[0,T]\times\R})^2 \, : \, \qfn[\w] = q \big\}.
\end{align*}
Since the SHE yields the KPZ equation through the inverse Hopf--Cole transform, this LDP can also be interpreted as an LDP for the KPZ equation.
Given the LDP, the infimum in~\eqref{e.minimization} yields the rate of the event $ \bigcap_{\ii=1}^{\mm}\{Z_\e(T,\xi_{\ii}) \in(-u+e^{\alpha_{\ii}},e^{\alpha_{\ii}}+u) \} $, with $ \e\to 0 $ first and $ u\to 0 $ later.

Let us specify the initial condition for $ \q $.
It is taken to be a sum of delta functions and a function that satisfies an exponential bound. 
Fix $ \zeta_1<\ldots<\zeta_{\nn} \in\R $, $ \beta_1,\ldots,\beta_{\nn} \in\R $, and $ \rateic \in \{-\infty\}\cup\R $, and set
\begin{align}
	\label{e.qic}
	\qic := \sum_{\jj=1}^{\nn} e^{\beta_{\jj}} \delta_{\zeta_{\jj}} + \qicfn,
	\qquad
	0 \leq \qicfn,
	\quad
	\sup_{x\in\R} \qicfn(x) /e^{\beta|x|} < \infty, \forall \beta>\rateic.
\end{align}
Note that $ \qic $ is nonnegative.
We allow $ \nn=0 $ or $ \qicfn = 0 $ but assume that $ \qic $ is not identically zero, namely
\begin{align}
	\label{e.qic>0}
	\int_{\R} \d x \, \qic(x) >0.
\end{align}

\subsection{Main result}
\label{s.results.main}

The main result consists of two parts: Theorems~\ref{t.NS} and \ref{t.formula}.

\begin{thm}
\label{t.NS}
A minimizer of the variational problem~\eqref{e.minimization} exists.
Let $ \w $ be any minimizer and set $ \q := \qfn[\w] $ and $ \p := \w/\q $.
The functions $ \q $ and $ \p $ are in $ \Csp^\infty((0,T)\times\R) $, satisfy the bounds given in \eqref{e.bd.q.dx}, and solve the imaginary-time NLS equations 
\begin{align}
	\label{e.NS.q}
	\partial_t \q &= \tfrac12 \partial_{xx} \q + \w\q = \tfrac12 \partial_{xx} \q + \p\q^2,
	&
	\q(0,\Cdot) &= \qic,
\\
	\label{e.NS.p}
	-\partial_t \p &= \tfrac12 \partial_{xx} \p + \p\w = \tfrac12 \partial_{xx} \p + \p^2\q,
	&
	\p(T,\Cdot) &= \ptc := \sum_{\ii=1}^{\mm} \gamma_\ii \delta_{\xi_\ii}, 
\end{align}
in the Duhamel sense (Definition~\ref{d.duhamelsense}), for some $ \gamma_i\in\R $, and in the classical sense in $ (0,T)\times\R $.
\end{thm}

\begin{rmk}\label{r.NS}
\begin{enumerate}[leftmargin=20pt,label=(\alph*)]
\item[]

\item \label{r.NS.backward}
The leading order derivatives in \eqref{e.NS.p} together give a \emph{backward} heat equation, so the equation needs to be solved backward in time, hence the terminal condition there.

\item \label{r.NS.gamma.value}
The values of $ \gamma_1,\ldots,\gamma_{\mm} $ are not determined in Theorem~\ref{t.NS}.
Determining the values is a major task in applications, as demonstrated in Theorem~\ref{c.1-1}.

\item \label{r.NS.focusing.type}
Depending on the values of $ \gamma_{1} $, \ldots , $ \gamma_{\mm} $, Equations~\eqref{e.NS.q}--\eqref{e.NS.p} could exhibit the focusing, defocusing, or mixed behavior.
When $ \gamma_i \geq 0 $ for all $ i $, both $ \p $ and $ \q $ are nonnegative, and the equations behave as the focusing NLS equations as written.
When $ \gamma_i \leq 0 $ for all $ i $, it is natural to change $ \p\mapsto -\p $ so that $ \p $ and $ \q $ are both nonnegative.
Upon the preceding change, the equations become the defocusing NLS equations.
When the $ \gamma_i $s do not all have the same sign, we expect \eqref{e.NS.q}--\eqref{e.NS.p} to exhibit mixed behaviors of the focusing and the defocusing equations.
\end{enumerate}
\end{rmk}

Turning to the second part of the main result, we fix $ \rateic = -\infty $.
This condition amounts to requiring $ \qicfn $ to decay faster than any exponential rate.
A quintessential instance is when $ \qicfn \equiv 0 $, or equivalently when $ \qic $ is a finite sum of delta functions.
As is well-known in the study of integrable PDEs, the analytic properties of the scattering coefficients depend on the $ |x|\to\infty $ behavior of the initial condition.
The initial conditions considered here make the scattering coefficients entire (analytic on $ \C $).

To state the second part of the main result we require some notation.
Fix any $ (\p,\q) $ as in Theorem~\ref{t.NS} with $ \rateic = -\infty $, and define the Lax pair
\begin{align}
	\label{e.lax}
	\U 
	:=
	\Matrix{ -\frac{\img}{2}\lambda &  -\p \\ \q & \frac{\img}{2}\lambda },
	\qquad
	\V
	:=
	\Matrix{ \frac14 \lambda^2 - \frac12 \p \q & -\frac{\img}{2}\lambda \p + \frac{1}{2}\partial_x \p \\ \frac{\img}{2}\lambda \q + \frac{1}{2}\partial_x \q & -\frac14\lambda^2 + \frac12 \p \q },
\end{align}
where $ \lambda\in\C $ denotes a spectral parameter.
This Lax pair checks the zero-curvature condition $ \partial_t \U - \partial_x \V + [\U,\V] =0 $.
From the matrix $ \U $, construct the scattering coefficients $ \Sa(\lambda) $, $ \Sb(\lambda) $, $ \Sat(\lambda) $, and $ \Sbt(\lambda) $ (defined in Section~\ref{s.solving.scatter}) and the reflection coefficients 
\begin{align*}
	\Sr(\lambda) := \Sb(\lambda)/\Sa(\lambda),
	\qquad
	\Srt(\lambda) := \Sbt(\lambda)/\Sat(\lambda).
\end{align*}
As will be explained in Section~\ref{s.solving.scatter}, the scattering coefficients are entire, and $ \Sa(\lambda) $ and $ \Sat(\lambda) $ have at most finitely many zeros in the upper and lower half planes respectively.
Fix any $ v_0\in(0,\infty) $ large enough so that $ \Sa(\lambda) $ and $ \Sat(\lambda) $ have no zeros on $ \{\im(\lambda) \geq v_0 \} $ and on $ \{\im(\lambda) \leq - v_0 \} $ respectively.
Define the Fourier transforms of the dressed reflection coefficients
\begin{align*}
	\Srho(s;t,x) := \int_{\R+\img v_0} \frac{\d\lambda}{2\pi} \, e^{\img\lambda s}  \Sr(\lambda) e^{-\lambda^2t/2+\img\lambda x},
	\qquad
	\Srhot(s;t,x) := \int_{\R-\img v_0} \frac{\d\lambda}{2\pi} \, e^{\img\lambda s}  \,\Srt(\lambda) \, e^{+\lambda^2t/2-\img\lambda x}.
\end{align*}
Let $ \orho_{t,x} $ denote the operator acting on $ \Lsp^2(\R) $ by $ (\orho_{t,x}\phi)(s) := \int_{\R} \d s' \Srho(s-s';t,x)\phi(s') $ and similarly for $ \orhot_{t,x} $.
Let $ \oid $ denote the identity operator on $ \Lsp^2(\R) $ and let $ \oindpm $ act by $ (\oindpm\phi)(s) := \ind_{\{\pm s>0\}}\phi(s) $.
We say an operator $ \mathbf{f} $ on $ \Lsp^2(\R) $ has an almost continuous kernel if the kernel is of the form $ \ind_{\{\pm s>0\}} g(s,s') \ind_{\{\pm s'>0\}} $ or $ g(s,s') \ind_{\{\pm s'>0\}} $ or $ \ind_{\{\pm s>0\}} g(s,s') $ or $ g(s,s') $, for some continuous $ g $.
We write $ {}_0[\mathbf{f}]_0 := \lim_{(s,s')\to(0^\pm,0^\pm)} g(s,s') $ for the value of the kernel `evaluated at $ (0,0) $', where the signs in the limit are chosen consistently with the signs of the indicators in the kernel.

\begin{thm}\label{t.formula}
Fix any $ (\p,\q) $ as in Theorem~\ref{t.NS} with $ \rateic = -\infty $, and use the preceding notation. 
For all $ (t,x)\in(0,T)\times\R $, the operators $ (\oid-\oindp\,\orho_{t,x}\oindm\,\orhot_{t,x}\,\oindp) $ and $ (\oid-\oindm\,\orhot_{t,x}\,\oindp\,\orho_{t,x}\,\oindm) $ have bounded inverses, and $ (\oindp\,\orho_{t,x}\,\oindm\,\orhot_{t,x}\,\oindp) $ is trace-class.
Further,
\begin{align}
	\label{e.t.formula.p}
	\p(t,x) 
	&=
	- \prescript{}{0}{\big[} \orhot_{t,x} \oindp (\oid-\oindp\,\orho_{t,x}\oindm\,\orhot_{t,x}\,\oindp)^{-1} \big]_{0}	 
	=
	- \prescript{}{0}{\big[} (\oid-\oindm\,\orhot_{t,x}\oindp\,\orho_{t,x}\,\oindm)^{-1} \oindm \orhot_{t,x} \big]_{0}\,,
\\	
	\label{e.t.formula.q}
	\q(t,x) 
	&= 
	\prescript{}{0}{\big[} (\oid-\oindp\,\orho_{t,x}\oindm\,\orhot_{t,x}\,\oindp)^{-1} \oindp \orho_{t,x} \big]_{0}
	=
	\prescript{}{0}{\big[} \orho_{t,x} \oindm (\oid-\oindm\,\orhot_{t,x}\oindp\,\orho_{t,x}\,\oindm)^{-1} \big]_{0}\,,
\\
	\label{e.t.formula.det}
	(\p\q)(t,x)
	&=
	\w(t,x)
	=
	\partial_{xx} \log\det\big( \oid-\oindp\,\orho_{t,x}\,\oindm\,\orhot_{t,x}\,\oindp \big),
\end{align}
where every operator in \eqref{e.t.formula.p}--\eqref{e.t.formula.q} has an almost continuous kernel so that $ {}_0 [\ldots] {}_0 $ is well-defined.
\end{thm}

\subsection{Application to the 1-to-1 initial-terminal condition}
\label{s.results.1-1}

As an application, consider
\begin{align}
	\label{e.1-1}
	&& \qic = \delta_0, && \mm=1, && \xi_1=0, && \alpha_1:=\alpha \in\R. &&
\end{align}
We call this initial-terminal condition `1-to-1' because $ \qic $ and $ \ptc $ each have one point in their supports.

Let us recall some important properties of the minimizer that are predicted in the physics work~\cite{krajenbrink21}.
The minimizer should behave differently for different values of $ \alpha_1 := \alpha\in\R $:
For smaller $ \alpha $, the minimizer should produce a non-solitonic solution of the NLS equations; for larger $ \alpha $, the minimizer should produce a solitonic solution.
Let $ \Li_\mu $ denote the polylogarithm.
The ranges are
\begin{subequations}
\label{e.soliton.range}
\begin{align}
	\sqrt{T/2}\,e^{\alpha} \in \nosoliton &:= ( -\infty, c_{\star} ],
\\
	\sqrt{T/2}\,e^{\alpha} \in \soliton &:= ( c_{\star} , \infty ),
	\quad
	c_\star := \tfrac{1}{\sqrt{4\pi}} \Li_{5/2}'(1) = 0.7369\ldots.
\end{align}
\end{subequations}
That is, the minimizer should transition from a non-solitonic solution to a solitonic solution when $ \sqrt{T/2}\,e^{\alpha} $ passes the threshold $ c_{\star} $.
Recall that the terminal condition $ \ptc $ contains a parameter $ \gamma_1 := \gamma $, whose value is yet to be determined.
To determine the value of $ \gamma $, set 
\begin{subequations}
\label{e.1-1.rate}
\begin{align}
	\ratenosoliton: (-\infty,1]\to\R,
	&&
	\ratenosoliton(\gamma) &:= \tfrac{1}{\sqrt{4\pi}} \Li_{5/2}(\gamma),
\\
	\ratesoliton: (0,1)\to\R,
	&&
	\ratesoliton(\gamma) &:= \tfrac{1}{\sqrt{4\pi}} \Li_{5/2}(\gamma) - \tfrac{4}{3} (\log(1/\gamma))^{3/2}.
\end{align}
\end{subequations}
The value of $ \gamma $ should be determined by $ \alpha $ through the equation
\begin{align}
\label{e.1-1.gamma}
	\sqrt{T/2} \, e^{\alpha} 
	= 
	\begin{cases}
		\ratenosoliton'(\gamma), & \text{when } \alpha \in \nosoliton,
	\\
		\ratesoliton'(\gamma), & \text{when } \alpha \in \soliton.
	\end{cases}
\end{align}
(We will verify in Lemma~\ref{l.1-1.convex} that, for each $ \alpha\in\R $, Equation~\eqref{e.1-1.gamma} has a unique solution.)
Further, the minimum in the variational problem~\eqref{e.minimization} should be given by
\begin{align}
	\label{e.1-1.rate.result}
	\phi_\star(\alpha)
	&:=
	\begin{cases}
		\displaystyle
		\max_{\sigma\in(-\infty,1]} \big\{ \sigma e^{\alpha} - \tfrac{1}{\sqrt{T/2}} \ratenosoliton(\sigma) \big \},
		&
		\text{ when } \sqrt{T/2}\,e^{\alpha} \in \nosoliton,
	\\
		\displaystyle
		\min_{\sigma\in(0,1)} \big\{ \sigma e^{\alpha} - \tfrac{1}{\sqrt{T/2}} \ratesoliton(\sigma) \big\},
		&
		\text{ when } \sqrt{T/2}\,e^{\alpha} \in \soliton.
	\end{cases}
\end{align}

\begin{rmk}
By solitonic solutions we mean those whose $ \Sa(\lambda)|_{\im(\lambda)>0} $ and $ \Sa(\lambda)|_{\im(\lambda)<0} $ have zeros.
The large-scale behaviors of a solitonic solution typically approximate solitons, hence the name solitonic solution.
For the particular solitonic solution considered here, the estimates in Corollary~\ref{c.1-1.large} (in the next subsection) show that $ (\p\q) $ approximates the stationary soliton $ \w_\mathrm{s}(t,x) := \sgamma\sech^2(\sqrt{\sgamma}x) $.
\end{rmk}

Turning to the mathematically rigorous results, we first make an assumption.
Let $ c_{\star,1} $ (which is larger than $ c_{\star} $) be as in \eqref{e.thresdhol.assump}.
When $ \sqrt{T/2}\,e^{\alpha} \geq c_{\star,1} $, our analysis does not rule out all `non-physical' candidates for the scattering coefficients, so we impose an assumption.
We believe this assumption is purely technical, and it is desirable to remove it in the results below.
\begin{assu}\label{assu.poles}
When $ \sqrt{T/2}\,e^{\alpha} \geq c_{\star,1} $, we assume all zeros of $ \Sa(\lambda) $ on the upper half plane and all zeros of $ \Sat(\lambda) $ on the lower half plane lie along the imaginary axis.
\end{assu}

The first result for the 1-to-1 initial-terminal condition is as follows.
\begin{thm}\label{c.1-1}
Consider the 1-to-1 initial-terminal condition~\eqref{e.1-1}.
\begin{enumerate}[leftmargin=20pt,label=(\alph*)]
\item \label{c.1-1.gamma}
For any minimizer of the variational problem~\eqref{e.minimization}, the parameter $ \gamma_1:=\gamma $ belongs to $ (-\infty,1] $.
\item \label{c.1-1.<c1}
For each $ \sqrt{T/2}\,e^{\alpha} < c_{\star,1} $, the variational problem~\eqref{e.minimization} has a unique minimizer $ \w $. 
This minimizer is given by Theorem~\ref{t.formula} with the $ \gamma $ given in \eqref{e.1-1.gamma}, and the explicit scattering coefficients given in \eqref{e.1-1.Sb} and \eqref{e.1-1.Sa.nonsoliton} when $ \sqrt{T/2}\,e^{\alpha} \in \nosoliton $, and in \eqref{e.1-1.Sb} and \eqref{e.1-1.Sa.soliton} when $ \sqrt{T/2}\,e^{\alpha}\in\soliton $.
Further,
\begin{align*}
	\tfrac12 (\norm{\w}_{2;[0,T]\times\R})^2
	=
	\phi_\star(\alpha)
	:=
	\eqref{e.1-1.rate.result}.
\end{align*}
\item \label{c.1-1.>c1}
Under Assumption~\ref{assu.poles}, the same conclusion in Part~\ref{c.1-1.<c1} holds for $ \sqrt{T/2}\,e^{\alpha} \geq c_{\star,1} $.
\end{enumerate}
\end{thm}
\noindent{}This result together with the LDP from \cite[Proposition~1.7(b)]{lin21} immediately yields the following.
\begin{customcor}{\ref*{c.1-1}'}\label{c.1-1.}
Let $ Z_\e $ be the solution of the SHE \eqref{e.she} with $ Z_\e(0,\Cdot) = \delta_0 $. 
Consider the event $ \mathcal{E}_\e(\alpha) := \{ Z_\e(T,0)  \geq e^{\alpha}\} $ when $ e^{\alpha} \geq 1/\sqrt{2\pi T} $ and $ \mathcal{E}(\alpha) := \{ Z_\e(T,0) \leq e^{\alpha} \} $ when $ e^{\alpha} < 1/\sqrt{2\pi T} $.
Let $ \q $ be as in Theorem~\ref{c.1-1}.
Under Assumption~\ref{assu.poles}, for every $ u>0 $,
\begin{align*}
	\lim_{\e\to 0 } \e^{-1} \log \P\big[ \mathcal{E}_\e(\alpha) \big] = - \phi_\star(\alpha),
	&&
	\lim_{\e\to 0 } \P\Big[ \sup_{[u,T]\times[-1/u,1/u]} \big|Z_\e(t,x)-\q(t,x)\big| > u \, \Big| \mathcal{E}_\e(\alpha) \Big] = 0.
\end{align*}
\end{customcor}

\begin{rmk}\label{r.soliton}
A highlight of the results in \cite{krajenbrink21} is the prediction of the non-solitonic-to-solitonic transition.
As mentioned after \eqref{e.soliton.range}, the transition occurs when $ \sqrt{T/2}\,e^{\alpha} $ passes the threshold $ c_{\star} $.
Theorem~\ref{c.1-1} covers this transition because $ c_{\star,1} = 9.4296\ldots > 0.7369\ldots = c_{\star} $.
\end{rmk}

\begin{rmk}\label{r.1-1.ratefn}
The expression~\eqref{e.1-1.rate.result} was derived in the physics works~\cite{ledoussal16short,krajenbrink21} (by different methods); see Equation~(S27) in the supplementary material of the latter work.
The expression can be viewed as a Legendre-like transform of $ \rateone $, but with a \emph{twist}.
Recall that the Legendre transform of $ f(\alpha) $ is $ \sup_{\sigma} \{\sigma\alpha - f(\sigma)\} $.
The first expression in~\eqref{e.1-1.rate.result} is Legendre-like in the sense that $ \sigma\alpha $ is replaced by $ \sigma e^{\alpha} $.
At a physics level of rigor, such a Legendre-like transform can be explained by the derivation in \cite{ledoussal16short}.
The work extracted an $ \e\to 0 $ limit of the log moment generating function of $ Z_\e(2,0) $ from a Fredholm determinant and related the limit to a Legendre-like transform of $ \phi_\star $:
\begin{align*}
	\lim_{\e\to 0} \e \log \E\big[ \exp\big( \e^{-1}\gamma Z_\e(2,0) \big) \big]
	=
	\tfrac{1}{\sqrt{4\pi}} \Li_{5/2}(\gamma) 
	=
	\sup_{\alpha\in\R}
	\big\{ \gamma e^{\alpha} - \phi_\star(\alpha) \big\},
	\qquad
	\gamma \in (-\infty,1].
\end{align*}
The Legendre-like transform is not invertible.
The work formally inverted the transform to get the non-solitonic branch in \eqref{e.1-1.rate.result} and analytically continued the result in $ \alpha $ to get the solitonic branch.
Intriguingly, the continued result flips the maximum into a minimum.
(In \cite[Equation~(20)]{ledoussal16short}, the second min should be a max.)
Such a flip is also observed in another initial condition in the physics work~\cite{krajenbrink17short}; see Equation~(150) in the supplementary material.
It is an interesting question to understand the physical mechanism, if any, behind such a flip.
\end{rmk}

Let us compare the approach of the physics works \cite{krajenbrink20det,krajenbrink21} and our proof of Theorem~\ref{t.formula}.
The works \cite{krajenbrink20det,krajenbrink21} started with the defocusing regime: $ \gamma \leq 0 $.
This regime is more tractable because $ \Sa(\lambda)|_{\im(\lambda)>0} $ and $ \Sa(\lambda)|_{\im(\lambda)<0} $ have no zeros.
In the defocusing regime, they solved the Riemann--Hilbert problem via a formal series-expansion procedure and obtained the solution formulas for $ \p $, $ \q $, and $ \w $.
For specific initial-terminal conditions, such as the 1-to-1 condition, they then analytically continued the solution formulas in $ \gamma $ into the focusing regime.
Our proof of Theorem~\ref{t.formula} also proceeds through the Riemann--Hilbert problem, but in a very different way.
We recognize that the jump conditions in the Riemann--Hilbert problem, upon Fourier transform, can be written as a linear equation in terms of operators.
Solving this linear equation gives the solution formulas of \cite{krajenbrink21}.
This Fourier-transform approach can handle both the defocusing and focusing regimes directly by choosing the jump contours appropriately.

Let us explain the key ingredients in proving Theorem~\ref{c.1-1} and the reason for imposing Assumption~\ref{assu.poles}.
Given Theorem~\ref{t.formula}, proving Theorem~\ref{c.1-1} amounts to finding the scattering coefficients under the 1-to-1 initial-terminal condition.
To find $ \Sa(\lambda) $ and $ \Sat(\lambda) $, we utilize the fact that they satisfy a scalar Riemann--Hilbert problem.
Solving the scalar problem yields all possible candidates for $ (\Sa,\Sat) $.
The scalar problem unfortunately has \emph{infinitely} many solutions, which yield infinitely many candidates for $ (\Sa,\Sat) $.
Among such candidates, only one is believed to be relevant, which we call `physical' and will be explained in the next paragraph.
We ruled out some non-physical candidates (for example by checking whether the candidate produces real conserved quantities) but were not able to rule out all candidates when $ \sqrt{T/2}\,e^{\alpha} \geq c_{\star,1} $.
Assumption~\ref{assu.poles} is in place to exclude the remaining non-physical candidates.

Let us recall how the physics work \cite{krajenbrink21} identified, at a physics level of rigor, a unique $ (\Sa,\Sat) $.
As mentioned previously, the work started with the defocusing regime $ \gamma \leq 0 $.
In the defocusing regime, $ \Sa(\lambda)|_{\im(\lambda)>0} $ and $ \Sat(\lambda)|_{\im(\lambda)<0} $ have no zeros, and this property allows to uniquely identify the $ (\Sa,\Sat) $ from the scalar problem mentioned in the previous paragraph.
Once the $ (\Sa,\Sat) $ is obtained for the defocusing regime $ \gamma\leq 0 $, the work \cite{krajenbrink21} analytically continued them in $ \gamma $ into the focusing regime.
This continuation procedure is convincing, though not mathematically rigorous: The analyticity of $ \Sa(\lambda) $ and $ \Sat(\lambda) $ in $ \gamma $ are not known a priori.
The analyticity is quite subtle since the mappings $ \gamma \mapsto \Sa(\lambda) $ and $ \gamma\mapsto\Sat(\lambda) $ are multi-valued when $ \gamma \in (0,1) $.
This multi-valued behavior was already seen in \eqref{e.1-1.rate}, where $ \psi_{\star} $ has two branches $ \ratenosoliton $ and $ \ratesoliton $ when $ \gamma\in(0,1) $.
The continuation procedure gives the physical candidate mentioned in the previous paragraph, and we hence believe Assumption~\ref{assu.poles} is purely technical.

Next, let us analyze the $ \p $, $ \q $, and $ \w $ in Theorem~\ref{c.1-1} under a scaling that corresponds to the so-called upper-tail limit.
Letting $ N\to\infty $ be the scaling parameter, we scale time and $ \alpha $ linearly in $ N $, more precisely $ T=2N $ and $ \alpha = N\salpha $ for some fixed $ \salpha\in(0,\infty) $.
This scaling is equivalent to the more-commonly-considered scaling $ T=2 $ and $ \alpha = (N \salpha)^{3/2} $ through a change of variables in $ (t,x) $.
The result gives detailed pointwise estimates of $ \p $, $\q $, and $ \w $ everywhere except near $ t=0 $ and $ t=2N $.
\begin{cor}\label{c.1-1.large}
Notation as in the preceding and Assumption~\ref{assu.poles} still in action.
Fix an $ \salpha\in(0,\infty) $ and let $ \sgamma = \salpha + \frac{1}{N} \log(2 \sqrt{\salpha}) + \ldots $ be given by the solution of \eqref{e.1-1.gamma.scaling}.
Set $ \tau := \min\{t,2N-t\} $.
For all $ \tau $ large enough (depending only on $ \salpha $),
\begin{subequations}
\label{e.c.1-1.large}
\begin{align}
	\label{e.c.1-1.large.1}
	\q(t,x)
	=\,&
	e^{\sgamma t/2} \sqrt{\sgamma}  \sech(\sqrt{\sgamma}x) \,
	\big( 1 + O(\sqrt{1+\tau} \, e^{-\sgamma\tau/2}) \big) 
	\ind_{\{|x|<\sqrt{\gamma} t\}} 
\\
	&
	\label{e.c.1-1.large.2}
	+
	\hk(t,x) \Big( 1 + O\big(1\big/\max\{| \tfrac{|x|}{t}-\sqrt{\sgamma}| , \tfrac{1}{\sqrt{t}} \}\big)\Big)\big( 1 + O(\sqrt{1+\tau} \, e^{-\sgamma\tau/2}) \big),
\\
	\label{e.c.1-1.large.3}	
	\p(t,x)
	=\,& 
	e^{-\sgamma N}\, \big( \text{ \eqref{e.c.1-1.large.1}--\eqref{e.c.1-1.large.2}  with } t \text{ replaced by } (2N-t)\, \big),
\\
	\label{e.c.1-1.large.4}	
	\w(t,x)
	=\,&
	\sgamma \sech^2(\sqrt{\sgamma}x) \, \big( 1 + O(\sqrt{1+\tau} \, e^{-\sgamma\tau/2}) \big)
\\
	\label{e.c.1-1.large.5}
	&\qquad
	\cdot
	\Big( 
		\ind_{\{|x|<\sqrt{\sgamma} \tau\}} 
		+
		O\big( e^{-\frac{1}{2t}(|x|-\sqrt{\sgamma}t)^2} + e^{-\frac{1}{2(2N-t)}(|x|-\sqrt{\sgamma}(2N-t))^2} \big)
	\Big),
\end{align}
where $ \hk(t,x) := \exp(-x^2/(2t))/\sqrt{2\pi t} $ denotes the heat kernel, and $ O(a) $ denotes a generic quantity that is bounded by a constant multiple of $ |a| $ for a fixed $ \salpha $.
\end{subequations}
\end{cor}
\noindent{}We note that the $ N\to\infty $ limit of $ \frac{1}{N} \log \q(Nt_*,Nx_*) $ was obtained in \cite[Theorem 1.2]{gaudreaulamarre21} without any assumptions.
Here, the estimate~\eqref{e.c.1-1.large.1}--\eqref{e.c.1-1.large.2} of $ \q $ is much more refined but requires Assumption~\ref{assu.poles}.

\subsection{Solitons in the NLS equations as non-entropy shocks in Burgers equation}
\label{s.results.burgers}

Corollary~\ref{c.1-1.large} confirms the physics prediction, from the weak noise theory, that a soliton controls $ \q $ and $ \p $ in the upper-tail limit~\cite{kolokolov07,meerson16,krajenbrink21} for the 1-to-1 initial-terminal condition. 
As is readily verified, the NLS equations have a solution given by $ \q_\mathrm{s}(t,x) := e^{-\sgamma t/2}\sqrt{\sgamma}\sech(\sqrt{\sgamma}x) $ and $ \p_\mathrm{s}(t,x) := e^{\sgamma t/2}\sqrt{\sgamma}\sech(\sqrt{\sgamma}x) $.
We refer to 
\begin{align*}
	\w_\mathrm{s}(t,x) := (\p_\mathrm{s}\q_\mathrm{s})(t,x) = \sgamma\sech^2(\sqrt{\sgamma}x)
\end{align*}
as a \textbf{soliton}.
Under Assumption~\ref{assu.poles}, Corollary~\ref{c.1-1.large} shows that the minimizer $ \w $ approximates this soliton except near $ t=0 $ and $ t=2N $.

Next, let us examine the effect of the soliton in the context of the vanishing viscosity limit.
Applying the inverse Hopf--Cole transform $ h := \log \q $ to the first of the NLS equations~\eqref{e.NS.q} gives
\begin{align}
	\label{e.HJ.burgers}
	\partial_t h &= \tfrac12 \partial_{xx} h + \tfrac12 (\partial_xh)^2 + \w.
\end{align}
This is the Hamilton--Jacobi (HJ) equation of the viscous Burgers equation driven by $ \w $.
Perform the scaling $ h_N(s,y) := \frac{1}{N} h(Ns,Ny) $ and $ \w_{N}(s,y) := \w(Ns,Ny) $ to get
\begin{align}
	\label{e.HJ.burgers.scaled}
	\partial_t h_N &= \tfrac{1}{2N} \partial_{xx} h_N + \tfrac12 (\partial_xh_N)^2 + \w_N.
\end{align}
The $ N\to\infty $ limit of \eqref{e.HJ.burgers.scaled} is often referred to as the \emph{vanishing viscosity limit}.
The $ \w $ here approximates the soliton $ \w_{\mathrm{s}} $, so $ w_N \to 0 $ almost everywhere. 
Hence, the limit of \eqref{e.HJ.burgers.scaled} formally reads
\begin{align*}
	\partial_t h_* & = \tfrac12 (\partial_x h_*)^2.
\end{align*}
Using Corollary~\ref{c.1-1.large} under Assumption~\ref{assu.poles} or using \cite[Theorem 1.2]{gaudreaulamarre21}, we have 
\begin{align*} 
	h_*(s,y) = (\sgamma\tfrac{s}{2}-\sqrt{\sgamma}y)\ind_{\{|y|\leq \sqrt{\sgamma} s\}} -\tfrac{y^2}{2s}\ind_{\{|y|>\sqrt{\sgamma}s\}}.
\end{align*}
This expression confirms that $ h_* $ solves the inviscid HJ equation, but is a \emph{non-entropy} solution.
It has a non-entropy shock at $ y=0 $, as shown in Figure~\ref{f.antishock}.

The non-entropy shock is produced by $ \w_N(s,y) $.
By Corollary~\ref{c.1-1.large}, the function approximates $ \w_{\mathrm{s}}(Ny) $, which has the shape of a bump of width $ 1/N $ and constant height.
Even though the width diminishes, the bump leaves a lasting impact on $ h_* $ by producing a non-entropy shock at $ y=0 $.
In the context of asymmetric exclusion processes, having the bump is analogous to slowing down the hopping rate near $ x=0 $. (This analogy can be seen by writing down the evolution equation of the height function of the exclusion process.) Such a slow down produces an artificial jump of the particle density around $ x=0 $, namely a non-entropy shock.
This bump/show-down picture appeared in the physics study of asymmetric interacting particle systems with open boundaries \cite{derrida01,derrida03,bodineau06,bahadoran10}. 
The picture also provides a transparent (though heuristic) explanation of the mechanism behind the Jensen--Varadhan LDPs \cite{jensen00,varadhan04}: the long-time LDPs for asymmetric interacting particle systems, which is largely-open and has only been proven for one instance~\cite{quastel21}.

It is interesting to further explore the connection between solitons and non-entropy shocks.
Indeed, the NLS equations have many more solitons, which exhibit richer behaviors.
Examining how these solitons produce non-entropy shocks can lead to a better understanding of the LDP for the SHE and the KPZ equation.

\begin{figure}
\includegraphics[width=.5\linewidth]{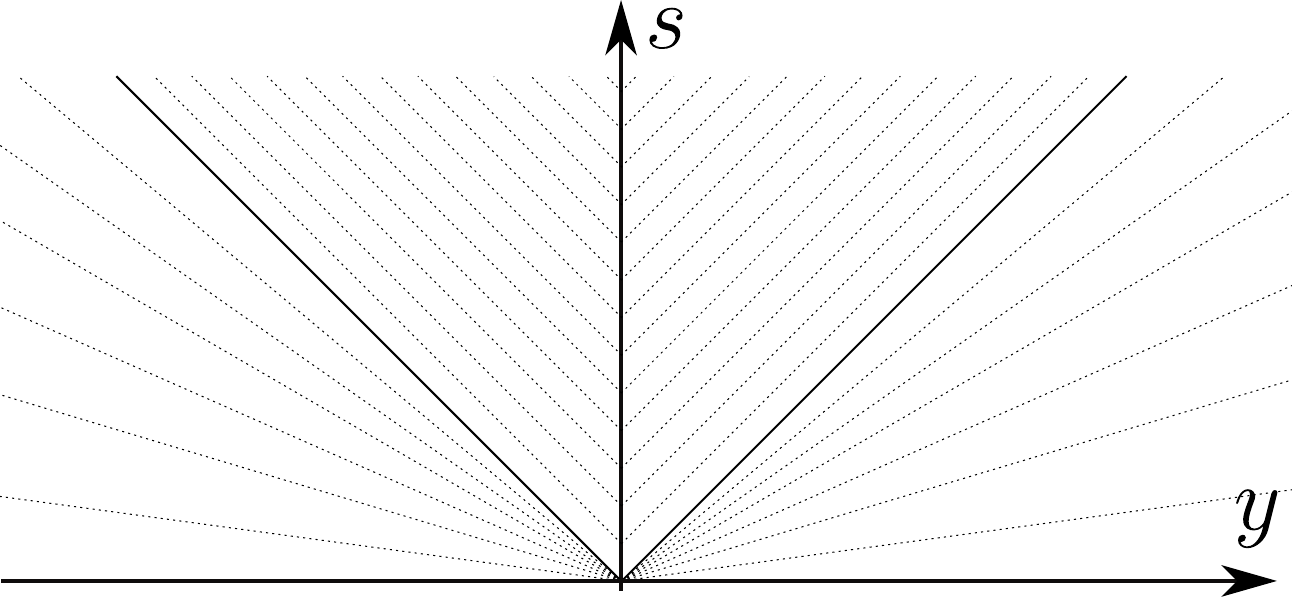}
\caption{The characteristics of $ h_* $. Recall that characteristics are curves along  which $ \partial_x h_* $ is constant. The solid characteristics have velocities $ \sqrt{\sgamma} $ and $ -\sqrt{\sgamma} $.}
\label{f.antishock}
\end{figure}

\subsection{Questions about uniqueness}
\label{s.results.uniqueness}

First, let us consider the unique solvability of the NLS equations~\eqref{e.NS.q}--\eqref{e.NS.p}, assuming that $ \qic $ and $ \ptc $ are given.
Note that our setting differs from the standard one.
The standard setting concerns the real-time NLS equations with given initial conditions for $ \q $ and $ \p $.
Our setting concerns the imaginary-time NLS equations with a given initial condition for $ \q $ and a given terminal condition for $ \p $.
In our setting, the NLS equations do \emph{not} have unique solvability in general.
The non uniqueness is seen in Theorem~\ref{c.1-1}\ref{c.1-1.<c1} (which does not require Assumption~\ref{assu.poles}).
For $ \qic = \delta_0 $ and $ \ptc = \gamma \delta_0 $ with $ \gamma \in (0, (\ratesoliton')^{-1}(\sqrt{T/2} \, c_{\star,1})) $, Theorem~\ref{c.1-1}\ref{c.1-1.<c1} gives two sets of solutions of the NLS equations: One corresponds to $ \sqrt{T/2}\,e^{\alpha}\in\nosoliton $ and the other corresponds to $ \sqrt{T/2}\,e^{\alpha}\in\soliton $.


Next, let us consider the uniqueness of the minimizers of the variational problem~\eqref{e.minimization}, assuming that $ \qic $ and $ (\xi_{\ii},e^{\alpha_\ii})_{\ii=1}^{\mm} $ are given.
For the 1-to-1 initial-terminal condition, Theorem~\ref{c.1-1} shows that the minimizer is unique.
For the Brownian initial condition, the physics work~\cite{janas16} predicted  an intriguing symmetry breaking (which is further analyzed in \cite{krajenbrink21flat}) that implies the \emph{non uniqueness} of the minimizers.
Following the spirit of this prediction, we formulate a conjecture.
\begin{conj}[symmetry breaking]\label{conj.symmetry.breaking}
Let $ \qic = \delta_{-\zeta} + \delta_{\zeta} $, for $ \zeta>0 $, and set $ \mm=1 $ and $ \xi_1 = 0 $.
There exists an $ \alpha_{\mathrm{c}} \in \R $, which depends on $ T $ and $ \zeta $, such that for all $ \alpha_1 > \alpha_{\mathrm{c}} $ , the variational problem~\eqref{e.minimization} has exactly \emph{two} minimizers, which are the reflection (abound $ x=0 $) of each other.
\end{conj}

\subsection{A discussion about the terminal condition $ \ptc $}
\label{s.results.tc}

Here we explain the challenge in proving that $ \ptc $ is a sum of delta functions as in \eqref{e.NS.p} and give the idea of our proof.

The terminal condition \eqref{e.NS.p} does \emph{not} just follow from the calculus of variation.
Indeed, the standard calculus of variation can show that $ \ptc $ is supported in $ \{\xi_1,\ldots,\xi_{\mm}\} $.
However, knowing this fact alone \emph{does not} guarantee the terminal condition stated in \eqref{e.NS.p}.
For example, we can add derivatives of delta functions without changing the support of $ \ptc $:
\begin{align*}
	\gamma_1 \delta_{\xi_1} + \ldots + \gamma_{\mm} \delta_{\xi_{\mm}}
	+
	\gamma_{1,1} \delta'_{\xi_1} + \ldots + \gamma_{\mm,1} \delta'_{\xi_{\mm}}
	+
	\ldots.
\end{align*}

\begin{rmk}
The Martin--Siggia--Rose (MSR) formalism from physics offers a different way to perform the calculus of variation, which, at a heuristic level, produces the sum-of-delta $ \ptc $.
This is carried out in Section~B in the supplementary material of \cite{krajenbrink21}.
In the MSR formalism, one replaces the `hard' conditioning $ \qfn[\w](T,\xi_i)=e^{\alpha_i} $ in the variational problem~\eqref{e.minimization} with certain `soft' weights.
The calculus of variation can show that any minimizer of the `soft' problem has a sum-of-delta terminal condition.
Non-rigorously applying the inverse of a Legendre-like transform maps the minimizers of the soft problem to the hard problem.
The last procedure is not rigorous because the Legendre-like transform is not bijective.
It is an interesting question to explore whether one can use the MSR formalism to rigorously prove Theorem~\ref{t.NS}. 
\end{rmk}

To gain an idea of how to proceed, let us examine the $ \mm=1 $ case and ask ourselves how to rule out the derivatives of $ \delta_{\xi_1} $.
When $ \mm=1 $, the value of $ \q[\w](T,\Cdot) $ is conditioned only at one point $ \xi_1 $.
In this case, any $ \w $ that minimizes the variational problem has a definite sign: either nonnegative everywhere or nonpositive everywhere.
The cases happen when $ e^{\alpha_1} > \qfn[0](T,\xi_1) $ and when $ e^{\alpha_1} < \qfn[0](T,\xi_1) $ respectively.
This property is intuitive to understand (and is in fact not difficult to prove): In order to shoot higher we need $ w \geq 0 $, and in order to shoot lower we need $ \w \leq 0 $.
Recall that $ \p = \w/\q $ and that $ \q >0 $.
Hence $ \w $ having a definite sign means the same for $ \p $.
The fact that $ \p $ has a definite sign rules out the possibility of having derivatives of $ \delta_{\xi_1} $ in $ \ptc $.
To see why, note that when $ T-t $ is small, we expect the solution of \eqref{e.NS.p} to approximate the solution of the backward heat equation.
Solving the backward heat equation with the terminal condition $ \gamma_1 \delta_{\xi_1} + \gamma_{1,1} \delta'_{\xi_1} + \ldots $ shows that the solution has a definite sign only in the absence of the derivatives of the delta functions.

The discussion of the $ \mm=1 $ case suggests a way to proceed, but there is still an issue.
When $ \mm >1 $, the function $ \w $ does \emph{not} have a definite sign in general. 
Indeed, when $ e^{\alpha_1}-\qfn[0](T,\xi_1) $, \ldots, $ e^{\alpha_{\mm}}-\qfn[0](T,\xi_{\mm}) $ do not all have the same sign, we do not expect $ \w $ to be everywhere nonnegative or nonpositive.

To resolve this issue, we look for a \emph{local} (and approximate) definite-sign property.
Indeed, as soon as we know that $ \p $ has a definite sign in a \emph{neighborhood} around each $ (T,\xi_{\ii}) $, $ \ii=1,\ldots,\mm $, the same argument in the second last paragraph rules out the derivatives of the delta functions.
In Proposition~\ref{p.definite}, we will state and prove a local (and approximate) definite-sign property.
This proposition will allow us to rigorously prove that $ \ptc $ is a sum of delta functions (in the sense of Definition~\ref{d.duhamelsense}).

\subsection{Notation and definitions}
\label{s.results.basic}

Throughout this paper, we write $ \hk(t,x) := (2\pi t)^{-1/2} \exp(-x^2/(2t)) \ind_{\{t>0\}} $ for the heat kernel, with the convention that $ \hk(t,x)|_{t \leq 0} := 0 $. 
We write $ (\hk(t)*f)(x) := \int_{\R} \d y \, \hk(t,x-y) f(y) $ for the spatial convolution of the heat kernel with $ f $, and write $ \norm{ f }_{p;\Omega} := (\int_\Omega \d t \d x |f|^p)^{1/p} $ for the $ \Lsp^p $ norm over a domain $ \Omega \subset [0,T]\times \R $ or $ \Omega \subset \R $.
We use $ c = c(a,b,\ldots) $ to denote a generic, positive, finite constant. 
The constant may change from place to place but depends only on the designated variables $ a,b,\ldots $.
The initial condition $ \qic $ and the terminal condition $ (\xi_{\ii},e^{\alpha_{\ii}})_{\ii=1}^{\mm} $ are fixed, so their dependence will not be designated in the constants, except in Section~\ref{s.1-1.large} where we perform scaling in the terminal condition.

Often in Section~\ref{s.NS}, we omit the domain when it is $ [0,T]\times\R $.
For example, $ \norm{f}_{p} := \norm{f}_{p;[0,T]\times\R} $, $ \Lsp^p := \Lsp^p([0,T]\times\R) $, and
\begin{align*}
	\int \d s \d y \, \hk(t-s,y-x) f(s,y) 
	&:= 
	\int_{[0,T]\times \R} \d s \d y \, \frac{1}{\sqrt{2\pi(t-s)}} e^{-\frac{(y-x)^2}{2(t-s)}} \ind_{\{t-s>0\}} f(s,y).
\end{align*}

Next, we state what it means to solve the NLS equations in the Duhamel sense.
\begin{defn}
\label{d.duhamelsense}
Given a $ w \in \Lsp^2([0,T]\times\R) $, we say $ \q $ solves $ \partial_t \q = \frac12 \partial_{xx} \q + \w\q $ with the initial condition $ \qic $~\eqref{e.qic} in the Duhamel sense if $ \q \in \Csp((0,T]\times\R) $ and, for all $ (t,x) \in (0,T]\times\R $, 
\begin{align}
	\tag{Int Eq q}
	\label{e.inteq.q}
	\q(t,x)
	=
	\big(\hk(t)*\qic\big)(x)
	+
	\int \d s \d y \, \hk(t-s,x-y) \big(\w\q\big)(s,y)
\end{align}
holds, with the last integral converging absolutely and being an $ \Lsp^2([0,T]\times\R) $ function in $ (t,x) $.
As will be explained in the following, the solution exists and is unique; we let $ \qfn[\w] $ denote this solution.

Similarly, we say $ \p $ solves $ -\partial_t \p = \frac12 \partial_{xx} \p + \p\w $ with the terminal condition $ \ptc := \sum_{\ii=1}^{\mm} \gamma_\ii \delta_{\xi_\ii} $ in the Duhamel sense if $ \p \in \Csp([T,0)\times\R) $ and, for all $ (t,x) \in [0,T)\times\R $,
\begin{align}
	\tag{Int Eq p}
	\label{e.inteq.p}
	\p(t,x)
	=
	\sum_{\ii=1}^{\mm} \gamma_\ii \hk(T-t,\xi_\ii-x)
	+
	\int \d s \d y \, \big(\p\w\big)(s,y) \hk(s-t,y-x)
\end{align}
holds, with the last integral converging absolutely and being an $ \Lsp^2([0,T]\times\R) $ function in $ (t,x) $.
As will be explained in the following, for a given $ \w\in\Lsp^2 $, the solution exists and is unique. 
\end{defn}

We will make a few remarks on Definition~\ref{d.duhamelsense}, but first we need some notation and an inequality.
For $ \w\in\Lsp^2 := \Lsp^2([0,T]\times\R) $, $ s<t \in [0,T] $, $ y,x\in\R $, and $ n\in \Z_{\geq 0} $, consider the $ n $-fold convolution of $ \hk $ with respect to the measure $ \w(\sigma,y) \ind_{[s,t]}(\sigma) \d \sigma\d y $
\begin{align}
\label{e.hkiter}
\begin{split}
	&\hkiter{\w}{n}(s,t,y,x)
\\
	&:=
	\int_{[s,t]^{n-1}\times\R^{n-1}} \prod_{i=1}^{n-1} \w(s_i,y_i) \d s_i \d y_i \, \hk(s_{i}-s_{i-1},y_{i}-y_{i-1}) \,  \cdot \hk(t-s_{n-1},x-y_{n-1}),
\end{split}
\end{align}
with the convention $ s_0:=s $, $ y_0:=y $, and $ \hkiter{\w}{1}(s,t,y,x) := \hk(t-s,x-y) $.
Note that $ \hk(s,y)|_{s\leq 0} := 0 $.
As is proven in Lemma~\ref{l.bd.iter},
\begin{align}
	\label{e.bd.hkiter}
	\big|\hkiter{\w}{n}(s,t,y,x)\big|
	\leq
	\hkiter{|\w|}{n}(s,t,y,x)
	\leq
	\frac{\pi^{1/4} \big( 2^{-1/2} \, (t-s)^{1/4} \norm{w}_{2; [s,t]\times\R} \big)^{n-1} }{((n-1)!\,\Gamma(\frac{n}{2}) )^{1/2}} 
	\,
	\hk(t-s,x-y),
\end{align}
where $ \Gamma $ denotes Euler's gamma function.
Granted this bound, we form the sum
\begin{align}
	\label{e.Kern}
	\Kern(s,t,y,x)
	:=
	\sum_{n=1}^\infty \hkiter{\w}{n}(0,t,y,x).
\end{align}

A few remarks on Definition~\ref{d.duhamelsense} are in place.
First, the Duhamel-sense solution of $ \partial_t \q = \frac12 \partial_{xx} \q + \w\q $ exits, is unique, and is given by
\begin{align}
\label{e.iter.q}
	\q(t,x) = \qfn[\w](t,x)
	&= 
	\sum_{n=1}^\infty \int_{\R} \d y \, \hkiter{\w}{n}(0,t,y,x) \qic(y)
	= 
	\int_{\R} \d y \, \Kern(0,t,y,x) \qic(y).
\end{align}
This fact follows by iterating \eqref{e.inteq.q} and using the bound~\eqref{e.bd.hkiter} and our assumptions on $ \int \d s \d y \, \hk(t-s,x-y)(\w\q)(s,y) $ to control the remainder term.
Further, using the bound~\eqref{e.bd.hkiter} in \eqref{e.iter.q} gives
\begin{align}
	\label{e.bd.q}
	|\qfn[\w](t,x)| \leq c \, \big( \hk(t) * \qic\big)(t,x),
	\qquad
	c = c(T,\norm{\w}_2).
\end{align}
The series in \eqref{e.iter.q} can be recognized as a Feynman--Kac formula
\begin{align}
	\label{e.FKq}
	\qfn[w](t,x) = \E_{x} \Big[ e^{\int_0^t \d s \, \w(s,B(t-s))} \qic(B(t)) \Big]
	=
	\int_{\R} \d y \, \E \Big[ e^{\int_0^t \d s \, \w(s,W(t-s)) }\, \Big] \hk(t,x-y)\qic(y),
\end{align}
where $ B $ denotes the Brownian motion starting from $ x $, and $ W $ denotes the Brownian bridge with $ W(0)=x $ and $ W(t)=y $.
Comparing \eqref{e.iter.q} and \eqref{e.FKq} gives
\begin{align}
	\label{e.Kern.FK}
	\E \Big[ e^{\int_0^t \d s \, \w(s,W(t-s)) }\, \Big] \hk(t,x-y)
	=
	\Kern(s,t,y,x).
\end{align}
Next, given the positivity of $ \qic $~\eqref{e.qic>0}, we see from the Feynman--Kac formula~\eqref{e.FKq} that $ \qfn[\w] $ is positive everywhere on $ (0,T]\times\R $.
In fact, a more careful analysis done in Lemma~\ref{l.q.bd} gives
\begin{align}
	\label{e.lwbd.q}
	\qfn[\w](t,x) \geq \tfrac{1}{c} \, \big( \hk(t) * \qic\big)(t,x),
	\qquad
	c = c(T,\norm{\w}_2).
\end{align}
The same arguments show that the unique Duhamel-sense solution of $ -\partial_t \p = \frac12 \partial_{xx} \p + \p\w $ is given by
\begin{align}
\label{e.iter.p}
	\p(t,x)
	= 
	\int_{\R} \d y \, \Kern(t,T,x,y) \ptc(y)
	:= 
	\sum_{\ii=1}^{\mm} \Kern(t,T,x,\xi_{\ii}) \gamma_{\ii}.
\end{align}

\section{The nonlinear Schr\"{o}dinger equations: Proof of Theorem~\ref{t.NS}}
\label{s.NS}

This section consists of the proof of Theorem~\ref{t.NS}.
The proof is carried out in steps that are designated by the titles of the subsections.

\subsection{Variation}
\label{s.NS.variation}

Hereafter, fix a minimizer $ w $ of \eqref{e.minimization}, and write $ \q := \qfn[w] $ and $ \p := \w/\q $.
The existence of a minimizer is proven in Lemma~\ref{l.minimization.exist}.

The first step is to perform variation in $ \q $ to show that $ \p $ solves $ - \partial_s \p = \frac12 \partial_{yy} \p + \p\w  $ in the weak sense.
We seek to vary $ \q \mapsto \til{\q} := \q + \e f $, for any test function $ f = f(s,y) \in \Csp^\infty([0,T]\times\R) $ such that $ \supp(f) $ is compact and does not overlap with $ \{0\}\times\R $ and $ \{(T,\xi_1),\ldots,(T,\xi_{\mm})\} $.
Note that we allow $ f(T,\Cdot) $ to be nonzero off a neighborhood of $ \{\xi_1,\ldots,\xi_{\mm}\} $.
We wish to realize the perturbed function $ \q + \e f $ as $ \qfn[\til{w}] $ for some $ \til{w} $.
To this end, set $ \til{w} := \frac{1}{\q + \e f} (\w\q + \e \partial_s f - \e \frac12 \partial_{yy} f) $.
Given the positivity of $ \q $ from \eqref{e.lwbd.q}, for all $ \e $ small enough, we have that $ q+\e f>0 $ everywhere on $ (0,T]\times\R $ and that $ \til{w}\in\Lsp^2([0,T]\times\R) := \Lsp^2 $.
From the definition of $ \til{w} $ and from Definition~\ref{d.duhamelsense}, it is not hard to verify that $ \q + \e f = \qfn[\til{w}] $.
Further, given our assumptions on $ f $, the quantity
\begin{align*}
	\frac12 \norm{\til{w}}_2^2
	=
	\int \d s \d y \, \frac{1}{2(q + \e f)^2} \big(\w\q + \e \partial_s f - \e \tfrac12 \partial_{yy} f\big)^2
\end{align*}
is $ \Csp^\infty $-smooth in $ \e $ around $ \e=0 $, for fixed $ \q $ and $ f $.
Since $ w $ is a minimizer, the last expression has zero derivative at $ \e =0 $.
Take the derivative in $ \e $, set $ \e =0 $, and substitute in $ \w/\q = \p $.
We arrive at
\begin{align}
	\label{e.NS.weak}
	\int \d s \d y \, \big( p \, \partial_s f  - \p \,\tfrac12 \partial_{yy} f - \w\p f  \big) = 0.
\end{align}
This shows that $ \p $ solves $ - \partial_s \p = \frac12 \partial_{yy} \p + \p\w  $ in the weak sense.

\subsection{The continuity of $ \p $}
\label{s.NS.conti.p}
In this subsection, we show $ \p \in \Csp((0,T)\times\R) $.
Following the progress in Section~\ref{s.NS.variation}, the natural next step is to obtain the smoothness of $ \q $ and $ \p $.
We will do so by using the standard regularity estimate of parabolic PDEs (stated in Section~\ref{s.NS.smooth}). 
The estimate, however, requires $ \w\p \in \Lsp^\alpha_\loc $ for some $ \alpha>1 $.
Hereafter, $ \Lsp^\alpha_\loc := \Lsp^\alpha_\loc((0,T)\times\R) $ denotes the space of functions such that $ f|_K \in \Lsp^\alpha(K) $, for all compact $ K \subset (0,T)\times\R $.
From $ \w\in\Lsp^2 $ and $ \p=\w/\q $, we can infer that $ \w\p \in \Lsp^\alpha_\loc $ only for $ \alpha=1 $.
Hence, in order to use the regularity estimates, here we prove that $ \p \in \Csp((0,T)\times\R) $, which implies $ \w\p =\q\p^2 \in \Csp((0,T)\times\R) \subset \Lsp^\infty_\loc $.

Our proof will involve localization onto spatial intervals.
To set up the notation, for an interval $ I=[a,b] $, consider the heat kernel $ \hk_{I}(s,x,y) $ on $ I $ with the Dirichlet boundary condition.
That is, $ (\partial_s - \frac12 \partial_{yy}) \hk_{I} = 0 $, $ \hk_{I}(s,x,y)|_{y=a,b} = 0 $, and $ \hk_{I}(0,x,y) = \delta_x(y) $.
Following the same convention for $ \hk $, we extend $ \hk_I(s,x,y) $ to $ s\leq 0 $ by setting $ \hk_I(s,x,y)|_{s\leq 0} := 0 $.
Recall $ \Kern $ from \eqref{e.Kern} and consider its localized analog:
\begin{align}
	\label{e.KernI}
	\Kern_{I}(t,s,x,y)
	:=
	\sum\nolimits_{n=1}^\infty \hk_{I}^{n(*\w)}(t,s,x,y),
\end{align}
where $ \hk_{I}^{n(*\w)}(t,s,x,y) $ is obtained by replacing $ \hk $ with $ \hk_{I} $ on the right side of \eqref{e.hkiter}.
This $ \Kern_{I}(t,s,x,y) $ is the fundamental solution of $ (\partial_s - \frac12 \partial_{yy} - \w) f = 0 $ on $ I $ with the Dirichlet boundary condition, as can be seen by the same iteration procedure done in \eqref{e.iter.q}.

Let us prepare some notation.
Fix $ \phi:\C^\infty(\R^2,[0,\infty)) $ with $ \int_{\R^2} \d t \d x \, \phi = 1 $ and $ \supp(\phi) = $ (unit ball). 
Set $ \phi_r(t,x) := \phi(t/r,x/r)/r^2 $, and let $ (\phi_r * \p)(t,x) := \int_{\R^2} \d t' \d x' \phi_r(t-t',x-x')\p(t',x') $ and $ (\phi_r * \Kern_{I})(t,s,x,y) := \int_{\R^2} \d t' \d x' \phi_r(t-t',x-x')\Kern_{I}(t',s,x',y) $  denote the two-dimensional convolution.

The first step of the proof is to develop a local integral representation of $ \phi_r*\p $.
\begin{lem}
\label{l.intrep.p.}
Fix any $ D = [t_1,t_2]\times[-a,a] \subset (0,T)\times\R $ and $ 0<r < \frac12 \max\{t_1,T-t_2,1\} $.
For almost every $ (t_3,b)\in((t_2+T)/2,T)\times(a+1,\infty) $, the following holds for all $ (t,x)\in D $ with $ I = [-b,b] $:
\begin{subequations}
\label{e.intrep.p.}
\begin{align}
\label{e.intrep.p.1}
	\big(\phi_r*\p\big)&(t,x)
\\
\label{e.intrep.p.2}
	&=
	\int_{-b}^b \d y \, \big( \phi_r*\Kern_{I} \big) (t,t_3,x,y) \p(t_3,y)
\\	
\label{e.intrep.p.3}
	&
	+
	\sum_{\sigma=\pm} \frac{-\sigma}{2} \int_{t-r}^{t_3} \d s \, \big( \partial_y(\phi_r*\Kern_{I}) \big)(t,s,x,\sigma b) \p(s,\sigma b).
\end{align}
\end{subequations}
\end{lem}

\begin{rmk}
\label{r.intrep.}
\begin{enumerate}[leftmargin=20pt,label=(\alph*)]
\item[]
\item The integrals in \eqref{e.intrep.p.} are along the parabolic boundary (going backward in time) of $ [t_1-r,t_3]\times[-b,b] $:
\begin{align*}
	\partial_\mathrm{P} \big( [t_1-r,t_3]\times[-b,b] \big)
	:=
	([t_1-r,t_3] \times \{-b\}) \cup ([t_1-r,t_3] \times \{b\}) \cup (\{t_3\}\times[-b,b]). 
\end{align*}
We refer to \eqref{e.intrep.p.} as a \emph{local} integral representation since the integrals in \eqref{e.intrep.p.2}--\eqref{e.intrep.p.3} are over bounded sets given by the parabolic boundary.
This is to be compared with \eqref{e.iter.p}, which we refer to as a \emph{global} integral representation.

\item Note the `almost every' quantifier in Lemma~\ref{l.intrep.p.}.
With $ \w \in \Lsp^2([0,T]\times\R) $ and $ \p = \w/\q $, the function $ \p $ is only defined almost everywhere on $ [0,T]\times\R $.
The integrals in \eqref{e.intrep.p.} are along line segments, which have zero Lebesgue measures in $ [0,T]\times\R $.
Hence, before knowing more information about $ \p $, it does not make sense to require \eqref{e.intrep.p.} to hold for every $ (t_3,b) $ in the designated range.
\item 
\label{r.e.intrep.}
For almost every $ (t_3,b) $ in the designated range, the integrals in \eqref{e.intrep.p.} are convergent and are continuous in $ (t,x)\in D $.
To see why, first note that, with $ \w \in \Lsp^2([0,T]\times\R) $ and $ \p = \w/\q $, we have $ \p(t_3,\Cdot) \in \Lsp^2([-b,b]) $ and $ \p(\Cdot,\pm b) \in \Lsp^2([t_1/2,t_3]) $ for almost every $ (t_3,b) $ in the designated range.
Next note that, the distance between $ D $ and $ \partial_\mathrm{P} ( [t_1-r,t_3]\times[-b,b] ) $ is positive.
Using this, it is not hard to check that $ \Kern_{I}(t,s,x,y) $ and $ \partial_y\Kern_{I} (t,s,x,y) $ are uniformly continuous on $ (t,x,s,y) \in D\times \partial_\mathrm{P} ( [t_1-r,t_3]\times[-b,b] ) $.
These properties together give the claim.
\end{enumerate}
\end{rmk}

\begin{proof}[Proof of Lemma~\ref{l.intrep.p.}]
We begin with a reduction.
It suffices to show that for any \emph{fixed} $ (t,x) \in D $, the integral representation \eqref{e.intrep.p.} holds for almost every $ (t_3,b) $ in the designated range.
Once this result is established, it automatically extends to a countable, dense set of points in $ D $.
Both sides of \eqref{e.intrep.p.} are continuous in $ (t,x) $: The left side is smooth; the right side is continuous by Remark~\ref{r.intrep.}\ref{r.e.intrep.}.
Hence the desired result would follow.

We now fix $ (t,x) \in D $ and prove the reduced statement.

The proof begins by constructing a suitable test function and inserting it into \eqref{e.NS.weak}.
Fix a smooth step function $ \step\in\Csp^\infty(\R,[0,1]) $ with $ \step'\geq 0 $, $ \supp(\step')=[-1,-1/2] $, $ \step|_{(-\infty,-1]} \equiv 0 $, and $ \step|_{[-1/2,\infty)} \equiv 1 $.
Set $ \step_{u}(y) := \step(y/u) $.
Recall that $ I := [-b,b] $. 
Let $ \plateau_{I,v}(y) := \step_v(b+y) \step_v(b-y) $, which is a smooth plateau function with a step size controlled by $ v $; see Figure~\ref{f.modifiers}.
For small $ u,v>0 $, set
\begin{align}
	\label{e.testfn.}
	f(s,y) := \big(\phi_r * \Kern_{I} \big)(t,s,x,y) \cdot \step_u( t_3-s ) \cdot \plateau_{I,v}(y).
\end{align}
Currently, the parameters $ t_3,b $ can take any value within the designated range.
Subsequent arguments will restrict them to being almost every.
The step function and plateau function in \eqref{e.testfn.} truncate $ (\phi_r*\Kern_{I}) $ near $ \partial_\mathrm{P} ( [0,t_3]\times[-b,b] ) $.
Insert the test function in \eqref{e.testfn.} into \eqref{e.NS.weak} and expand the result.
Doing so gives
\begin{align}
	\label{e.intrep.Js}
	J_{1,u,v} = J_{2,u,v} + J_{3,u,v} + J_{4,u,v},
\end{align}
where 
\begin{align}
	\label{e.intrep.J1}
	J_{1,u,v} &= \int \d s \d y \, \p \cdot \big( ( \partial_s - \tfrac12 \partial_{yy} + \w )(\phi_r * \Kern_{I}) \big) \cdot \step_u \cdot \plateau_{I,v},
\\
	\label{e.intrep.J2}
	J_{2,u,v} &= \int \d s \d y \, \p \cdot (\phi_r * \Kern_{I}) \cdot \big(-\partial_s\step_u\big) \cdot \plateau_{I,v},
\\
	\label{e.intrep.J3}
	J_{3,u,v} &= \int \d s \d y \, \p \cdot \big( \partial_{y} (\phi_r * \Kern_{I}) \big) \cdot \step_u \cdot \big( \partial_y \plateau_{I,v} \big),
\\
	\label{e.intrep.J4}
	J_{4,u,v} &= \int \d s \d y \, \p \cdot (\phi_r * \Kern_{I}) \cdot \step_u \cdot \big( \tfrac12 \partial_{yy} \plateau_{I,v} \big),
\end{align}
and, we wrote $ (\phi_r * \Kern_{I}) = (\phi_r * \Kern_{I})(t,\Cdot,x,\Cdot) $ and $ \step_u = \step_u(t_3-\Cdot) $ to alleviate heavy notation.

Next we simplify $ J_1,\ldots,J_4 $ and take the limits $ u\to 0 $ and $ v\to 0 $ in order.

We begin with $ J_1 $.
First, note the identity $ ( \partial_s - \tfrac12 \partial_{yy} + \w )(\phi_r * \Kern_{I})(t,s,x,y) = \phi_r(t-s,x-y) $.
This holds because, by Duhamel's principle, the function $ (\phi_r * \Kern_{I}) $ solves the PDE $ ( \partial_s - \tfrac12 \partial_{yy} + \w )f = \phi_r(t-s,x-y) $.
Apply the identity to \eqref{e.intrep.J1}.
In the result, observe that, on the support of $ \phi_r(t-s,x-y) $, the functions $ \step_u $ and $ \plateau_{I,v} $ are constant $ 1 $.
Hence $ J_{1,u,v} = \int \d s \d y \, \p(s,y) \phi_r(s-t,x-y) \cdot 1 \cdot 1 = \phi_r * \p = \eqref{e.intrep.p.1} $.

Move on to $ J_{2,u,v} $.
Observe that $ -\partial_s\step_u(t_3-s) = \frac{1}{u} \step'((t_3-s)/u) $ acts as an approximation to the delta function $ \delta_{t_3}(s) $.
This observation suggests that, as $ u\to 0 $, we should have 
\begin{align*}
	J_{2,u,v} \to J_{2,v} := \int \d y \, (\phi_r * \Kern_{I})(t,t_3,x,y) \, \p(t_3,y)\plateau_{I,v}(t_3,y).
\end{align*}
To prove this convergence, with $ (t,x) $ being fixed, view $ J_{2,v} = J_{2,v}(t_3) $ as a function of $ t_3 \in ((t_2+T)/2,T) $.
Given that $ \p \in \Lsp^2_\loc\subset\Lsp^1_\loc  $, Fubini's theorem implies $ J_{2,v} \in \Lsp^1((t_2+T)/2,T) $.
Write
\begin{align*}
	\big|J_{2,u,v} - J_{2,v}\big|
	&=
	\Big|  \int_{t_3}^{t_3+u} \d s \, \tfrac{1}{u} \step'((t_3-s)/u) \, \big( J_{2,v}(s) - J_{2,v}(t_3) \big) \Big|
\\
	&\leq
	\norm{\step'}_{\infty;\R} \, \frac{1}{u} \int_{t_3}^{t_3+u} \d s \big| J_{2,v}(s) - J_{2,v}(t_3) \big|
\end{align*}
and apply the Lebesgue differentiation theorem to $ J_{2,v} $.
Doing so shows that $ J_{2,u,v} \to J_{2,v} $ for almost every $ t_3 \in ((t_2+T)/2,T) $.
The remaining limit $ v\to 0 $ is straightforward: For every $ t_3 $ such that $ \p(t_3,\Cdot) \in \Lsp^2_\text{loc}(\R) $, it is straightforward to check that $ J_{2,v} \to \eqref{e.intrep.p.2} $.

Now turn to $ J_{3,u,v} $.
The first limit $ u\to 0 $ is straightforward: We have $ J_{3,u,v} \to J_{3,v} := \int_{t-r}^{t_3} \d s \int_{\R} \d y \, \p \cdot \big( \partial_{y} (\phi_r * \Kern_{I}) \big) \cdot ( \partial_y \plateau_{I,v} ) $.
To proceed, write $ \partial_y \plateau_{I,v}(y) = \frac{1}{v} \step'((b+y)/v) \cdot \step_v(b-y) - \step_v(b+y) \cdot \frac{1}{v} \step'((b-y)/v) $.
Observe that each derivative is nonzero only when its companion step function is constant $ 1 $.
Hence $ \partial_y \plateau_{I,v}(y) = \frac{1}{v} \step'((y+b)/v) - \frac{1}{v} \step'((b-y)/v) $, which acts as an approximation to $ \delta_{-b}(y) - \delta_{b}(y) $.
Apply the same procedure for showing $ J_{2,u,v} \to J_{2,u} $.
Doing so gives, for almost every $ b $ in the designated range, the convergence $ J_{3,v} \to \text{(twice of \eqref{e.intrep.p.3})} $.

Now proceed to $ J_{4,u,v} $.
The first limit $ u\to 0 $ is straightforward: We have $ J_{4,u,v} \to J_{4,v} := \int_{t-r}^{t_3} \d s \int_{\R} \d y \, \p \cdot (\phi_r * \Kern_{I}) \cdot ( \frac{1}{2}\partial_{yy} \plateau_{I,v} ) $.
Next, write $ \partial_{yy} \plateau_{I,v}(y) = \frac{1}{v^2} \step''((y+b)/v) + \frac{1}{v^2} \step''((b-y)/v) $.
The factor $ \frac{1}{v^2} $ is too large to deal with, so we seek to reduce it.
The key is that $ (\phi_r * \Kern_{I})(t,x,s,\pm b) \equiv 0 $. (The convolution acts on $ (t,x) $ and hence does not change the zero boundary value.)
Given this observation, we Taylor expand $ (\phi_r * \Kern_{I}) $ in $ y $ around $ y=\pm b $ up to the first order and use the result to express $ J_{4,v} $:
\begin{align}
	\label{e.J4.expansion.1}
	J_{4,v}
	=&
	\sum_{\sigma=\pm }
	\frac{1}{2} \int_{t-r}^{t_3} \d s \int_{\sigma b}^{\sigma(b+v)} \d y \, 
	\p(s,y) \cdot \big(\partial_y(\phi_r * \Kern_{I})(t,x,s,\sigma b)\big) 
	\cdot \frac{1}{v} \step''\big(\tfrac{b -\sigma y}{v}\big) \tfrac{\sigma y-b}{v}
\\
	\label{e.J4.expansion.2}
	&+
	\sum_{\sigma=\pm}
	\int_{t-r}^{t_3} \d s \int_{\sigma b}^{\sigma(b+v)} \d y \, 
	\p(s,y) \cdot O(|\sigma b-y|^2) \cdot \frac{-\sigma}{v^2} \step''\big(\tfrac{b -\sigma y}{v}\big).
\end{align}
The expression in \eqref{e.J4.expansion.2} converges to $ 0 $.
To see why, note that the factor $ 1/v^2 $ is balanced by $ O(|\sigma b-y|^2) $, so the integral is bounded by a constant multiple of $ | \int_{t-r}^{t_3} \d s \int_{\sigma b}^{\sigma(b+v)} \d y \, p | $.
The last integral converges to zero as $ v\to 0 $, because $ p\in \Lsp^2_\loc $.
To handle the right side of \eqref{e.J4.expansion.1}, note that $ \int_{\pm b}^{\pm (b+v)} \d y \frac{1}{v} \step''(\tfrac{b \mp y}{v}) \frac{\pm y-b}{v} = \pm 1 $ and $ \norm{ \step''(\tfrac{b \mp y}{v}) \frac{\pm y-b}{v} }_{\infty;\pm[b,b+v]} = c <\infty $, where $ c $ is independent of $ v $.
This shows that $ \frac{1}{v} \step''(\tfrac{b \mp y}{v}) \tfrac{\pm y-b}{v} $ acts as an approximation to $ \pm \delta_{\pm b}(y) $, and the same procedure for showing $ J_{2,u,v} \to J_{2,u} $ applies here.
Applying the procedure gives, for almost every $ b $ in the designated range, the convergence $ J_{4,v} \to $ (negative of \eqref{e.intrep.p.3}).

Combining the preceding analysis of $ J_1,\ldots,J_4 $ with \eqref{e.intrep.Js} completes the proof.
\end{proof}

Sending $ r\to 0 $ in Lemma~\ref{l.intrep.p.} gives the following.
\begin{cor}
\label{c.intrep.p}
Fix any $ D = [t_1,t_2]\times[-a,a] \subset (0,T)\times\R $ and $ 0<r < \frac12 \max\{t_1,T-t_2,1\} $.
For almost every $ (t,x,t_3,b)\in D\times((T+t_2)/2,T)\times(a+1,\infty) $, the following holds for $ I = [-b,b] $:
\begin{align}
\label{e.intrep.p}
	\p(t,x)
	=
	\int_{-b}^b \d y \, \Kern_{I}  (t,t_3,x,y) \p(t_3,y)
	+
	\sum_{\sigma=\pm} \frac{-\sigma}{2} \int_{t}^{t_3} \d s \, (\partial_y\Kern_{I})(t,s,x,\sigma b) \p(s,\sigma b).
\end{align}
\end{cor}

\begin{rmk}
\label{r.intrep}
For almost every $ (t_3,b) $ in the designated range, the integrals in \eqref{e.intrep.p} are convergent and are continuous in $ (t,x)\in D $.
\end{rmk}
\begin{proof}[Proof of Corollary~\ref{c.intrep.p}]
As stated in Remark~\ref{r.intrep}, the right side of \eqref{e.intrep.p} is continuous in $ (t,x)\in D $.
Hence \eqref{e.intrep.p.2}+\eqref{e.intrep.p.3} converges to the right side of \eqref{e.intrep.p} as $ r\to 0 $.
Next, since $ \p \in \Lsp^2_\loc $, the function \eqref{e.intrep.p.2} converges to $ \p $ almost everywhere on $ D $ as $ r\to 0 $. 
\end{proof}

We can now conclude the continuity of $ \p $.
Fix any $ (t_3,b) \in (T+t_2)/2,T)\times(a+1,\infty) $ so that \eqref{e.intrep.p} holds for almost every $ (t,x)\in D $.
As stated in Remark~\ref{r.intrep}, the right side of \eqref{e.intrep.p} is continuous in $ (t,x)\in D $.
Hence, after redefining $ \p $ on a set of measure zero, the function $ \p $ is continuous in $ D $.
As this holds for all $ D \subset (0,T)\times\R $, the continuity of $ \p $ on $ (0,T)\times\R $ follows.

\subsection{The smoothness of $ \p $ and $ \q $}
\label{s.NS.smooth}
To obtain the $ \Csp^\infty $ smoothness we appeal to the standard regularity estimate; see \cite[Theorem~6]{wang03} for example.
The estimate states that, if $ \varphi $ solves $ \partial_t \varphi = \frac12 \partial_{xx} \varphi + \psi $ weakly with $ \psi \in \Lsp^\alpha_\loc $ and $ \alpha\in(1,\infty) $, then $ \partial_t \varphi, \partial_{xx} \varphi \in \Lsp^\alpha_\loc $.
We seek to apply this estimate with $ (\varphi,\psi) = (\q, \w\q) $ and with $ (\varphi,\psi) = (\p, \w\p)|_{t \,\mapsto T-t} $.
The function $ \q $ solves the equation in the weak sense because it solves the equation in the Duhamel sense; we proved in Section~\ref{s.NS.variation} that $ \p $ solves the equation in the weak sense.
The forcing terms $ \w\q = \p\q^2 $ and $ \w\p = \q\p^2 $ are in $ \Lsp^\infty_\loc $ because $ \q $ and $ \p $ are continuous.
Applying the regularity estimate gives $ \partial_t \q, \partial_t \p, \partial_{xx} \q, \partial_{xx} \p \in \cap_{\alpha\in(1,\infty)} \Lsp^\alpha_\loc $.
Higher-order regularity can be obtained inductively.
For example, applying $ \partial_{xx} $ to the second of the NLS equation gives $ -\partial_t(\partial_{xx} \p) = \frac12 \partial_{xx} (\partial_{xx}\p) + \partial_{xx}(\p^2 \q) $.
The forcing term $ \partial_{xx}(\p^2 \q) $, after being expanded, is seen to be in $ \cap_{\alpha\in(1,\infty)} \Lsp^\alpha_\loc $.
Applying the regularity estimate to $ \partial_{xx}\p $ gives $ \partial_{txx} \p, \partial_{xxxx} \p \in \cap_{\alpha\in(1,\infty)} \Lsp^\alpha_\loc $.
Proceeding this way yields that any derivative (in $ (t,x) $) of $ \q $ and $ \p $ is in $ \cap_{\alpha\in(1,\infty)} \Lsp^\alpha_\loc $, which implies $ \q,\p \in \Csp^\infty((0,T)\times\R) $. 

\subsection{Toward the Duhamel sense}
\label{s.NS.toDuhamel}
Our next task is to show that the integral representation~\eqref{e.iter.p} for $ \p $ holds.
Doing so will conclude that $ \p $ solves $ -\partial_t \p = \frac12 \partial_{xx} \p + \w \p $ in the Duhamel sense.
We proceed similarly to the proof of Lemma~\ref{l.intrep.p.}: Choose a test function similar to \eqref{e.testfn.} and insert it into the weak-sense equation.
Note that our goal here is a \emph{global} integral representation, not a local one. (See Remark~\ref{r.intrep.}\ref{r.e.intrep.} for the description of local versus global representations.)
With this in mind, for $ L > 0 $, we replace the plateau function in \eqref{e.testfn.} with a \emph{slowly-decaying plateau function} $ \sdplateau_L \in \Csp^\infty_\mathrm{c}(\R,[0,1]) $ such that $ \sdplateau_L|_{[-L,L]} \equiv 1 $ and that
\begin{align}
	\label{e.sdplateau}
	\sup \big\{ \big( |\partial_{y}\sdplateau_L(y)| + |\partial_{yy}\sdplateau_L(y)| \big) e^{|y|^3} \, : \, y\in\R, L >0 \big\} <\infty,
\end{align}
and will send $ L\to\infty $ later.
We require $ \sdplateau_L $ to have a compact support but do not specify its size.
We further forgo the step function in \eqref{e.testfn.} so that the support of the test function extends to $ s=T $.
The weak-sense equation \eqref{e.NS.weak} does allow the support to extend to $ s=T $ but requires the support to avoid a neighborhood of each $ (T,\xi_{\ii}) $.
We hence device the \emph{hole-puncher functions} $ \puncher_{u,v} := \prod_{\ii=1}^{\mm} \puncher_{u,v}^{\ii} $ and
\begin{align}
	\label{e.puncher}
	\puncher_{u,v}^{\ii}(s,y) 
	:=
	1 - \step_u(s-T) \cdot \plateau_{[-v+\xi_{\ii}, \xi_{\ii}+v], v}(y)
\end{align}
to remove the mass of the test function around each $ (T,\xi_{\ii}) $, where $ \plateau_{[a,b],v}(y) := \step_v(y-a) \step_v(b-y) $.
The test function here reads
\begin{align}
	\label{e.testfn}
	f(s,y) := \big(\phi_r * \Kern_{I} \big)(t,s,x,y) \cdot \sdplateau_L(y) \cdot \puncher_{u,v}(s,y),
	\qquad
	I = [-L,L].
\end{align}
Let $ \actregion_\ii(u,v) $ denote the region where the hole-puncher functions act:
\begin{align}
	\label{e.actregion}
	\actregion_\ii(u,v) := [-u+T,T]\times[-2v+\xi_\ii,\xi_\ii+2v] = \bar{\{ \puncher^\ii_{u,v} <1 \}},
	&&
	\actregion(u,v) := \cup_{\ii=1}^\mm \actregion_\ii(u,v).
\end{align}
We assume $ r,u,v,1/L $ are small enough so that the regions $ \bar{\{\phi_{r}(t-\Cdot,x-\Cdot) > 0\}} $, $ [0,T]\times\{|y|\geq L\} $, $ \actregion_1(u,v) $, \ldots, $ \actregion_\mm(u,v) $, and $ \{0\}\times(-L,L) $ do not overlap; see Figure~\ref{f.testfn-range}.

\begin{figure}
\begin{minipage}{.6\linewidth}
\includegraphics[width=1\linewidth]{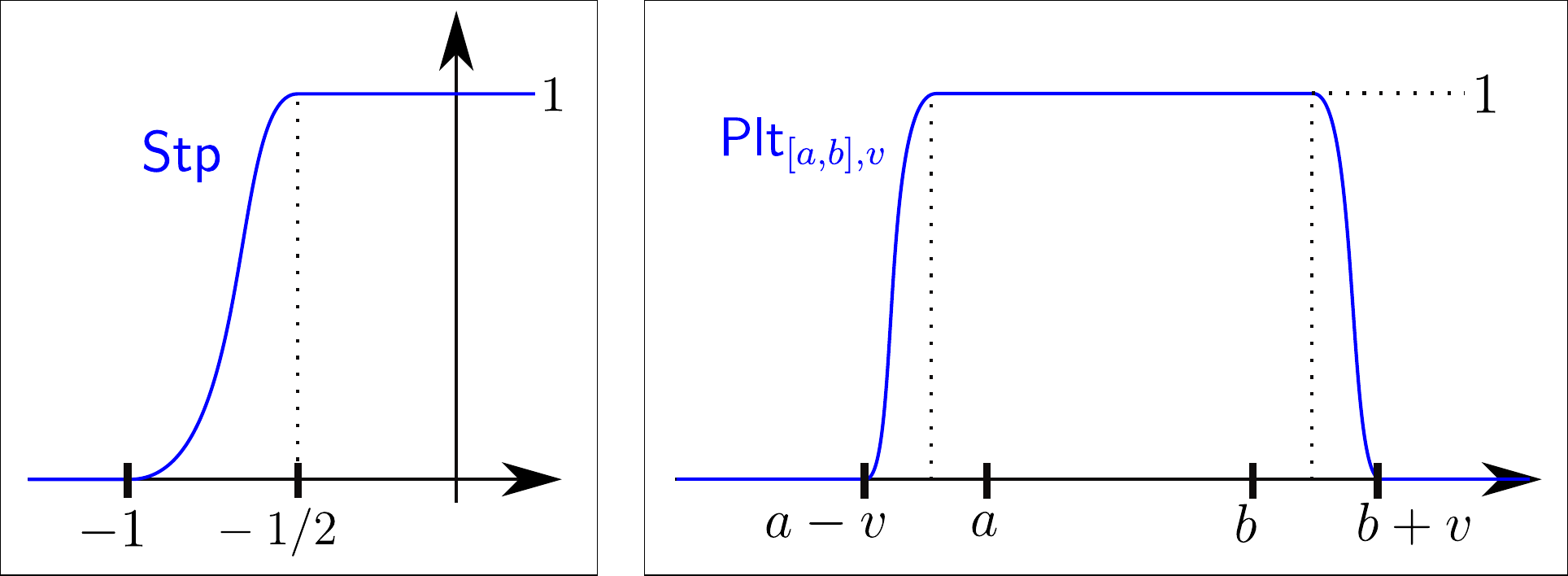}
\caption{The functions $ \step $ and $ \plateau_{[a,b],v} $.}
\label{f.modifiers}
\end{minipage}
\hfill
\begin{minipage}{.38\linewidth}
\includegraphics[width=1\linewidth]{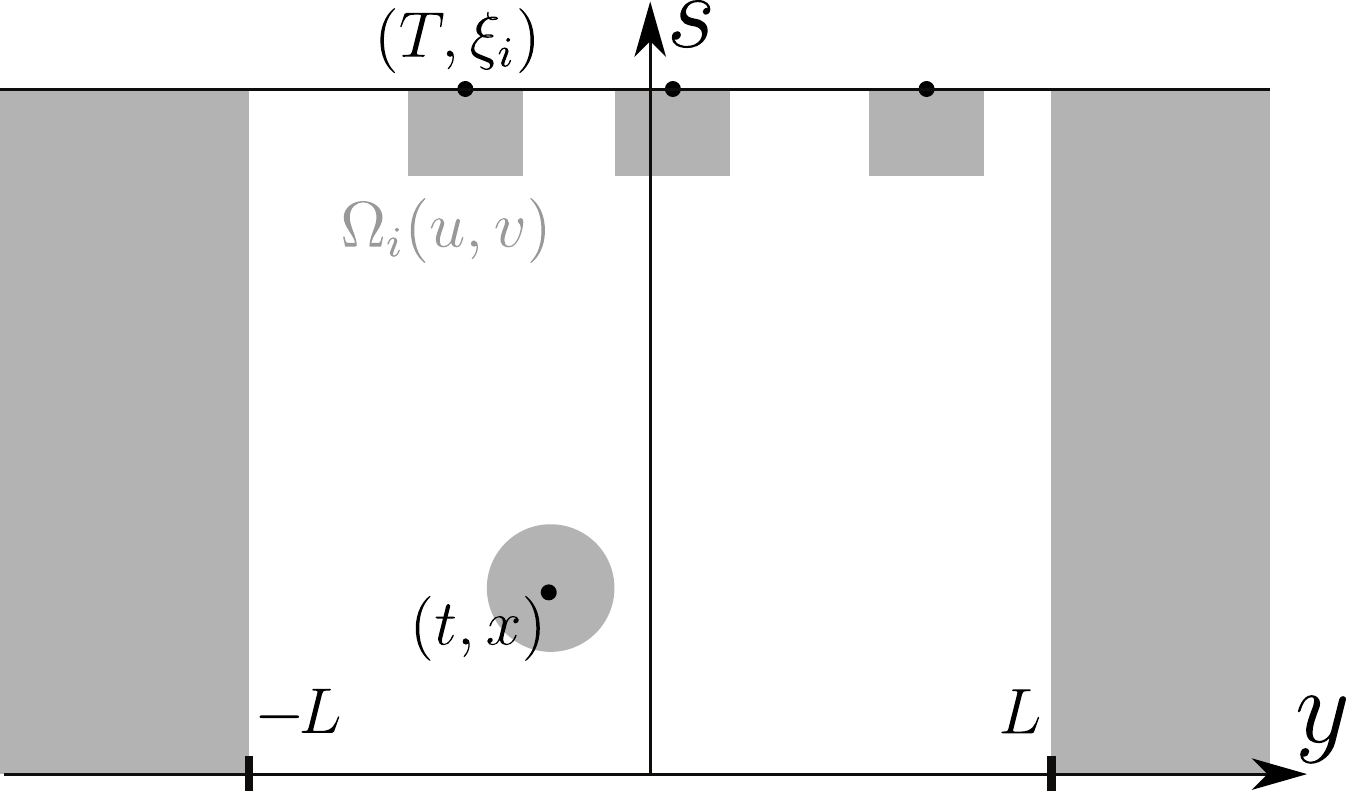}
\caption{The regimes involved in \eqref{e.testfn}.}
\label{f.testfn-range}
\end{minipage}
\end{figure}

Insert the test function in \eqref{e.testfn} into \eqref{e.NS.weak}, simplify the result, and send $ r\to 0 $.
The procedure is similar to the proof of Lemma~\ref{l.intrep.p.} and Corollary~\ref{c.intrep.p}, so we do not repeat it here. 
The result gives, for all $ (t,x)\in ((0,T)\times [-L,L])\setminus \actregion(u,v) $,
\begin{align}
	\label{e.inteq.approx.Luv}
	 \p(t,x) = \sum_{\ii=1}^{\mm}\tcterm^{\ii}_{L,u,v}(t,x)  + \bdyterm_{L,u,v}(t,x).
\end{align}
The function $ \tcterm^{\ii}_{L,u,v} $ is given by 
\begin{subequations}
\label{e.tcterm.Luv}
\begin{align}
	\tcterm_{L,u,v}^{\ii}&(t,x)
	:=
	\int_{\actregion_\ii(u,v)} \d s \d y \, 
	\Kern_I(t,s,x,y)  \p(s,y) \cdot \tfrac{1}{u} \step'\big(\tfrac{s-T}{u}\big) \plateau_{[-v+\xi_{\ii}, \xi_{\ii}+v],v}(y)
\\
\begin{split}
	&
	- 
	\int_{\actregion_\ii(u,v)} \d s \d y \, 
	\p(s,y) \step_u\big(s-T\big)
	\Big( 
		(\partial_y \Kern_I)(t,s,x,y)  \cdot \sum_{\sigma=\pm} \frac{-\sigma}{v}\step'\big(\tfrac{y-\sigma\xi_\ii}{v}\big) 
\\
	&
	\hphantom{\int_{\actregion_\ii(u,v)} \d s \d y \, \p(s,y) \step_u\big(s-T\big)\Big(  }
		+ \tfrac12 \Kern_I(t,s,x,y) \cdot \sum_{\sigma=\pm} \frac{1}{v^2}\step''\big(\tfrac{y-\sigma\xi_\ii}{v}\big)
	\Big).
\end{split}
\end{align}
\end{subequations}
To reiterate, $ I=[-L,L] $.
We use $ \tcterm $, which stands for \emph{terminal condition}, to denote this function, because the integrals in \eqref{e.tcterm.uv} are restricted to the region $ \actregion_\ii(u,v) $ around $ (T,\xi_\ii) $.
Next, 
\begin{subequations}
\label{e.bdyterm.Luv}
\begin{align}
	\bdyterm_{L,u,v}(t,x)
	:=
	\sum_{\sigma=\pm} 
	\int_t^T \d s \int_{\{y>\sigma L\}}\d y \,  \p(s,y)  
	\Big(&
		\tfrac12 \Kern_I(t,s,x,y) \partial_{yy} \sdplateau_L(y)
\\
	&
	+ (\partial_y \Kern_I) (t,s,x,y) \partial_y \sdplateau_L(y)	\Big).
\end{align}
\end{subequations}
We use $ \bdyterm $, which stands for \emph{boundary}, to denote this function, as the integrals are restricted to $ \{L< |y| \} $.

Next we send $ L\to\infty $ in \eqref{e.inteq.approx.Luv}.
This amounts to taking the $ L\to\infty $ limits of $ \tcterm^{\ii}_{L,u,v} $ and $ \bdyterm_{L,u,v} $.
Fix $ (t,x)\in((0,T)\times\R)\setminus \actregion(u,v) $.
We begin by showing that $ \bdyterm_{L,u,v}(t,x)\to 0 $ as $ L\to\infty $.
In \eqref{e.bdyterm.Luv}, write $ \p = \w/\q $ and apply the Cauchy--Schwarz inequality to get
\begin{align*}
	|\bdyterm_{L,u,v}(t,x)|^2
	\leq
	\norm{\w}^2_2
	\int_t^T \d s \int_{\{|y|>L\}} \d y \, \big((\Kern_I)^2 + (\partial_{y} \Kern_I)^2 \big) \cdot \frac{|\partial_y\sdplateau_L|^2+|\partial_{yy}\sdplateau_L|^2 }{\q^2},
\end{align*}
where we wrote $ \Kern_{I} := \Kern_{I}(t,\Cdot,x,\Cdot) $ and similarly for $ \partial_y\Kern_I $.
Using \eqref{e.lwbd.q} and \eqref{e.sdplateau} shows that the fraction in the last integral is bounded on $ (s,y)\in[t,T]\times \R $.
With $ (t,x) $ being fixed, from \eqref{e.bd.hkiter} and \eqref{e.Kern}, it is not hard to verify that $ \Kern_I(t,s,x,y) $ and $ (\partial_y \Kern_I) (t,s,x,y) $ converges to zero as $ L\to\infty $ uniformly on $ (s,y)\in[t,T]\times\{|y|\geq L\} $.
Hence $ \bdyterm_{L,u,v}(t,x) \to 0 $.
Next, to handle the $ L\to\infty $ limit of $ \tcterm^{\ii}_{L,u,v} $, first note that $ \Kern_{I}(t,s,x,y) $ with $ I=[-L,L] $ converges to $ \Kern(t,s,x,y) $ uniformly on $ (s,y)\in \actregion_\ii(u,v) $.
This property can be straightforwardly (though tediously) verified from \eqref{e.Kern} and \eqref{e.KernI} with the aid of \eqref{e.bd.hkiter}.
Also, from $ \p=\w/\q $ and the bound \eqref{e.lwbd.q}, we have $ \p \in \Lsp^2(\actregion_{\ii}(u,v)) $. 
Combining these properties gives $ \tcterm^{\ii}_{L,u,v}(t,x) \to \tcterm^{\ii}_{u,v}(t,x) $, where $ \tcterm^{\ii}_{u,v}(t,x) $ is obtained by replacing $ \Kern_{I} $ with $ \Kern $ in \eqref{e.tcterm.Luv}, namely
\begin{subequations}
\label{e.tcterm.uv}
\begin{align}
	\tcterm_{u,v}^{\ii}&(t,x)
	\label{e.tcterm.uv.main}
	:=
	\int_{\actregion_\ii(u,v)} \d s \d y \, 
	\Kern(t,s,x,y)  \p(s,y) \cdot \tfrac{1}{u} \step'\big(\tfrac{s-T}{u}\big) \plateau_{[-v+\xi_{\ii}, \xi_{\ii}+v],v}(y)
\\
	\label{e.tcterm.uv.remainder}
\begin{split}
	&
	- 
	\int_{\actregion_\ii(u,v)} \d s \d y \, 
	\p(s,y) \step_u\big(s-T\big)
	\Big( 
		(\partial_y \Kern)(t,s,x,y)  \cdot \sum_{\sigma=\pm} \frac{-\sigma}{v}\step'\big(\tfrac{y-\sigma\xi_\ii}{v}\big) 
\\
	&
	\hphantom{\int_{\actregion_\ii(u,v)} \d s \d y \, \p(s,y) \step_u\big(s-T\big)\Big(  }
		+ \tfrac12 \Kern(t,s,x,y) \cdot \sum_{\sigma=\pm} \frac{1}{v^2}\step''\big(\tfrac{y-\sigma\xi_\ii}{v}\big)
	\Big).
\end{split}
\end{align}
\end{subequations}
Combining these $ L\to\infty $ limits with \eqref{e.inteq.approx.Luv} gives, for all $ (t,x)\in ((0,T)\times \R)\setminus \actregion(u,v) $,
\begin{align}
	\label{e.inteq.approx.uv}
	 \p(t,x) = \sum_{\ii=1}^{\mm}\tcterm^{\ii}_{u,v}(t,x).
\end{align}

\subsection{The limit $ u,v\to 0 $}
\label{s.NS.uv}
We seek to send $ u,v\to 0 $ in \eqref{e.inteq.approx.uv} and show that the result gives \eqref{e.iter.p}.
So far $ u,v $ can be arbitrary (as long as they are small enough).
Hereafter, set $ u=v^4 $.
Our proof actually works for $ u=v^\alpha $ for any $ \alpha>3 $.

We need to take the $ v\to 0 $ limit of $ \tcterm_{v^4,v}^{\ii} $.
It consists of two terms in \eqref{e.tcterm.uv.main}--\eqref{e.tcterm.uv.remainder} and we begin with the latter.
Fix $ (t,x)\in(0,T)\times\R $ and consider $ v $ small enough such that $ (t,x) \notin \actregion_{\ii}(v^4,v) $.
Within \eqref{e.tcterm.uv.remainder}, set $ u= v^4 $ and write $ \p = \w/\q $; use the boundedness of $ \Kern, \step', \step'' $, and $ 1/\q $ over $ (s,y)\in\actregion_{\ii}(v^4,v) $ to bound the result.
We have $ |\eqref{e.tcterm.uv}| \leq c(\norm{w}_2,t,x) \frac{1}{v^2} \norm{\w}_{1;\actregion_\ii(v^4,v)} $.
The last factor, by the Cauchy--Schwarz inequality, is bounded by $ \norm{\w}_2 \cdot |\actregion_{\ii}(v^4,v)|^{1/2} = \norm{\w}_2 (4v^{4+1})^{1/2} $.
Consequently, $ \eqref{e.tcterm.uv.remainder} \to 0 $.
Let $ o(1) $ denote a generic quantity that converges to zero as $ v\to 0 $ for any fixed $ (t,x) \in (0,T)\times\R $.
So far we have, for any $ (t,x) \in (0,T)\times\R $,
\begin{align}
	\label{e.inteq.approx.v}
	\p(t,x) = \sum_{\ii=1}^{\mm} \eqref{e.tcterm.uv.main}|_{u=v^4}  + o(1).
\end{align}

To take the limit $ v\to 0 $ of \eqref{e.tcterm.uv.main} requires a property of $ \w $.
The property is the `local, approximate definite-sign property' discussed in Section~\ref{s.results.tc}, which we now state precisely.
Let $ f_\pm := |f| \ind_{\{\pm f > 0\}} $ denote the positive/negative part of a function $ f $.
\begin{prop}\label{p.definite}
Let $ \w $ be a minimizer of \eqref{e.minimization}.
For all $ \mu\in[1,2] $ and $ \ii=1\,\ldots,\mm $,
\begin{align*}
	\min\Big\{ \norm{\w_-}_{\mu;\actregion_{\ii}(v^4,v)} \, , \norm{\w_+}_{\mu;\actregion_{\ii}(v^4,v)} \Big\}
	\leq
	c(T,\w) \exp\big(-1/(cv^2)\big),
\end{align*}
where the second constant $ c\in(0,\infty) $ is universal.
\end{prop}
\noindent{}%
We will only use this proposition for $ \mu=1 $.
To interpret this proposition, let $ \sign_{\ii}(v) := + $ when $ \norm{\w_+}_{1;\actregion_{\ii}(v^4,v)} \geq \norm{\w_-}_{1;\actregion_{\ii}(v^4,v)} $ and $ \sign_{\ii}(v) := -1 $ when $ \norm{\w_+}_{1;\actregion_{\ii}(v^4,v)} < \norm{\w_-}_{1;\actregion_{\ii}(v^4,v)} $, and decompose $ \w|_{\actregion_{\ii}(v^4,v)} $ into its positive and negative parts as $ \w = \sign_{\ii}(v)\w_{\sign_{\ii}(v)} - \sign_{\ii}(v) \w_{-\sign_{\ii}(v)} $.
By the definition of $ \sign_{\ii}(v) $ and by Proposition~\ref{p.definite} for $ \mu=1 $, the second part of the decomposition is small in $ \Lsp^1 $, more precisely $ \norm{\w_{-\sign_{\ii}(v)}}_{1;\actregion_{\ii}(v^4,v)} \leq c(T,\w) \exp(-1/(cv^2)) $.
Hence, $ \w|_{\actregion_{\ii}(v^4,v)} $ approximates $ \sign_{\ii}(v)\w_{\sign_{\ii}(v)} $, which has a definite sign.

Let us finish taking the $ v\to 0 $ limit of \eqref{e.inteq.approx.v} and postpone the proof of Proposition~\ref{p.definite} to Section~\ref{s.NS.definite}.
Within the integral of \eqref{e.tcterm.uv.main}, write $ \p = \sign_{\ii}(v) \p_{\sign_{\ii}(v)} - \sign_{\ii}(v) \p_{-\sign_{\ii}(v)} $ and divide the integral into two accordingly.
The first integral is
\begin{subequations}
\label{e.tcterm.v.main}
\begin{align}
	\tcterm^{\star,\ii}_{v}(t,x)
	:=
	\sum_{\ii=1}^{\mm} (\sign_{\ii})(v) \int_{\actregion_\ii(v^4,v)}  \d s \d y \,
	&\p_{\sign_{\ii}(v)}(s,y) \Kern(t,s,x,y)
\\
	&
	\cdot 
	\tfrac{1}{v^4}\step'\big(\tfrac{s-T}{v^4}\big) \plateau_{[-v+\xi_{\ii}, \xi_{\ii}+v],v}(y).
\end{align} 
\end{subequations}
For the second integral, write $ \p=\w/\q $, use the boundedness of $ 1/\q $ and $ \Kern $ on $ (s,y)\in\actregion_\ii(v^4,v) $, and use $ \norm{\w_{-\sign_{\ii}(v)}}_{1;\actregion_{\ii}(v^4,v)} \leq c(T,\w) \exp(-1/(cv^2)) $.
Doing so shows that the second integral is $ o(1) $.
Hence,
\begin{align}
	\label{e.inteq.approx.v.}
	\p(t,x) 
	=
	\sum_{\ii=1}^{\mm} \tcterm^{\star,\ii}_{v}(t,x)+ o(1).
\end{align}

We next take the $ v\to 0 $ limit in \eqref{e.inteq.approx.v.} and show that the result gives \eqref{e.iter.p}.
The key is to observe that the integrand in \eqref{e.tcterm.v.main} is non-negative:
The functions $ \p_{\sign_{\ii}(v)}, \step' $, and $ \plateau $ are non-negative by definition; the function $ \Kern $ is non-negative thanks to \eqref{e.Kern.FK}.
Granted these observations, defining
\begin{align*}
	\gamma_\jj(v) 
	&:=
	(\sign_{\jj})(v) \int_{\actregion_\ii(v^4,v)} \d s \d y \, \p_{\sign_{\jj}(v)}(s,y) 
	\tfrac{1}{v^4}\step'\big(\tfrac{s-T}{v^4}\big) \plateau_{[-v+\xi_{\jj}, \xi_{\jj}+v],v}(y),
\\
	\bar{\Kern}_{\jj}(t,x,v)
	&:=
	\begin{cases}
		\sup\limits_{(s,y)\in\actregion_{\jj}(v^4,v)} \Kern(t,s,x,y),	&\text{when } \sign_{\jj}(v) = +,
		\\
		\inf\limits_{(s,y)\in\actregion_{\jj}(v^4,v)} \Kern(t,s,x,y),	&\text{when } \sign_{\jj}(v) = -,
	\end{cases}
\\
	\underline{\Kern}_{\jj}(t,x,v)
	&:=
	\begin{cases}
		\inf\limits_{(s,y)\in\actregion_{\jj}(v^4,v)} \Kern(t,s,x,y),	&\text{when } \sign_{\jj}(v) = +,
		\\
		\sup\limits_{(s,y)\in\actregion_{\jj}(v^4,v)} \Kern(t,s,x,y),	&\text{when } \sign_{\jj}(v) = -,
	\end{cases}
\end{align*}
we have
\begin{align}
	\label{e.vlimit0}
	\underline{\Kern}_{\jj}(t,x,v) \gamma_\jj(v) \leq \tcterm^{\jj,\star}_v(t,x) \leq \bar{\Kern}_{\jj}(t,x,v) \gamma_\jj(v).
\end{align}
Using the last inequality in \eqref{e.inteq.approx.v.} gives, for any fixed $ (t,x)\in(0,T)\times\R $,
\begin{align}
	\label{e.vlimit}
	\p(t,x) + o(1) \leq \sum_{\jj=1}^{\mm} \bar{\Kern}_{\jj}(t,x,v) \gamma_\jj(v),
	\qquad
	\sum_{\jj=1}^{\mm} \underline{\Kern}_{\jj}(t,x,v) \gamma_\jj(v) \leq \p(t,x) + o(1).
\end{align}
Fix a $ t_0\in(0,T) $ such that the matrix $ (\Kern(t_0,T,\xi_{\ii},\xi_{\jj}))_{\ii,\jj=1}^{\mm} $ is invertible.
Such a $ t_0 $ exists because, as $ t\to T $, for all $ \ii\neq\jj $ we have $ \Kern(t,T,\xi_{\ii},\xi_{\jj}) \to 0 $, while $ \Kern(t,T,\xi_{\jj},\xi_{\jj}) \to +\infty $.
Set $ (t,x) = (t_0,\xi_\ii) $ in \eqref{e.vlimit} for $ \ii=1,\ldots,\mm $ to get
\begin{align}
	\label{e.vlimit1}
	\p(t_0,\xi_{\ii}) + o(1) \leq \sum_{\jj=1}^{\mm} \bar{\Kern}_{\jj}(t_0,\xi_{\ii},v) \gamma_\jj(v),
	\qquad
	\sum_{\jj=1}^{\mm} \underline{\Kern}_{\jj}(t_0,\xi_{\ii},v) \gamma_\jj(v) \leq \p(t_0,\xi_{\ii},x) + o(1).
\end{align}
Since $ \Kern(t,x,s,y) $ is continuous in $ (s,y) $, the functions $ \bar{\Kern}_{\jj}(t_0,\xi_{\ii},v) $ and $ \und{\Kern}_{\jj}(t_0,\xi_{\ii},v) $ converge to $ \Kern(t_0,T,\xi_{\ii},\xi_{\jj}) $ as $ v\to 0 $.
Using this property in \eqref{e.vlimit1}, with the aid of a simple linear algebra tool, Lemma~\ref{l.la}, we have
\begin{align*}
	\lim_{v\to 0}
	\begin{pmatrix}
		\gamma_1(v) \\ \vdots \\ \gamma_\mm(v) 
	\end{pmatrix}
	=
	\begin{pmatrix}
		~\\
		&\Kern(t_0,T,\xi_{\ii},\xi_{\jj}) & \\~ 
	\end{pmatrix}^{-1}
	\begin{pmatrix}
		\p(t_0,\xi_{1}) \\ \vdots \\ \p(t_0,\xi_{\mm})
	\end{pmatrix}
	:=
	\begin{pmatrix}
		\gamma_1 \\ \vdots \\ \gamma_\mm 
	\end{pmatrix}.
\end{align*}
Combining this with \eqref{e.vlimit0} gives $ \tcterm^{\star,\ii}_v(t,x) \to \gamma_{\ii} \hk(T-t,\xi_{\ii}-x) $, for any fixed $ (t,x) \in (0,T)\times\R $.
Using this property to pass \eqref{e.inteq.approx.v.} to the limit $ v\to 0 $ gives \eqref{e.iter.p}.

\subsection{Proof of Proposition~\ref{p.definite}: the local, approximate definite-sign property of $ \w $}
\label{s.NS.definite}

We begin by establishing Lemma~\ref{l.vary}.
To simplify notation, write $ \qfn_{\jj}[\theta] := \qfn_\jj[\theta](T,\xi_\ii) $, $ \actregion_{\jj} := \actregion_{\jj}(v^4,v) $, and set
\begin{align}
	\label{e.maxfn}
	\maxfn_{\jj}[r] 
	:= 
	\sup_{\norm{\theta}_2 \leq r} 
	\int_{\R} \E_{x\to y}\big[ e^{\int_0^T \d t \theta(t,W(T-t))}  \big]^{1/2} \qic(y) \hk(T,y),
\end{align}
where $ W $ denotes the Brownian bridge with $ W(0)=\xi_{\jj} $ and $ W(T)=y $.
By \eqref{e.Kern.FK}, the expectation in \eqref{e.maxfn} is $ \Kern(0,T,y,\xi_{\jj})/\hk(T,\xi_{\jj}-y) $.
Using this identity and the bound \eqref{e.bd.hkiter} shows that $ \maxfn_{\jj}[r] < \infty $, for all $ r<\infty $.
\begin{lem}\label{l.vary}
Notation as in the preceding. Fix $ \theta\in\Lsp^2 $ and $ \ii\in\{1,\ldots,\mm\} $.
\begin{enumerate}[leftmargin=20pt,label=(\alph*)]
\item \label{l.vary.1}
The function $ \qfn_{\ii}[\Cdot]: \Lsp^2 \to (0,\infty) $ is continuous.
\item \label{l.vary.2}
For all $ \eta \in \Lsp^2 $ with $ \supp(\eta) \subset \actregion_{\ii} = \actregion_{\ii}(v^4,v) $ and all $ \jj \neq \ii $,
\begin{align*}
	\big| \qfn_{\jj}[\theta+\eta] - \qfn_{\jj}[\theta] \big|
	\leq
	c\, e^{-1/(cv^4)} \, \maxfn_{\jj}\big[2\norm{\theta}_2+2\norm{\eta}_2\big].
\end{align*}
That is, for $ \jj\neq\ii $, perturbing $ \theta $ within $ \actregion_{\ii} $ changes $ \qfn_{\jj}[\theta+\Cdot] $ by very little.
\item \label{l.vary.4}
Set $ \vec{f}(\vec{a}) := (\qfn_{\ii}[\theta  + a_1 \ind_{\actregion_1} + \ldots + a_{\mm} \ind_{\actregion_{\mm}}])_{\ii=1}^{\mm} $.
There exists $ c=c(T,\theta) $ such that, for all $ v \leq 1/c $ and $ \vec{b} \in \R^\mm $ with $ |\vec{b}| \leq 1/c $, the equation $ \vec{f}(\vec{a}) - \vec{f}(\vec{0}) = \vec{b} $ has a unique solution within $ \{|\vec{a}| \leq v^{-4}\} $, where $ |\ | $ denotes the Euclidean norm.
Further, the solution satisfies $ |\vec{a}| \leq c(T,\theta) |\vec{b}| v^{-4} $.
\end{enumerate}
\end{lem}
\begin{proof}
\ref{l.vary.1}
This follows from the argument in \cite[Lemma~3.7]{lin21}, (which shows that the function is continuous even under a weaker topology than $ \Lsp^2 $).

\ref{l.vary.2}
Use the Feynman--Kac formula~\eqref{e.FKq} to write
\begin{align}
	\label{l.vary.2.1}
	\qfn_{\jj}[\theta+\eta] - \qfn_{\jj}[\theta]
	=
	\int_{\R} \E_{x\to y}\big[ e^{\int_0^T \d t \, (\theta+\eta)(t,W(T-t))} - e^{\int_0^T \d t \, \theta(t,W(T-t))} \big] \qic(y) \hk(T,y).
\end{align}
The last expectation is nonzero only if the Brownian bridge $ W $ visits $ \actregion_{\ii}(v^4,v) $.
Bound the expectation by 
\begin{align}
	\label{l.vary.2.2}
	\E_{x\to y}\big[ \ind\{\text{visit happens}\} \big( e^{\int_0^T \d t \, (\theta+\eta)(t,W(T-t))} + e^{\int_0^T \d t \, \theta(t,W(T-t))} \big) \big].
\end{align}
Note that the Brownian bridge (which travels backward in time) starts from $ (T,\xi_{\ii}) $.
For the visit to happen, the Brownian bridge needs to travel a distance of at least $ |\xi_{\ii}-\xi_{\jj}|-2v $ within $ v^4 $ unit of time.
Such an event happens with probability $ \leq c\,\exp(-1/(cv^4)) $.
Distribute the sum in \eqref{l.vary.2.2} to decompose the expectations into two expectations; apply the Cauchy--Schwarz inequality to both expectations; insert the result into \eqref{l.vary.2.1}.
Doing so gives the desired result.

\ref{l.vary.4}
The first step is to analyze the derivative $ \partial_{\jj} f_{\ii} := \partial_{a_{\jj}} f_{\ii} $.
We claim that
\begin{align}
	\label{e.observ.diff}
	\big| \partial_{\jj} f_{\ii} - \ind_{\{\ii=\jj\}} v^4 f_{\ii} \big| \leq c\, e^{-1/(cv^2)} \maxfn_{\ii}\big[ 2\big( \norm{\theta}_2 + |a_1| + \ldots + |a_\mm| \big) \big],
\end{align}
for some universal $ c \in (0,\infty) $.
Express $ f_{\ii} $ by the Feynman--Kac formula~\eqref{e.FKq} and differentiate it to get
\begin{align}
	\label{e.FK.diff}
	\partial_{\jj} f_{\ii} 
	=
	\int_\R \d y\, \E_{x\to y} \Big[ \Big(\int_{T-v^4}^T \d t \ind_{[-2v+\xi_{\ii},\xi_{\ii}+2v]}(W(T-t)) \Big)\cdot \exp(\ldots) \Big] \hk(T,x-y)\qic(y),
\end{align}
where $ \exp(\ldots) := \exp( \int_0^T \d t\, (\theta+\sum_{\ii=1}^{\mm}a_{\ii} \ind_{\actregion_\ii})(t,W(T-t)) ) $, and $ W $ starts from $ \xi_{\jj} $ and ends at $ x $.
Consider first $ \jj\neq\ii $.
The integral in \eqref{e.FK.diff} is nonzero only if the Brownian bridge $ W $ visits $ \actregion_{\ii} $.
Bounding the integral by $ v^4 \ind\{ \text{visit happens} \} $ and applying the same argument in \ref{l.vary.2} give the claim \eqref{e.observ.diff}.
Consider next $ \jj=\ii $.
Observe that, if one replaces the integral in \eqref{e.FK.diff} with $ v^4 $, the result becomes $ v^4 f_{\ii} $. 
Subtract $ v^4 f_{\ii} $ from both sides of \eqref{e.FK.diff} and use the observation to simplify the result.
In the result, note that the difference $ (\int_{T-v^4}^T \d t \ind_{[-2v+\xi_{\ii},\xi_{\ii}+2v]}(W(T-t)) - v^4) $ is nonzero only if the Brownian bridge exits $ \actregion_{\ii}(v^4,v) $ with $ t\in[T-v^4,T] $.
Bound the last difference by $ v^4 \ind\{ \text{exit happens} \} $.
With $ \jj=\ii $, the Brownian bridge starts from $ \xi_{\ii} $, so the exit happens with probability $ \leq c\, \exp(-1/(cv^{4-2})) $.
From here, applying the same argument in \ref{l.vary.2} gives the claim \eqref{e.observ.diff}.

We now use \eqref{e.observ.diff} to prove the desired result.
Consider $ \vec{g}(\vec{a}_1) := (\vec{f}(v^{-4}\vec{a}_1)-\vec{f}(\vec{0})) : \{|\vec{a}_1| \leq 1\} \to \R^{\mm} $.
By \eqref{e.observ.diff}, for all $ v $ small enough (depending only on $ T $ and $ \theta $), the function $ \vec{g} $ is bi-Lipschitz, namely
\begin{align*}
	\tfrac{1}{c(\theta,T)} \, \big| \vec{a}_1 - \vec{a}_2 \big|	
	\leq
	\big| \vec{g}(\vec{a}_1) - \vec{g}(\vec{a}_2) \big|
	\leq
	c(\theta,T) \, \big| \vec{a}_1 - \vec{a}_2 \big|,
	\quad
	\text{for all }
	|\vec{a}_1|, |\vec{a}_2| \leq 1.
\end{align*}
Such a function is a homeomorphism from its domain to its image; see \cite[Observation~28.8]{yeh14}.
Since the image is homeomorphic to a closed ball and contains $ \vec{g}(\vec{0})=\vec{0} $, it must contains $ \{\vec{b}: |\vec{b}| \leq c \} $ for a small enough $ c>0 $.
Hence the desired solvability follows.
The bound $ |\vec{a}| \leq c(T,\theta) |\vec{b}| v^{-4} $ follows from the bi-Lipschitz property with $ (\vec{a}_1,\vec{a}_2) \mapsto (v^{4}\vec{a},\vec{0}) $.
\end{proof}

We now begin the proof of Proposition~\ref{p.definite}.
Fix a $ \w $ that minimizes \eqref{e.minimization} and fix an $ \ii $.
We will only consider $ \mu=2 $, which suffices since $ \norm{\ }_{\mu;\actregion_{\ii}} := \norm{\ }_{\mu;\actregion_{\ii}(v^4,v)} $ increases in $ \mu $ for all $ (v^4\cdot 4v) \leq 1 $.

The proof consists of two surgeries on $ \w $: $ \w\mapsto\theta $ and $ \theta\mapsto\eta $.

In the first surgery, we modify $ \w $ within $ \actregion_{\ii} $ so that the result $ \theta $ has a definite sign in $ \actregion_{\ii} $ and that the terminal value at $ \xi_{\ii} $ remains unchanged: $ \qfn_{\ii}[\theta] = e^{\alpha_{\ii}} $.
Before performing the surgery, we need to decide whether to make the definite sign of $ \theta $ positive or negative, namely whether to make $ \theta|_{\actregion_{\ii}} \geq 0 $ or $ \theta|_{\actregion_{\ii}} < 0 $.
To make the decision, remove the portion of $ \w $ within $ \actregion_{\ii} $ to get $ \w \ind_{\actregion_{\ii}^\cc} $ and examine whether the terminal value at $ \xi_{\ii} $ lies below or above the target, namely whether
\begin{align*}
	\qfn_{\ii}[\w \ind_{\actregion_{\ii}^\cc}] \leq e^{\alpha_{\ii}} 
	\quad\text{or}\quad 
	\qfn_{\ii}[\w \ind_{\actregion_{\ii}^\cc}] > e^{\alpha_{\ii}}.
\end{align*}
In the first case we make the definite sign positive and in the second case negative.
Let us consider the first case, and the second case can be proven by the same argument.
To perform the surgery, forgo the negative part of $ \w|_{\actregion_{\ii}} $ and scale its positive part by $ b \in [0,1] $ to get
$
	\theta(b) := w \, \ind_{\actregion_{\ii}^\cc} + b \w_+ \, \ind_{\actregion_{\ii}}.
$
When $ b=0 $, we have $ \theta(0)= \w \, \ind_{\actregion_{\ii}^\cc} $ so $ \qfn_{\ii}[\theta(0)] \leq e^{\alpha_{\ii}} $.
When $ b=1 $, we have $ \theta(1) \geq \w $ everywhere, so by the monotonicity of $ \qfn[\Cdot] $ (which follows from the Feynman--Kac formula~\eqref{e.FKq}) we have $ \qfn_{\ii}[\theta(1)] \geq \qfn_{\ii}[\w] = e^{\alpha_{\ii}} $.
By Lemma~\ref{l.vary}\ref{l.vary.1}, the mapping $ b\mapsto \qfn_{\ii}[\theta(b)] $ is continuous, so there exists $ b_*\in[0,1] $ such that $ \qfn_{\ii}[\theta(b_*)] = e^{\alpha_{\ii}} $.
This gives the desired product of the first surgery $ \theta := \theta(b_*) $.

Before performing the second surgery, let us examine the properties of $ \theta $.
Going from $ \w $ to $ \theta $, we managed to keep the value of $ \qfn_{\ii}[\Cdot] $ unchanged.
On the other hand, for $ \jj\neq\ii $, the values of $ \q_{\jj}[\Cdot] $ may have changed, namely $ \q_{\jj}[\theta]\neq\q_{\jj}[\w]=e^{\alpha_{\jj}} $ in general.
Yet, since $ \theta $ and $ \w $ differ only within $ \actregion_{\ii} $, by Lemma~\ref{l.vary}\ref{l.vary.2} the change is very small:
\begin{align}
	\label{e.small}
	\big| \qfn_{\jj}[\theta] - e^{\alpha_{\jj}} \big| \leq c(T,\w) \, \exp(-1/(cv^4)).
\end{align}
We next examine the square $ \Lsp^2 $ norm of $ \theta $.
By the definition of $ \theta $,
\begin{align}
	\norm{\theta}_2^2
	&= 
	\norm{\w}^2_2 - (1-b_*^2) \tfrac12\norm{\w_+}^2_{2;\actregion_{\ii}} - \norm{\w_-}^2_{2;\actregion_{\ii}}
	\label{e.cost}
	\leq
	\norm{\w}^2_2 - \norm{\w_-}^2_{2;\actregion_{\ii}}.
\end{align}
Namely, the quantity $ \norm{\theta}_2^2 $ is at least $ (\norm{\w_-}^2_{2;\actregion_{\ii}}) $ smaller than the minimal cost $ \norm{\w}^2_2 $ of attending all the target terminal values.

The second surgery is to fine tune $ \theta $ so that the resulting $ \qfn_1[\Cdot],\ldots,\qfn_\mm[\Cdot] $ attend the target terminal values $ e^{\alpha_1},\ldots,e^{\alpha_\mm} $.
The tuning works by adding a constant on each $ \actregion_\jj $:
$
	\eta(\vec{a}) := \theta  + a_1 \ind_{\actregion_1} + \ldots + a_{\mm} \ind_{\actregion_{\mm}}.
$
We seek to apply Lemma~\ref{l.vary}\ref{l.vary.4} with $ b_{\jj} := e^{\alpha_{\jj}}-\qfn_{\jj}[\theta] $.
By \eqref{e.small}, we have $ |\vec{b}| \leq c(T,\w) \exp(-1/(cv^4)) $, which is smaller than the required threshold $ 1/c(T,\w) $ when $ v $ is small enough.
Note that \eqref{e.small} holds also for $ \jj=\ii $ because $ \qfn_{\ii}[\theta] - e^{\alpha_{\ii}} =0 $.
Apply Lemma~\ref{l.vary}\ref{l.vary.4} to obtain an $ \vec{a}_* $ such that $ \qfn_{\jj}[\eta(\vec{a}_*)] = e^{\alpha_{\jj}} $ for $ \jj=1,\ldots,\mm $ and that $ |\vec{a}_*| \leq c(T,\w) v^{-4}\exp(-1/(cv^2)) $.
This gives the desired product of the second surgery $ \eta := \eta(\vec{a}_*) $.

We now use $ \eta $ to argue for the desired result.
By construction, this $ \eta $ attends all the target terminal values, namely $ \qfn_{\ii}[\eta] = e^{\alpha_\ii} $ for all $ \ii $. 
This and the fact that $ \w $ is a minimizer force $ \norm{\w}^2_2 \leq \norm{\eta}^2_2 $.
By the construction of $ \eta=\eta(\vec{a}_*) $, the quantity $ \norm{\eta}^2_2 $ is bounded by $ \norm{\theta}^2_2 + (\mm\cdot v^4\cdot 4v) \frac12 |\vec{a_*}|^2  \leq \norm{\theta}^2 + c(T,\w) \exp(-1/(cv^2)) $.
The right side, by \eqref{e.cost}, is further bounded by $ \norm{\w}^2_2  - \norm{\w_-}^2_{2;\actregion_{\ii}} + c(T,\w) \exp(-1/(cv^2)) $.
Altogether, we have $ \norm{\w}^2_2 \leq \norm{\w}^2_2  - \norm{\w_-}^2_{2;\actregion_{\ii}} + c(T,\w) \exp(-1/(cv^2)) $.
Simplifying this gives $ \norm{\w_-}^2_{2;\actregion_{\ii}} \leq c(T,\w) \exp(-1/(cv^2)) $, the desired result for Proposition~\ref{p.definite}.

\subsection{Bounding $ \partial_x^\ell\q $ and $ \partial_x^\ell\p $}
\label{s.NS.bounds}

We complete the proof of Theorem~\ref{t.NS} by establishing the following bounds.
Recall $ \rateic $ from \eqref{e.qic}.
For any $ \ell\in\Z_{\geq 0} $, $ \delta>0 $, and $ \beta>\rateic $, there exists $ c=c(\ell,\delta,\beta,\norm{\w}_2) $ such that, for all $ (t,x)\in[\delta,T-\delta]\times\R $,
\begin{align}
	\label{e.bd.q.dx}
	&|\big(\partial^\ell_{x} \q\big)(t,x)|
	\leq
	c\, \exp\big(\beta|x|\big),&
	&
	|\big(\partial^\ell_{x} \p\big)(t,x)|
	\leq
	c\, \sum_{\ii=1}^\mm \exp\big(-\tfrac{(\xi_\ii-x)^2}{2(1+\delta)(T-t)}\big).&
\end{align}
To simplify notation, we will write $ c=c(\ell,\delta,\beta,\norm{\w}_2) $.
The $ \ell=0 $ bound for $ \q $ follows from \eqref{e.qic} and \eqref{e.bd.q}. 
The $ \ell=0 $ bound for $ \p $ follows by combining \eqref{e.iter.p} and \eqref{e.bd.hkiter}, which gives
\begin{align}
	\label{e.bd.p}
	|\p(t,x)| \leq c \sum_{\ii=1}^{\mm} \hk(T-t,x-\xi_{\ii}),
	\qquad
	(t,x) \in [0,T)\times\R.
\end{align}
To proceed, we use induction.
Assume that the bounds have been established for derivatives up to the $ \ell $-th order.
In \eqref{e.inteq.q}, divide the time integral into over $ s\in[0,\delta/2] $ and $ s\in[\delta/2,t] $ and apply $ \partial_{x}^{\ell+1} $:
\begin{subequations}
\label{e.dxbd}
\begin{align}
	\label{e.dxbd1}
	\partial^{\ell+1}_x\q(t,x)
	=&
	\int_{\R} \d y \, \partial^{\ell+1}_x \hk(t,x-y) \qic(y)
\\
	\label{e.dxbd2}
	&+
	\int_{0}^{\delta/2} \d s \int_{\R} \d y \, \big(\partial^{\ell+1}_{x} \hk\big)(t-s,x-y) \cdot (\w\q)(s,y)
\\
	\label{e.dxbd3}
	&+
	\int_{\delta/2}^{t} \d s \int_{\R} \d y \, \big(\partial_{x} \hk\big)(t-s,x-y) \cdot \big(\partial^\ell_{y}(\p\q^2)\big)(s,y).
\end{align}
\end{subequations}
In \eqref{e.dxbd3}, we used $ \partial^{\ell+1}_{x} \hk(t-s,x-y) = \partial_x(-\partial_y)^\ell \hk(t-s,x-y) $ and integration by parts to transfer the derivatives to $ \p\q^2 $.
To proceed, we will be using the standard bound
\begin{align}
	\label{e.bd.hkdx}
	|\partial^{i}_{x} \hk(t,x-y)|
	\leq
	c(i,\delta) \, t^{-i/2} \hk((1+\delta)t,x-y),
	\qquad
	(t,x,y) \in (0,\infty)\times\R\times\R.
\end{align}
In \eqref{e.dxbd1}, use \eqref{e.qic} and \eqref{e.bd.hkdx} with $ t \geq \delta $ to bound the integral.
The result is bounded by $ c\exp(\beta|x|) $.
In \eqref{e.dxbd2}, apply the Cauchy--Schwarz inequality to bound the integral by $ \norm{\w}_2 \cdot (\int{}_{[0,\delta/2]\times\R} \d s \d y \, (\partial^{\ell+1}_x \hk(t-s,x-y)\q(s,y))^2)^{1/2} $.
Bound the last integral by using \eqref{e.qic}, \eqref{e.bd.q}, and \eqref{e.bd.hkdx} with $ t\mapsto (t-s) \geq \delta-\delta/2 $.
The result is bounded by $ c \exp(2\beta|x|) $.
Move onto \eqref{e.dxbd3}.
Expand $ \partial^\ell_{y}(\p\q^2) $ into a sum of terms of the form $ \partial^{i}_x \p \cdot \partial^{j} \q $, with $ i+j = \ell $.
Apply the induction hypothesis to bound these terms and use \eqref{e.bd.hkdx} with $ t \mapsto t-s $ to bound $ \partial_x \hk $.
The result is bounded by $ c\exp(\beta|x|) $.
This completes the induction for $ \q $.
The argument for $ \p $ is similar, with \eqref{e.bd.p} playing the role of \eqref{e.bd.q}.

\section{Solving the Nonlinear Shr\"{o}dinger equations: Proof of Theorem~\ref{t.formula}}
\label{s.solving}

\subsection{The forward scattering transform}
\label{s.solving.scatter}

We recall the relevant properties of the forward scattering transform of the NLS equations~\eqref{e.NS.q}--\eqref{e.NS.p}.

First, some notation.
Hereafter $ \lambda\in\C $ denotes the spectral parameter.
To alleviate heavy notation, we will often omit some of the dependence on $ \lambda $, $ t $, $ x $, for example $ \U = \U(\lambda;t,x) $ or $ \U(x) = \U(\lambda;t,x) $.
Let $ \sigma_3 := \mathrm{diag}(1,-1) $ denote the third Pauli matrix, whereby $ e^{a\sigma_3} = \mathrm{diag}(e^{a},e^{-a}) $.
For $ \mu\geq 0 $, we write $ O(|u|^\mu) $ for a generic quantity that is bounded by a constant multiple of $ |u|^\mu $ for small $ |u| $.

We now consider the forward scattering transform for $ t\in(0,T) $.
Fix $ \p $ and $ \q $ as in Theorem~\ref{t.NS} with $ \rateic = -\infty $.
Such $ \p $ and $ \q $ solve the NLS equations classically within $ (0,T)\times\R $ and are Schwartz in $ x $ for each fixed $ t $, namely being $ \Csp^\infty $ in $ x $ and having all $ x $ derivatives decaying super-polynomially as $ |x|\to\infty $.
Recall the Lax pair from~\eqref{e.lax}.
The \textbf{Jost solutions} are $ 2\times 2 $ matrices $ \jost^\pm =\jost^\pm(\lambda;t,x) $ that solve the \textbf{auxiliary linear problem} $ \partial_x \jost = \U \jost $ and are suitably normalized at $ x=\pm\infty $: 
\begin{align*}
	\partial_x \jost^\pm = \U \jost^\pm,
	\qquad
	\lim_{x\to\pm\infty} \jost^\pm(x) e^{\img\frac{\lambda}{2}x\sigma_3} = I.
\end{align*}
Since $ \trace(\U) = 0 $, by Louisville's formula $ \partial_x \det(\jost^\pm) = 0 $ and hence $ \det(\jost^\pm) \equiv 1 $.
Since $ \jost^+ $ and $ \jost^- $ are solutions of the same linear equation and since both are invertible matrices, we have that $ \jost^-(\lambda;t,x) = \jost^+(\lambda;t,x) \S(\lambda;t) $, for some $ x $-independent $ \S(\lambda;t) $.
This $ \S $ is the \textbf{scattering matrix}.
It evolves in time as
\begin{align*}
	\S(\lambda;t) = \Matrix{ \Sa(\lambda) & \Sbt(\lambda) e^{\lambda^2 t/2} \\ \Sb(\lambda) e^{-\lambda^2 t/2} & \Sat(\lambda) }.
\end{align*}
We refer to $ \Sa(\lambda) $, $\Sat(\lambda) $, $ \Sb(\lambda) $, and $ \Sbt(\lambda) $ as the \textbf{scattering coefficients}.
The process of going from $ \p $ and $ \q $ to the scattering coefficients is the forward scattering transform.

We next list a few useful properties related to the forward scattering transform.
\begin{enumerate}[leftmargin=20pt,label=(\Alph*)]
\item \label{property.S.entire}
The scattering coefficients $ \Sa(\lambda) $, $\Sat(\lambda) $, $ \Sb(\lambda) $, and $ \Sbt(\lambda) $ are entire, namely analytic on $ \{\lambda\in\C\} $.
\item \label{property.S.schwartz}
For any fixed $ t\in(0,T) $ and $ v<\infty $, the entries of $ \S(\lambda;t) - I $ are Schwartz within $ \{|\im(\lambda)|\leq v\} $, namely
$
	(\tfrac{\d}{\d \lambda})^m \{ \Sa(\lambda)-1,  \Sat(\lambda)-1, \Sb(\lambda), \Sb\til(\lambda) \}
	=
	O(1/|\lambda|^n)
$
on $ \{|\im(\lambda)|\leq v\} $, for any $ m,n\geq 0 $.

\item \label{property.S.aa-bb=1}
For all $ \lambda\in\C $, the identity $ \Sa(\lambda)\Sat(\lambda) - \Sb(\lambda) \Sbt(\lambda) = 1 $ holds.
\item \label{property.S.ato1}
For any fixed $ v<\infty $, $ \Sa(\lambda) \to 1 $ as $ |\lambda|\to\infty $ uniformly on $ \{ \im(\lambda) \geq -v\} $ and $ \Sat(\lambda) \to 1 $ as $ |\lambda|\to\infty $ uniformly on $ \{\im(\lambda) \leq v\} $.
\item \label{property.jost.entire}
For any fixed $ (t,x)\in(0,T)\times\R $, the Jost solutions are entire in $ \lambda $.
\item \label{property.jost}
Let $ \jost^{\pm,i} $ denote the $ i $-th column of the Jost solutions.
For any fixed $ (t,x)\in(0,T)\times\R $ and $ v<\infty $,
\begin{align*}
	\Matrix{ e^{\img\frac{\lambda}{2}x}\jost^{-,1}(\lambda) & e^{-\img\frac{\lambda}{2}x}\jost^{+,2}(\lambda) }  &= I + \frac{1}{\img\lambda} \Matrix{ \ldots & -\p \\ -\q & \ldots } + O\Big(\frac{1}{|\lambda|^2}\Big),
	\quad
	\text{on } \{\im(\lambda) \geq -v \}
\\
	\Matrix{ e^{\img\frac{\lambda}{2}x}\jost^{+,1}(\lambda) & e^{-\img\frac{\lambda}{2}x}\jost^{-,2}(\lambda) }  &= I + \frac{1}{\img\lambda} \Matrix{ \ldots & -\p \\ -\q & \ldots } + O\Big(\frac{1}{|\lambda|^2}\Big),
	\quad
	\text{on } \{\im(\lambda) \leq v \},
\end{align*}
where the $ (\ldots) $s denote finite quantities that may depend on $ (t,x) $.
\item \label{propert.conserved}
The NLS equations have a family of conserved --- namely time-independent --- quantities
\begin{align*}
	&
	\conserved_1 := \int_{\R} \d x \, (\p\q)(t,x),
	&&
	\conserved_2 := \int_{\R} \d x \, (\p \, \partial_x \q)(t,x),
	&&
	\conserved_3 := \int_{\R} \d x \, ( \p \, \partial_{xx} \q + \p^2 \q^2)(t,x),
	&&
	\ldots.
	&
\end{align*} 
Further, for all $ n\in\Z_{>0} $ and $ v<\infty $, $ \log\Sa(\lambda) = \sum_{k=1}^n \frac{\conserved_k}{(\img\lambda)^{k}} + O(1/|\lambda|^{n+1}) $ on $ \{|\im(\lambda)| \leq v\} $. 
\end{enumerate}

\begin{rmk}
Some of these properties are stronger than the standard ones in the literature.
For example, for an exponentially decaying initial condition, the standard result states that $ \Sa(\lambda) $ and $ \Sb(\lambda) $ are analytic on $ \{ \im(\lambda)>-c\} $ only for a $ c\in(0,\infty) $, which is weaker than \ref{property.S.entire}.
The stronger properties hold because of the initial-terminal condition considered here: With $ \rateic=-\infty $, the initial-terminal condition decays \emph{super-exponentially} in $ x $.
\end{rmk}

These properties arise from the analysis of the Jost solutions.
The analysis follows the standard one (see \cite[Chapters I--II]{faddeev07}, \cite[Chapter~15]{fokas08}, \cite[Chapter~3]{trogdon15} for example) with suitable adaptation.
We demonstrate the analysis for \ref{property.S.entire} and \ref{property.jost.entire}.
Fix $ t\in(0,T) $ and consider the auxiliary linear problem $ \partial_x \jost^\pm = \U \jost^\pm $.
Recall the definition of $ \U $ from \eqref{e.lax}, and decompose it into $ -\frac{\lambda}{2}\img\sigma_3 + \U_0 $, where $ (\U_0)_{11}=(\U_0)_{22} =0 $, $ (\U_0)_{21} = \q $, and $ (\U_0)_{22} = -\p $.
Rewrite  the auxiliary linear problem in the Duhamel form as
\begin{align}
	\label{e.jost.duhamel}
	\jost^\pm(x)
	&=
	e^{-\img\frac{\lambda}{2}x\sigma_3}
	+
	\int_{\pm\infty}^{x} \d y \, e^{-\img\frac{\lambda}{2}(x-y)\sigma_3}\,\big( \U_0 \jost^\pm\big)(y).
\end{align}
Iterating \eqref{e.jost.duhamel} gives a series representation of $ \jost^- $ (and similarly for $ \jost^+ $)
\begin{align}
	\label{e.jost.series}
	\jost^-(x)
	&=
	e^{-\img\frac{\lambda}{2}x\sigma_3}
	+
	\sum_{n=1}^\infty \int \prod_{i=1}^n \d y_i \, e^{-\img\frac{\lambda}{2}(y_{i-1}-y_{i})\sigma_3}\,\big( \U_0 \jost^-\big)(y_i) \cdot e^{-\img\frac{\lambda}{2}y_n\sigma_3},	
\end{align}
where the integral is over $ -\infty<y_n<\ldots<y_1<y_0:= x $.
Recall that the entries of $ \U_0 $ are given by $ \p $ and $ \q $. 
With $ \rateic=-\infty $, the bounds in~\eqref{e.bd.q.dx} assert that $ \U_0(y) $ decay super-exponentially in $ y $.
From this super-exponential decay, it is not hard to check that the right side of \eqref{e.jost.series} forms an absolutely convergent series of entire functions of $ \lambda $, and the convergence is uniform over bounded sets in $ \C \ni \lambda $.
Property~\ref{property.jost.entire} follows.
Move onto \ref{property.S.entire}. 
Recall that $ \jost^-(\lambda;t,x) = \jost^+(\lambda;t,x) \S(\lambda;t) $, right multiply both sides by $ e^{\img\frac{\lambda}{2}x\sigma_3} $, and send $ x\to\infty $ with the aid of $ e^{\img\frac{\lambda}{2}x\sigma_3}\jost^+(\lambda;t,x) \to I $.
This gives a representation $ \S(\lambda;t) = \lim_{x\to\infty}e^{\img\frac{\lambda}{2}x\sigma_3}\jost^-(\lambda;t,x) $ of the scattering matrix.
To utilize this representation, right multiply both sides of \eqref{e.jost.series} by $ e^{\img\frac{\lambda}{2}x\sigma_3} $ and send $ x\to\infty $.
From the super-exponential decay of $ \U_0 $, it is not hard to check that the result is an absolutely and uniformly (over bounded sets in $ \C $) convergent series of entire functions.
Property~\ref{property.S.entire} follows.

Next we turn to the forward scattering transform at $ t=0 $ and $ t=T $.
The major difference, compared to $ t\in(0,T) $, is that here we allow $ \q(0,\Cdot) $ and $ \p(T,\Cdot) $ to contain delta functions.
Recall that, for $ t\in(0,T) $, the Jost solutions satisfy~\eqref{e.jost.duhamel}.
Such an equation readily generalizes to $ t=0 $ and $ t=T $.
Recall that $ \qic = \sum_{\jj=1}^{\nn} e^{\beta_{\jj}} \delta_{\zeta_{\jj}} + \qicfn $.
For $ t=0 $, the Jost solutions $ \jost^\pm(x) = \jost^\pm(\lambda;0,x) $ are the piecewise continuous functions with jump discontinuity at each $ \zeta_{\jj} $ that satisfy
\begin{subequations}
\label{e.jost.t=0}
\begin{align}
	\jost^\pm(x)
	=
	e^{-\img\frac{\lambda}{2}x\sigma_3}
	&+
	\int_{\pm\infty}^{x} \d y \, e^{-\img\frac{\lambda}{2}(x-y)\sigma_3} \Matrix{ 0 & -\p(0,y) \\ \qicfn(y) & 0 } \jost^\pm(y)
\\
	&+
	\sum_{\jj=1}^{\nn} \ind_{I_\jj}(x) e^{-\img\frac{\lambda}{2}(x-\zeta_{\jj})\sigma_3} \Matrix{ 0 & 0 \\ e^{\beta_{\jj}} & 0 } \jost^\pm(\zeta_{\jj}^\pm),
	\qquad
	\text{when } t=0,
\end{align}
\end{subequations}
where $ I_{\jj} := [\zeta_{\jj},\infty) $ for $ \jost^- $ and $ I_{\jj} := (-\infty,\zeta_{\jj}] $ for $ \jost^+ $.
Similarly, for $ t=T $, the Jost solutions $ \jost^\pm(x) = \jost^\pm(\lambda;T,x) $ are the piecewise continuous functions with jump discontinuity at each $ \xi_{\ii} $ that satisfy
\begin{subequations}
\label{e.jost.t=T}
\begin{align}
	\jost^\pm(x)
	=
	e^{-\img\frac{\lambda}{2}x\sigma_3}
	&+
	\int_{\pm\infty}^{x} \d y \, e^{-\img\frac{\lambda}{2}(x-y)\sigma_3} \Matrix{ 0 & 0 \\ \q(T,y) & 0 } \jost^\pm(y)
\\
	&+
	\sum_{\ii=1}^{\mm} \ind_{I_\ii}(x)  e^{-\img\frac{\lambda}{2}(x-\xi_{\ii})\sigma_3} \Matrix{ 0 & \gamma_{\ii} \\ 0 & 0 } \jost^\pm(\xi_{\ii}^\pm),
	\qquad
	\text{when } t=T,	
\end{align}
\end{subequations}
where $ I_{\ii} := [\xi_{\ii},\infty) $ for $ \jost^- $ and $ I_{\ii} := (-\infty,\xi_{\ii}] $ for $ \jost^+ $.

We next discuss how to extract the scattering coefficients from the Jost solutions at $ t=0 $ and $ t=T $.
For $ t\in (0,T) $, we have $ \jost^-(\lambda;t,x) = \jost^+(\lambda;t,x) \S(\lambda;t) $.
From \eqref{e.jost.duhamel}--\eqref{e.jost.t=0}, it is not difficult (though tedious) to show that $ \jost^{\pm}(\lambda;t,x) \to \jost^{\pm}(\lambda;0,x) $ as $ t\to 0 $, for any fixed $ \lambda\in\C $ and any fixed $ x \neq \xi_{1},\ldots,\xi_{\mm} $. 
Hence $ \jost^-(\lambda;0,x) = \jost^+(\lambda;0,x) \S(\lambda;0) $ except at $ x=\xi_{1},\ldots,\xi_{\mm} $.
Right multiplying both sides by $ (\jost^{+}(\lambda;0,x))^{-1} $ and sending $ x\to\infty $ gives
\begin{align*}
	\S(\lambda;0) 
	= 
	\Matrix{ \Sa(\lambda) & \Sbt(\lambda)  \\ \Sb(\lambda) & \Sat(\lambda) }
	=
	\lim_{x\to\infty}  e^{\img\frac{\lambda}{2}x\sigma_3} \jost^-(\lambda;0,x).
\end{align*}
Similar properties hold at $ t=T $. In particular,
\begin{align*}
	\S(\lambda;T) 
	= 
	\Matrix{ \Sa(\lambda) & \Sbt(\lambda)e^{\lambda^2T/2}  \\ \Sb(\lambda)e^{-\lambda^2T/2}  & \Sat(\lambda) }
	=
	\lim_{x\to\infty} e^{\img\frac{\lambda}{2}x\sigma_3} \jost^-(\lambda;T,x) .
\end{align*}

\subsection{The Riemann--Hilbert problem}
\label{s.solving.RH}
We will formulate the Riemann--Hilbert problem that performs the inverse scattering transform: the process of recovering $ \p $ and $ \q $ from the scattering coefficients.

The first step is standard.
For each fixed $ (t,x)\in(0,T)\times\R $, define the $ 2\times 2 $ matrices
\begin{align*}
	\X^{\up}(\lambda) 
	:= 
	\Matrix{ 
		\frac{e^{\img\frac{\lambda}{2}x}}{\Sa(\lambda)} \jost^{-,1}(\lambda;t,x) 
	&
		e^{-\img\frac{\lambda}{2}x} \jost^{+,2}(\lambda;t,x) 
	},
	&&
	\X^{\lw}(\lambda) 
	:= 
	\Matrix{ 
		e^{\img\frac{\lambda}{2}x} \jost^{+,1}(\lambda;t,x) 
	&
		\frac{e^{-\img\frac{\lambda}{2}x}}{\Sat(\lambda)} \jost^{-,2}(\lambda;t,x) 
	}.	
\end{align*} 
By Properties~\ref{property.S.ato1} and \ref{property.jost}, we have $ \X^{\up}(\lambda) \to I $ and $ \X^{\lw}(\lambda) \to I $ as $ |\lambda|\to\infty $ on $ \{ +\im(\lambda) \geq 0\} $ and $ \{ \im(\lambda) \leq 0\} $ respectively.
Define the \textbf{reflection coefficients} $ \Sr(\lambda) := \Sb(\lambda) / \Sa(\lambda) $ and $ \Srt(\lambda) := \Sbt(\lambda) / \Sat(\lambda) $,
and the jump matrix
\begin{align*}
	\G(\lambda)
	:=
	\Matrix{ 1 - \Sr(\lambda) \Srt(\lambda) & -\Srt(\lambda)e^{\lambda^2t/2-\img\lambda x}  \\ \Sr(\lambda) e^{- \lambda^2t/2+\img\lambda x } & 1 }.
\end{align*} 
The relation $ \jost^- = \jost^+ \S $ translates into $ \X^{\up}(\lambda) = \X^{\lw}(\lambda) \G(\lambda) $, which holds everywhere on $ \C $ except at the zeros of $ \Sa(\lambda) $ and $ \Sat(\lambda) $, where $ \X^\up(\lambda) $ or $ \X^\lw(\lambda) $ has a pole.

\begin{rmk}
The scripts `up' and `lw' refer to `upper' and `lower', which differ from the conventional notation that uses $ + $ and $ - $.
We do this in order to reserve $ + $ and $ - $ for the operators $ \oindp $ and $ \oindm $ later.
\end{rmk}

The standard formulation of the Riemann--Hilbert problem takes $ \R $ as the contour, but we will do differently and take two contours $ \R+\img v_0 $ and $ \R-\img v_0 $.
Doing so has the advantage of avoiding any pole caused by $ \Sa(\lambda) $ and $ \Sat(\lambda) $.
First, note that by Properties~\ref{property.S.entire} and \ref{property.S.ato1}, the functions $ \Sa(\lambda) $ and $ \Sat(\lambda) $ have finitely many zeros in the upper and lower half planes respectively.
For $ v_0 \in (0,\infty) $, let $ \region_{\up} := \{\im(\lambda)>v_0\} $, $ \region_{\md} := \{ |\im(\lambda)| < v_0 \} $, and  $ \region_{\lw} := \{\im(\lambda)<v_0\} $ denote the regions separated by the contours $ \R\pm \img v_0 $.
We assume $ v_0 $ is large enough so that $ \Sa(\lambda) $ and $ \Sat(\lambda) $ have no zeros on $ \bar{\region_{\up}} $ and $ \bar{\region_{\lw}} $ respectively.
This way, $ \X^{\up/\lw}(\lambda) $ is analytic on $ \region_{\up/\lw} $.
Along the upper contour $ \R + \img v_0 $, define the jump matrix
\begin{align}
	\label{e.Gup}
	\G^{\up}(\lambda)
	:=
	\Matrix{ 1/\Sa(\lambda) & 0 \\ 0 & 1 },
	\qquad
	\lambda \in \R + \img v_0,
\end{align} 
and on $ \bar{\region_{\md}} $ consider
$
	\X^\md (\lambda)
	:= 
	(
		e^{\img\frac{\lambda}{2}x} \jost^{-,1}(\lambda;t,x) 
	\ \
		e^{-\img\frac{\lambda}{2}x} \jost^{+,2}(\lambda;t,x) 
	)
$.
This $ \X^\md $ is analytic on $ \region_{\md} $ (and in fact on $ \C $) and satisfies the jump condition $ \X^{\up}(\lambda) = \X^\md(\lambda) \G^{\up}(\lambda) $ along the upper contour $ \R + \img v_0 $.
Along the lower contour $ \R - \img v_0 $, the jump condition $ \X^{\md}(\lambda) = \X^{\lw}(\lambda) \G^{\lw}(\lambda) $ holds for
\begin{align}
	\label{e.Glw}
	\G^{\lw}(\lambda)
	:=
	\Matrix{ \Sa(\lambda) - \Sb(\lambda)\Srt(\lambda) & -\Srt(\lambda)e^{\lambda^2t/2-\img\lambda x } \\ \Sa(\lambda) \Sr(\lambda) e^{- \lambda^2t/2+\img\lambda x }  & 1 },
	\qquad
	\lambda \in \R - \img v_0.
\end{align} 

From the above discussion and from Property~\ref{property.jost}, we see that the following holds.
\begin{enumerate}[leftmargin=40pt, label=(RH \roman*)]
\item \label{RH.analytic}
$ \X^{\up/\md/\lw}(\lambda) $ are analytic on $ \region_{\up/\md/\lw} $ and extend continuously onto $ \bar{\region_{\up/\md/\lw}} $.
\item \label{RH.jump.+}
$
	\X^{\up}(\lambda) 
	= \X^\md(\lambda) \G^{\up}(\lambda)
$
for
$ \lambda \in \R + \img v_0 $.
\item \label{RH.jump.-}
$
	\X^\md(\lambda) 
	= \X^{\lw}(\lambda) \G^{\lw}(\lambda)
$
for
$ \lambda \in \R - \img v_0 $.
\item \label{RH.jump.asymptotics}
$ \displaystyle \X^{\up/\md/\lw}(\lambda) = I + \frac{1}{\img\lambda} \Matrix{ \X^1_{11} & -\p \\ -\q & \X^1_{22} } + O(|\lambda|^{-2}) $ on $ \bar{\region_{\up/\md/\lw}} $, for some $ \X^{1}_{11},\X^{1}_{22}\in\C $.
\end{enumerate}

\subsection{The Fourier transforms and the integral equation}
\label{s.solving.fourier}

The next step is to transform the Riemann--Hilbert problem into an integral equation.
We will do so by using the following Fourier transforms.
Let
\begin{align*}
	\fourier_{\up}\big[ f \big](s)
	:=
	\int_{\R+\img v_0} \frac{\d \lambda}{2\pi} \, e^{\img s\lambda } f(\lambda),
	&&
	\fourier_{\lw}\big[ f \big](s)
	:=
	\int_{\R-\img v_0} \frac{\d \lambda}{2\pi} \, e^{\img s\lambda } f(\lambda),
	&&
	s\in\R
\end{align*}
denote the Fourier transforms along the upper and lower contours.

We now apply the Fourier transforms.
In the jump conditions \ref{RH.jump.+}--\ref{RH.jump.-}, write the $ \X $s as $ \X = (\X - I) + I $ and the $ \G $s as $ \G = (\G - I) + I $ and simplify the result.
Doing so gives
\begin{align}
	\label{e.jump+}
	&&(\X^{\up} - I)
	&=
	(\X^\md - I)(\G^{\up}-I) + (\G^{\up}-I) + (\X^\md-I)&
	&
	\text{along } \R +\img v_0,&
\\
	\label{e.jump-}
	&&(\X^\md - I)
	&=
	(\X^{\lw} - I)(\G^{\lw}-I) + (\G^{\lw}-I) + (\X^{\lw}-I)&
	&
	\text{along } \R -\img v_0.&
\end{align}
Define (the transpose of) the Fourier transforms of the terms in \eqref{e.jump+}--\eqref{e.jump-} as
\begin{align*}
	\SGamma^{\up} := (\fourier_{\up}[\G^{\up}-I])^\transp,
	&&
	\SGamma^{\lw} := (\fourier_{\lw}[\G^{\lw}-I])^\transp,
	&&
	(\SXi^{\up}) := (\fourier_{\up}[\X^{\up}-I])^\transp,
\\
	(\SXi^{\lw}) := (\fourier_{\lw}[\X^{\lw}-I])^\transp,
	&&
	(\SXi^{\md}_{\up}) := (\fourier_\up[\X^\md-I])^\transp,
	&&
	(\SXi^{\md}_{\lw}) := (\fourier_\lw[\X^\md-I])^\transp,	
\end{align*}
where $ \fourier_{\up} $ and $ \fourier_{\lw} $ apply to matrices entry-by-entry, and we take the transpose $ (\ldots)^\transp $ to streamline our subsequent notation.
Thanks to Property~\ref{property.S.schwartz}, the entries of $ (\G^{\up}-I) $ and $ (\G^{\lw}-I) $ are Schwartz, hence $ \SGamma^{\up} $ and $ \SGamma^{\lw} $ are well-defined and Schwartz.
As for the $ (\SXi) $s, recall that the $ \X $s are continuous along the relevant contours and satisfy the asymptotics given in \ref{RH.jump.asymptotics}.
We interpret $ \int_{\R\pm\img u_0} \frac{\d\lambda}{2\pi}\, e^{\img s\lambda} \frac{1}{\img\lambda} $ in the Cauchy sense, so that the $ (\SXi) $s are well-defined.
Since $ \X^{\up/\md/\lw}(\lambda) $ is analytic on $ \region_{\up/\md/\lw} $, we have $ (\SXi^{\up})(s)|_{s>0} \equiv 0 $, $ (\SXi^{\md}_{\up}) = (\SXi^{\md}_{\lw}) $, and $ (\SXi^{\lw})(s)|_{s<0} \equiv 0 $.
We henceforward write both $ (\SXi^{\md}_{\up}) $ and $ (\SXi^{\md}_{\lw}) $ as $ (\SXi^{\md}) $. 
Further, using \ref{RH.jump.asymptotics} to analyze $ (\SXi^{\lw}) := \fourier_{\lw}[\X^{\lw}-I]^\transp $ yields that 
\begin{align}
	\label{e.sXi.asymptotics}
	(\SXi^{\lw}(s))|_{s>0} = \Matrix{ \X^1_{11} & -\q \\ -\p & \X^2_{22} } + O(|s|)
\end{align}
and that $ (\SXi^{\lw})(s) $ is continuous on $ (0,\infty) $.
Equipped with these properties, we apply $ \fourier_{\up} $ and $ \fourier_{\lw} $ to \eqref{e.jump+}--\eqref{e.jump-} respectively and take the transpose to get
\begin{align}
	\label{e.RH.inteq.1}
	&&0 &=\int_{\R} \d s'\, \SGamma^{\up}(s-s') (\SXi^{\md})(s') + (\SXi^{\md})(s) + \SGamma^{\up}(s),
	&
	s>0,&&&
\\
	\label{e.RH.inteq.2}
	&&(\SXi^{\md})(s) &=\int_{\R} \d s'\, \SGamma^{\lw}(s-s') (\SXi^{\lw})(s') + (\SXi^{\lw})(s) + \SGamma^{\lw}(s),
	&
	s\in\R.&&&
\end{align}

The next step is to eliminate $ (\SXi^{\md})(s) $ in \eqref{e.RH.inteq.1}--\eqref{e.RH.inteq.2} by combining the two equations.
From this point onward, it is more convenient to use operator language.

Let us set up the operator language.
Given an $ f=f(s)\in\Lsp^2(\R) $, let $ \mathbf{f} $ denote the bounded operator that acts on $ \Lsp^2(\R) $ by $ (\mathbf{f} \phi)(s) := \int_{\R} \d s' \, f(s-s') \phi(s') $.
Recall the definition of `having an almost continuous kernel' from before Theorem~\ref{t.formula}, and recall the notation $ {}_{0}[\mathbf{f}]_{0} $ from there. 
Similarly define $ \mathbf{f}]_0 := f(\Cdot-0^\pm) = f(\Cdot) \in \Lsp^2(\R) $.
For a matrix $ A=A(s) $ with entries $ A_{ij}(s)\in\Lsp^2(\R) $, the preceding notation generalizes entry-by-entry.
The entries of $ \SGamma^{\up} $ and $ \SGamma^{\lw} $ are in $ \Lsp^2(\R) $ because they are Schwartz; the entries of $ (\SXi^{\up}) $, $ (\SXi^{\md}) $, and $ (\SXi^{\lw}) $ are in $ \Lsp^2(\R) $ because the entries of $ (\X^{\up}-I) $, $ (\X^\md-I) $, and $ (\X^{\lw}-I) $ are in $ \Lsp^2(\R) $.
Let $ \oGamma^{\up} $, $ \oGamma^{\lw} $, $ \oXi^{\up} $, $ \oXi^\md $, and $ \oXi^{\lw} $ denote the corresponding operator-valued matrices.
Note that they have almost continuous kernels.
Let $ \oid $ denote the identity operator on $ \Lsp^2(\R) $, set $ \oI := \mathrm{diag}(\oid,\oid) $, let $ \oindpm $ acts on $ \Lsp^2(\R) $ by $ (\oindpm \phi)(s) := \ind_{\{\pm s>0\}} \phi(s) $, and set $ \oIpm :=  \mathrm{diag}(\oindpm,\oindpm) $.

We return to the analysis of \eqref{e.RH.inteq.1}--\eqref{e.RH.inteq.2}.
In operator language, they read
\begin{align*}
	&&\ozero = \oIp ( \oGamma^{\up} +\oI )\,\oXi^\md ]_{0} + \oIp \oGamma^{\up} ]_{0},&
	&
	\oXi^\md = (\oGamma^{\lw} +\oI)\,\oXi^{\lw} + \oGamma^{\lw}.&&
\end{align*}
Insert the second equation into the first, simplify the result, and use $ \oXi^{\lw}]_0 = (\SXi^{\lw}) $.
Doing so gives
\begin{align*}
	(\oIp +\oIp\oGamma^{\up}+\oIp\oGamma^{\lw}+\oIp\oGamma^{\up}\oGamma^{\lw}) (\SXi^{\lw})
	= 
	\big( -\oIp\oGamma^{\up}-\oIp\oGamma^{\lw}-\oIp\oGamma^{\up}\oGamma^{\lw} \big) \big]_0.
\end{align*}
Let us further simplify this equation.
On the left side, since $ (\SXi^{\lw})(s)|_{s<0} \equiv 0 $, we have $ (\SXi^{\lw}) = \oIp (\SXi^{\lw}) $.
We hence put an $ \oIp $ just in front of $ (\SXi^{\lw}) $ and replace the leftmost $ \oIp $ with $ \oI $.
As for the right side, since $ \oGamma^{\up} $ and $ \oGamma^{\lw} $ have continuous kernels, we right multiply every term on the right side with $ \oIp $ without changing the result.
Set $ \oGamma := \oGamma^{\up}+\oGamma^{\lw}+\oGamma^{\up}\oGamma^{\lw} $.
We arrive at the desired integral equation
\begin{align}
	\label{e.RH.inteq}
	(\oI + \oIp\oGamma\oIp) (\SXi^{\lw})
	= 
	\big( \oI-(\oI + \oIp\oGamma\oIp) \big) \big]_0.
\end{align}

\subsection{Proof of Theorem~\ref{t.formula}}
\label{s.solving.proof}
Here we solve the integral equation~\eqref{e.RH.inteq} and thereby prove Theorem~\ref{t.formula}.
Indeed, given the asymptotics of $ (\SXi^{\lw}) $ in \eqref{e.sXi.asymptotics}, once we can solve the equation for $ (\SXi^{\lw}) $, sending $ s\to 0^+ $ in the result gives $ \p $ and $ \q $.

We begin by deriving an expression for the operator $ \oGamma := \oGamma^{\up}+\oGamma^{\lw}+\oGamma^{\up}\oGamma^{\lw} $. 
Recall that the kernels $ \SGamma^{\up} $ and $ \SGamma^{\lw} $ of $ \oGamma^{\up} $ and $\oGamma^{\lw} $ are the Fourier transforms of $ (\G^{\up} -I) $ and $ (\G^{\lw} -I) $.
Refer to the expressions \eqref{e.Gup}--\eqref{e.Glw} of $ \G^{\up} $ and $ \G^{\lw} $, subtract $ I $ from both sides, and apply $ \fourier_{\up} $ and $ \fourier_{\lw} $ respectively.
Doing so gives
\begin{align*}
	&&\SGamma^{\up} 
	=
	\Matrix{ \fourier_{\up}[\frac{1}{\Sa}-1] & 0 \\ 0 & 0 },&
	&
	\SGamma^{\lw} 
	=
	\Matrix{ \fourier_{\lw}[\Sa-1]-\fourier_{\lw}[\Sb\dressf]*\fourier_{\lw}[\,\Srt/g\,] & \fourier_{\lw}[\Sa\Sr\dressf] \\ -\fourier_{\lw}[\,\Srt/\dressf] & 0 },&
\end{align*}
where $ * $ denotes the convolution on $ \R $ and $ \dressf(\lambda) = \dressf(\lambda;t,x) := e^{- \lambda^2t/2+\img\lambda x } $.
Recall that $ \Sa(\lambda) $ and $ \Sb(\lambda) $ are entire and note that $ (\Sa\Sr\dressf)(\lambda) = \Sb(\lambda) e^{-\lambda^2 t/2+\img\lambda x} $ is also entire.
Hence, in the first row of $ \SGamma^{\lw} $, the first, second, and last $ \fourier_{\lw} $ can be replaced by $ \fourier_{\up} $.
Using this property to calculate the kernel of $ \oGamma $ and simplifying the result give
\begin{align}
	\label{e.SGamma.}
	(\text{the kernel of } \oGamma)
	=
	\SGamma^{\up} + \SGamma^{\lw} + \SGamma^{\up}*\SGamma^{\lw}
	=
	\Matrix{  
		- \fourier_{\up}[\Sr\dressf] * \fourier_{\lw}[\,\Srt/\dressf\,] & \fourier_{\up}[\Sr\dressf]
		\\
		- \fourier_\lw[\,\Srt/\dressf] & 0
	}.
\end{align}
Now define 
\begin{align}
	\label{e.Srho}
	\Srho(s) = \Srho(s;t,x)
	&:=
	\fourier_{\up}[\Sr\dressf] 
	=
	\int_{\R+\img v_0} \frac{\d \lambda}{2\pi} \, e^{\img s \lambda} \Sr(\lambda) e^{- \lambda^2t/2+\img\lambda x},
\\
	\label{e.Srhot}
	\Srhot(s) = \Srhot(s;t,x)
	&:=
	\fourier_{\lw}[\,\Srt/\dressf\,] 
	=
	\int_{\R-\img v_0} \frac{\d \lambda}{2\pi} \, e^{\img s \lambda} \Srt(\lambda) e^{\lambda^2t/2-\img\lambda x },
\end{align}
and let $ \orho = \orho_{t,x} $ and $ \orhot = \orhot_{t,x} $ denote the corresponding operators on $ \Lsp^2(\R) $.
Note that, thanks to Property~\ref{property.S.schwartz}, the functions $ \Sr\dressf $ and $ \Srt/\dressf $ are Schwartz on $ \R\pm\img v_0 $, so $ \Srho $ and $ \Srhot $ are Schwartz in $ s $ and $ \orho $ and $ \orhot $ are bounded operators on $ \Lsp^2(\R) $.
In operator language, \eqref{e.SGamma.} reads
\begin{align}
	\label{e.SGamma}
	\oGamma
	=
	\Matrix{  
		- \orho \, \orhot & \orho
		\\
		- \orhot & \mathbf{0}
	}.
\end{align}

Next, we investigate the invertibility of $ (\oI + \oIp \oGamma \oIp) $, which is key to solving \eqref{e.RH.inteq}.
As it turns out, the invertibility follows if $ \det( \oid - \oindp \orho \oindm \orhot \oindp ) \neq 0 $.
This can be heuristically understood by evaluating the determinant of $ (\oI + \oIp \oGamma \oIp) $ from \eqref{e.SGamma}.
To \emph{prove} the statement, we begin by verifying that $ \oindp\orho\oindm\orhot\oindp $ is trace-class.
To alleviate heavy notation, we will often write 
\begin{align*}
	{}_\smallpm\orho{}_\smallpm
	:=
	\oindpm\orho\oindpm,
	\
	\text{etc.}
\end{align*}
Let $ \norm{\ }_\tracen $ denote the trace norm.
By the Cauchy--Schwarz inequality,
\begin{align}
	\label{e.traceclass}
	\norm{ {}_\smallp \orho {}_\smallm \orhot {}_\smallp }_\tracen
	\leq
	\Big( \int_{(0,\infty)\times(-\infty,0)} \hspace{-30pt} \d s \d s'\, \big| \Srho(s-s';t,x) \big|^2 \Big)^{1/2}
	\Big( \int_{(-\infty,0)\times(0,\infty)} \hspace{-30pt} \d s \d s'\, \big| \Srhot(s-s';t,x) \big|^2 \Big)^{1/2}.
\end{align}
The right side is finite because $ \Srho $ and $ \Srhot $ are Schwartz in $ s $.
Having verified the trace-class property, we see that $ \det( \oid - {}_\smallp \orho {}_\smallm \orhot {}_\smallp ) \neq 0 $ implies that $ ( \oid - {}_\smallp \orho {}_\smallm \orhot {}_\smallp )^{-1} $ and $ (\oid-{}_\smallm\orhot{}_\smallp\orho{}_\smallm)^{-1} $ exist and are bounded on $ \Lsp^2(\R) $.
Once the last two inverses exist, the inverse of $ (\oI + \oIp \oGamma \oIp) $ is given by
\begin{align}
\label{e.inverse}
	(\oI + \oIp \oGamma \oIp)^{-1}
	&=
	\Matrix{
		( \oid - {}_\smallp \orho {}_\smallm \orhot {}_\smallp )^{-1} & -( \oid - {}_\smallp \orho {}_\smallm \orhot {}_\smallp )^{-1} {}_\smallp \orho {}_\smallp
		\\
		{}_\smallp \orhot {}_\smallp ( \oid - {}_\smallp \orho {}_\smallm \orhot {}_\smallp )^{-1} & \oid - {}_\smallp \orhot {}_\smallp ( \oid - {}_\smallp \orho {}_\smallm \orhot {}_\smallp )^{-1} {}_\smallp \orho {}_\smallp
	}.
\end{align}
This follows by straightforward calculations from \eqref{e.SGamma}, with the aid of the readily-verified identities
$
	(\oid-{}_\smallp\orho{}_\smallm\orhot{}_\smallp)^{-1} (\oid-{}_\smallp\orho\orhot{}_\smallp) = \oid - (\oid-{}_\smallp\orho{}_\smallm\orhot{}_\smallp)^{-1} {}_\smallp\orho{}_\smallp\orhot{}_\smallp,
$
and
$
	(\oid-{}_\smallp\orho\orhot{}_\smallp) (\oid-{}_\smallp\orho{}_\smallm\orhot{}_\smallp)^{-1} = \oid - {}_\smallp\orho{}_\smallp\orhot{}_\smallp (\oid-{}_\smallp\orho{}_\smallm\orhot{}_\smallp)^{-1}.
$

Let us finish the proof of Theorem~\ref{t.formula} for every $ (t,x)\in(0,T)\times\R $ such that
\begin{align}
	\tag{nondegenerate det}
	\label{e.detassump}
	\det( \oid - {}_\smallp \orho {}_\smallm \orhot {}_\smallp ) 
	=
	\det( \oid - \oindp \orho_{t,x} \oindm \orhot_{t,x} \oindp )
	\notin (-\infty,0].
\end{align}
Note that, at this stage we do not know whether the determinant is real.
This condition ensures that $ \det( \oid - {}_\smallp \orho {}_\smallm \orhot {}_\smallp ) \neq 0 $ and that its (potentially complex) logarithm is single-valued.
Fix any $ (t,x)\in(0,T)\times\R $ such that \eqref{e.detassump} holds.
Apply the inverse in \eqref{e.inverse} to the integral equation~\eqref{e.RH.inteq} to get
\begin{align}
	\label{e.sXi.solution}	
	(\SXi^{\lw})
	&=
	\Matrix{
		( \oid - {}_\smallp \orho {}_\smallm \orhot {}_\smallp )^{-1} - \oid & -( \oid - {}_\smallp \orho {}_\smallm \orhot {}_\smallp )^{-1} {}_\smallp \orho {}_\smallp
		\\
		{}_\smallp \orhot {}_\smallp ( \oid - {}_\smallp \orho {}_\smallm \orhot {}_\smallp )^{-1} & - {}_\smallp \orho {}_\smallm ( \oid - {}_\smallp \orho {}_\smallm \orhot {}_\smallp )^{-1} {}_\smallp \orho {}_\smallm
	}
	\Bigg]_0.
\end{align}
With the aid of the readily-verified identity $ (\oid - \mathbf{u}\mathbf{v})^{-1}  = \oid -\mathbf{u} (\oid-\mathbf{v}\mathbf{u})^{-1} \mathbf{v} $, it is not hard to check that every entry in \eqref{e.sXi.solution} has an almost continuous kernel.
Take the $ (1,1) $ entry for example.
Let $ \ip{\Cdot,\Cdot} $ denote the inner product on $ \Lsp^2(\R) $. 
The $ (1,1) $ entry has the kernel
\begin{align*}
	&&
	\ind_{\{s>0\}} g(s,s') \ind_{\{s'>0\}}, &
	&
	g(s,s') := \IP{ \Srho(s-\Cdot) \, , \, \big(- \oindm ( \oid - {}_\smallm \orhot {}_\smallp \orho {}_\smallm )^{-1} \oindm \big) \,\Srhot(\Cdot-s') }&
\end{align*}
and $ g(s,s') $ is continuous because $ (- \oindm (\oid - {}_\smallm \orhot {}_\smallp \orho {}_\smallm )^{-1} \oindm) $ is bounded and $ \Srho(s) $ and $ \Srhot(s') $ are Schwartz.
This shows that the $ (1,1) $ entry has an almost continuous kernel.
Combining~\eqref{e.sXi.solution} with \eqref{e.sXi.asymptotics} gives the first equality in \eqref{e.t.formula.p}--\eqref{e.t.formula.q}.
To obtain the second equality, use the readily-verified identities
$
	(\oid - \mathbf{u}\mathbf{v})^{-1} \mathbf{u}
	=
	\mathbf{u} (\oid - \mathbf{v}\mathbf{u})^{-1}
$
and
$
	\mathbf{v} (\oid - \mathbf{u}\mathbf{v})^{-1} 
	=
	(\oid - \mathbf{v}\mathbf{u})^{-1} \mathbf{v}.
$

Under \eqref{e.detassump}, it remains to prove \eqref{e.t.formula.det} in Theorem~\ref{t.formula}.
Restore the dependence on $ (t,x) $ and write $ \orho_{t,x} $ and $ \orhot_{t,x} $.
The proof will involve operator calculus, so we begin by establishing the smoothness of $ \oindp\orho_{t,x}\oindm $ and $ \oindm\orhot_{t,x}\oindp $ in $ x $.
First, by \eqref{e.Srho}, we have $ \Srho(s;t,x) = \Srho(s+x;t,0) $ and $ \Srhot(s;t,x) = \Srhot(s-x;t,0) $.
Given that $ \Srho $ and $ \Srhot $ are Schwartz in $ s $, they are also Schwartz in $ x $.
Let $ \norm{\ }_\mathrm{HS} $ denote the Hilbert--Schmidt norm.
Writing $ \norm{\oindp(\orho_{t,x}-\orho_{t,x'})\oindm}^2_\mathrm{HS} = \int_0^\infty \d s \int_{-\infty}^0 \d s' |\Srho(s-s';t,x)-\Srho(s-s';t,x')|^2 $ shows that $ \oindp\orho_{t,x}\oindm $ is $ \Csp^\infty $ in $ x $ with respect to the Hilbert--Schmidt norm; the same property holds for $ \oindm\orhot_{t,x}\oindp $.
We now follow the calculations in \cite[Equations~(S50)--(S53)]{krajenbrink21} to prove \eqref{e.t.formula.det}.
First, the operator $ \oindp \orho_{t,x} \oindm \orhot_{t,x} \oindp  $ has the kernel 
$ 
	\int_0^\infty \d s'' \ind_{\{s>0\}} \Srho(s-s''+x;t,0) \Srhot(s''-s'-x;t,0) \ind_{\{s'>0\}}.
$
Perform the change of variables $ s''+x \mapsto s''' $ to transfer the $ x $ dependence from $ \Srho,\Srhot $ to the range of the integral.
Differentiating the result in $ x $ gives $ \ind_{\{s>0\}} \Srho(s+x;t,0) \Srhot(-x-s';t,0) \ind_{\{s'>0\}}  = \ind_{\{s>0\}} \Srho(s-0;t,x) \Srhot(0-s';t,x) \ind_{\{s'>0\}}  $.
In operator language, this reads
\begin{align}
	\label{e.opdiff1}
	\partial_x \big( \oindp \orho_{t,x} \oindm \orhot_{t,x} \oindp  \big)
	=
	 \oindp \orho_{t,x} \big]_{0} \, \prescript{}{0}{\big[} \orhot_{t,x} \oindp .
\end{align}
Next, apply the identity $ \partial_x \log(1+\mathbf{u}) = \trace((\partial_x \mathbf{u})(1+\mathbf{u})^{-1}) $ with $ \mathbf{u} = - \oindp \orho_{t,x} \oindm \orhot_{t,x} \oindp $.
The identity requires $ \mathbf{u} $ to be differentiable in $ x $ with respect to the trace norm, which holds in our application since $ \orho_{t,x} $ and $ \orhot_{t,x} $ are $ \Csp^\infty $ in $ x $ with respect to the Hilbert--Schmidt norm.
Applying the identity and using \eqref{e.opdiff1} give
\begin{align*}
	\partial_{x} &\log \det( \oid - \oindp \orho_{t,x} \oindm \orhot_{t,x} \oindp ) 
	=
	-\trace\big( \ \oindp \orho_{t,x} \big]_{0} \, \prescript{}{0}{\big[} \orhot_{t,x} \oindp (\oid - \oindp \orho_{t,x} \oindm \orhot_{t,x} )^{-1} \ \big).
\end{align*}
The operator within the last trace is rank-one, more precisely of the form $ f\otimes \ip{ g, \Cdot } $ with $ f(s) = {}_s[ \oindp \orho_{t,x} ]_{0} $ and $ g(s') = {}_0[\orhot_{t,x} \oindp (\oid - \oindp \orho_{t,x} \oindm \orhot_{t,x})^{-1}]_{s'} $.
Such an operator has trace $ \ip{g,f} $.
Hence
\begin{align*}
	\partial_{x} \log \det( \oid - \oindp \orho_{t,x} \oindm \orhot_{t,x} \oindp ) 
	&=
	-\prescript{}{0}{\big[} \orhot_{t,x} \oindp \big(\oid - \oindp \orho_{t,x} \oindm \orhot_{t,x} \oindp \big)^{-1} \oindp \orho_{t,x} \big]_{0}.
\end{align*}
Use the identity $ -\mathbf{u} (\oid-\mathbf{v}\mathbf{u})^{-1} \mathbf{v}  = \oid - (\oid - \mathbf{u}\mathbf{v})^{-1} $ to rewrite the last expression; differentiate the result in $ x $ with the aid of the identity $ \partial_x (1+\mathbf{u})^{-1} = (1+\mathbf{u})^{-1} \cdot\partial_x \mathbf{u}\cdot (1+\mathbf{u})^{-1} $ and use \eqref{e.opdiff1}.
We arrive at
\begin{align*}
	\partial_{xx} &\log \det( \oid - \oindp \orho_{t,x} \oindm \orhot_{t,x} \oindp ) 
\\
	&=
	\prescript{}{0}{\big[} \big(\oid - \oindp \orho_{t,x} \oindm \orhot_{t,x} \oindp \big)^{-1} \oindp \orho_{t,x} \big]_{0} 
	\, 
	\prescript{}{0}{\big[} \orhot_{t,x} \oindp 
	\big(\oid - \oindp \orho_{t,x} \oindm \orhot_{t,x} \oindp \big)^{-1} \big]_{0}.
\end{align*}
By \eqref{e.t.formula.p}--\eqref{e.t.formula.q}, the last expression is $ (\p\q)(t,x) $.
This completes the proof of \eqref{e.t.formula.det}.

Finally, we show that \eqref{e.detassump} holds for all $ (t,x)\in(0,T)\times\R $.
Fix any $ t\in(0,T) $.
Use $ \Srho(s;t,x) = \Srho(s+x;t,0) $ and $ \Srhot(s;t,x) = \Srhot(s-x;t,0) $ on the right side of \eqref{e.traceclass} and note that $ \Srho $ and $ \Srhot $ are Schwartz in $ s $.
We see that the trace norm of $ \oindp \orho_{t,x} \oindm \orhot_{t,x} \oindp $ tends to zero as $ x\to\infty $.
Consequently, the determinant tends to $ 1 $ as $ x\to\infty $.
Set 
\begin{align*}
	x_0 = x_0(t) := \inf\big\{ x\in\R : \det(\oid - \oindp \orho_{t,y} \oindm \orhot_{t,y} \oindp) \notin (-\infty,0], \, \forall y>x \big\} < +\infty.
\end{align*}
For all $ x>x_0 $, by the last two paragraphs, the determinant formula~\eqref{e.t.formula.det} holds.
Integrating this formula on $ (x,+\infty) $ twice and exponentiating the result give
\begin{align*}
	\det(\oid - \oindp \orho_{t,x} \oindm \orhot_{t,x} \oindp) = \exp\Big( \int_x^{\infty} \d y_1 \int_{y_1}^{\infty} \d y_2 \, (\p\q)(t,y_2) \Big),
	\quad
	x>x_0.
\end{align*}
Both sides vary continuously in $ x $ (smoothly in fact); the integral on the right side is real and bounded on $ \{x\in\R\} $, so the exponential is real, positive, and bounded away from $ 0 $ for all $ x\in\R $.
These properties force $ x_0 = -\infty $ and force the determinant to always be real and positive. 

\section{The 1-to-1 initial-terminal condition: Proof of Theorem~\ref{c.1-1}}
\label{s.1-1}

To set up the scene for the proof, consider the 1-to-1 initial-terminal condition~\eqref{e.1-1}, and fix any $ (\p,\q) $ as in Theorems~\ref{t.NS} and \ref{t.formula}.
Recall that $ \gamma =:\gamma_1 $ is the parameter in $ \ptc $.
Our task is to find the scattering coefficients and the value of $ \gamma $ based on the information $ \qic=\delta_0 $ and $ \q(T,0)=e^{\alpha} $.
We will first treat $ \gamma\in\R $ as a generic parameter, find the scattering coefficients in terms of $ \gamma $, and later fix $ \gamma $.

We begin with $ \Sb(\lambda) $ and $ \Sbt(\lambda) $.
Recall that $ \jost^{-,i} $ denotes the $ i $-th column of the Jost solution $ \jost^{-} $.
Solve the auxiliary linear problem at $ t=0 $ \eqref{e.jost.t=0} for $ \jost^{-,1} $ to get
\begin{align*}
	&\jost^{-,1}(\lambda;0,x)
	=
	e^{-\img\frac{\lambda}{2}x\sigma_3} \Matrix{ 1 \\ 1 }
	+
	\int_0^x \d y \, e^{-\img\frac{\lambda}{2}(x-y)\sigma_3} \Matrix{ - \p(0,y) e^{\img\frac{\lambda}{2}y} \\ 0 },
	&
	& x > 0.&
\end{align*}
Recall that $ \S(\lambda) = \lim_{x\to\infty}  e^{\img\frac{\lambda}{2}x\sigma_3}\, \jost^{-}(\lambda;0,x) $.
This gives $ \Sb(\lambda) = 1 $. 
The same argument applied to $ \jost^{-,2}(\lambda;T,x) $ gives $ \Sbt(\lambda) e^{\lambda^2 T/2} = -\gamma $.
We have obtained
\begin{align}
	\label{e.1-1.Sb}
	&&&
	\Sb(\lambda) = 1,
	&&
	\Sbt(\lambda) = -\gamma e^{-\lambda^2 T/ 2},
	&&
	\lambda\in\C.&&&
\end{align}

The strategy for finding $ \Sa(\lambda) $ and $ \Sat(\lambda) $ is to use Property~\ref{property.S.aa-bb=1} in Section~\ref{s.solving.scatter}, which we rewrite here as
\begin{align}
	\label{e.scalar.RH}
	&&&&&\Sa(\lambda) \Sat(\lambda) = 1 - \gamma e^{-\lambda^2 T/ 2} := \g(\lambda),
	&&
	\lambda \in \C.&&&&
\end{align}
We will divide $ \C $ into two regions $ D_\up $ and $ D_\lw $ by a contour $ C_0 $.
Then, with the aid of \eqref{e.scalar.RH}, we will identify the zeros of $ \Sa(\lambda) $ and $ \Sat(\lambda) $ on $ D_\up $ and $ D_\lw $ respectively.
Once the zeros are identified, we will formulate a \emph{scalar} Riemann--Hilbert problem for $ (\Sa,\Sat) $, with $ C_0 $ being the jump contour and with $ \eqref{e.scalar.RH}|_{C_0} $ being the jump condition, and solve the problem to find $ (\Sa,\Sat) $.

To execute the strategy outlined previously, we need to first locate the zeros of $ \g(\lambda) $.
Indeed, by \eqref{e.scalar.RH}, the function $ \Sa(\lambda) $ or $ \Sat(\lambda) $ can be zero only when $ \g(\lambda) $ is.
The function $ \g(\lambda) $ never vanishes when $ \gamma=0 $, so we consider $ \gamma\in\R\setminus\{0\} $.
Let $ \sqrt{T/2}z $ denote a generic zero of $ \g(\lambda) $, and write $ z = x +\img y $ for the real and imaginary parts.
Straightforward calculations give $ x^2 - y^2 = \log|\gamma| $ and $ 2xy = \img n \pi $, where $ n\in 2\Z $ when $ \gamma>0 $ and $ n\in2\Z-1 $ when $ \gamma<0 $.
Let us index the zeros of $ \g(\lambda) $ on $ \{\im(\lambda)>0 \}\cup[0,\infty) $ as $ \{\ldots,z_{-1},z_{1},\ldots\} $ or $ \{\ldots,z_{-2},z_0,z_{2},\ldots\} $, where the index denotes the `$ n $' in the equation $ 2xy = \img n \pi $.
Accordingly, the zeros of $ \g(\lambda) $ on $ \{\im(\lambda)<0 \}\cup(-\infty,0] $ are given by $ \{-z_n\}_n $.
Note that $ z_n $ depends on $ \gamma $, and we will mostly omit the dependence to simplify notation.
See Figure~\ref{f.zeros} for an illustration.

\begin{figure}
\includegraphics[width=.8\linewidth]{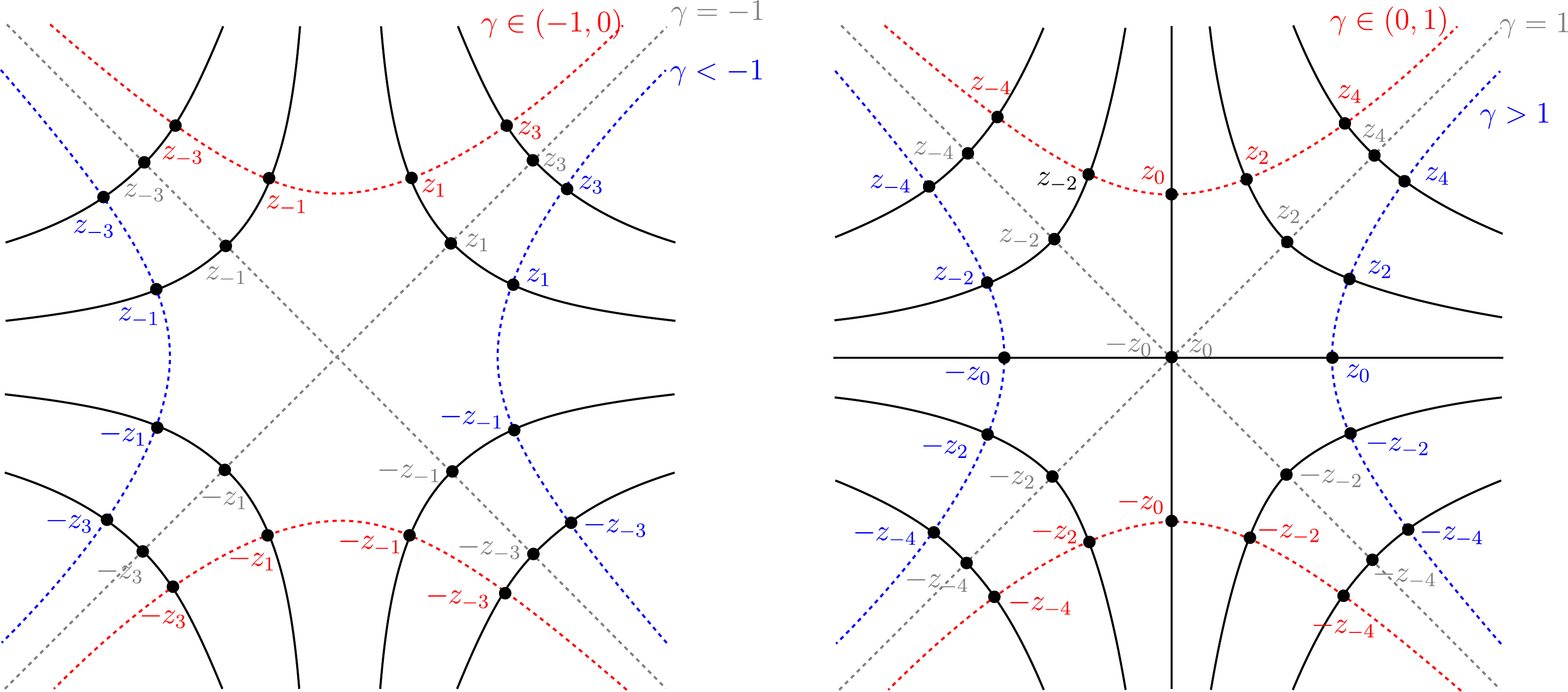}
\caption{The zeros of $ \g $. When $ \gamma=1 $, both $ z_0 $ and $ -z_0 $ are equal to $ 0 $.}
\label{f.zeros}
\end{figure}

Let $ C_0 := \R $ when $ \gamma \leq 1 $, let $ C_0 $ be as depicted in Figure~\ref{f.C0} when $ \gamma>1 $, and let $ D_\up $ and $ D_\lw $ denote the respective open regions above and below $ C_0 $.
\begin{figure}
\includegraphics[width=.5\linewidth]{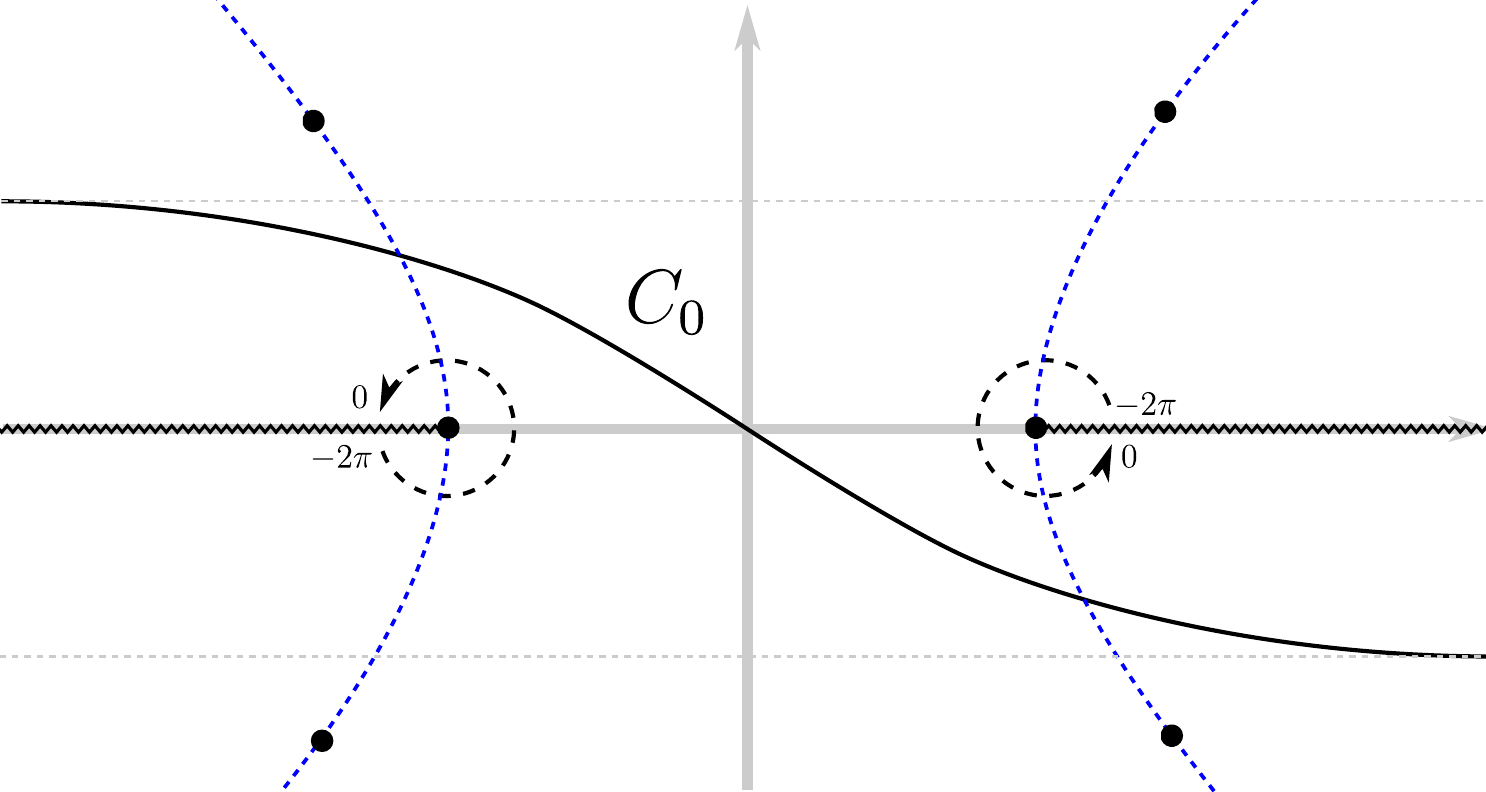}
\caption{When $ \gamma>1 $: the contour $ C_0 $ and the function $ \log(1-\gamma e^{-\eta^2T/2})|_{C_0} $.
In this figure, the dots along the dashed curves are $ \{\sqrt{2/T}z_n,-\sqrt{2/T}z_n\}_{n\in2\Z} $.
In particular, the dots on the real axis are $ \sqrt{2/T}z_0 $ and $ -\sqrt{2/T}z_0 $.
The contour $ C_0 $ is a smooth curve that avoids these two dots, has the symmetry with respect to the origin, and decreases to $ -u \in(\im(-\sqrt{T/2}z_{2}),0) $ as $ x\to\infty $.
For example $ C_0 := \{x-\img u\tanh(x): x\in\R\} $, for $ u\in(0,\sqrt{\pi T/2}) $, will do.
In the integral~\eqref{e.phi} that defines $ \varphi $, we define $ \log(1-\gamma e^{-\eta^2T/2})|_{\R\times(-\img u,\img u)} $ with the branch cuts as shown in this figure;
the values of $ \im(\log(1-\gamma e^{-\eta^2T/2})) $ around $ \sqrt{2/T}z_0 $ and $ -\sqrt{2/T}z_0 $ are also shown.%
}
\label{f.C0}
\end{figure}

Proceeding to identify the zeros of $ \Sa(\lambda) $ and $ \Sat(\lambda) $, we begin with some notation and basic properties.
Set $ \Za := \{z \in D_\up : \Sa(z) = 0 \} $ and $ \Zat := \{z \in D_\lw : \Sat(z) = 0 \} $.
These zeros of $ \Sa(\lambda) $ and $ \Sat(\lambda) $ are simple, because they are simple zeros of $ \g(\lambda) $.
Except when $ \gamma=1 $, the functions $ \Sa(\lambda) $ and $ \Sat(\lambda) $ have no zeros on $ C_0 $.
When $ \gamma =1 $, the function $ \g(\lambda) $ has a degree-two zero at $ z_0 = - z_0= 0 $, so by~\eqref{e.scalar.RH} the functions $ \Sa(\lambda) $ and $ \Sat(\lambda) $ each have a simple zero there.
This is the only possible zero of $ \Sa(\lambda) $ or $ \Sat(\lambda) $ on $ C_0 $.
The sets $ \Za  $ and $ \Zat $ are finite thanks to Properties~\ref{property.S.entire} and \ref{property.S.ato1} in Section~\ref{s.solving.scatter}.
Further, they have the same cardinality:
\begin{align}
	\label{e.|Za|=|Zat|}
	\# \Za = \# \Zat < \infty.
\end{align}
This follows by the standard index (winding number) argument, with the aid of Properties~\ref{property.S.entire}, \ref{property.S.ato1}, and \eqref{e.scalar.RH}.

We now formulate the scalar Riemann--Hilbert problem and solve it to find $ \Sa(\lambda) $ and $ \Sat(\lambda) $.
\begin{enumerate}[leftmargin=60pt, label=(sRH \roman*)]
\item \label{sRH.analytic}
The functions $ \Sa(\lambda) $ and $ \Sat(\lambda) $ are entire.
\item \label{sRH.zeros}
The sets $ \Za $ and $ \Zat $ are given and satisfy~\eqref{e.|Za|=|Zat|}; when $ \gamma\neq 1 $, the functions $ \Sa(\lambda) $ and $ \Sat(\lambda) $ do not vanish on $ C_0 $; when $ \gamma=1 $, the functions $ \Sa(\lambda)|_{C_0} $ and $ \Sat(\lambda)|_{C_0}  $ vanish at and only at $ 0 $; all these zeros are simple.
\item \label{sRH.to1}
The functions $ \Sa(\lambda) \to 1 $ and $ \Sat(\lambda) \to 1 $ as $ |\lambda|\to\infty $ on $ \bar{D_\up} $ and $ \bar{D_\lw} $ respectively.
\item \label{sRH.jump}
The jump condition $ \Sa(\lambda) \Sat(\lambda) = \g(\lambda) $ holds along $ C_0 $.
\end{enumerate}
This problem has at most one solution: Given a potentially different solution $ (\Sa_1(\lambda),\Sat_1(\lambda)) $, the function $ (\Sa_1(\lambda)/\Sa(\lambda))\ind_{D_\up\cup C_0}(\lambda) + (\Sat_1(\lambda)/\Sat(\lambda))\ind_{D_{\lw}}(\lambda) $ is entire and tends to $ 1 $ as $ |\lambda|\to\infty $ on $ \C $. 
Such a function must be constant $ 1 $ by Liouville's theorem, so $ \Sa_1=\Sa $ and $ \Sat_1 = \Sat $.
We claim that the following expressions solve the problem:
\begin{align}
	\label{e.1-1.Sa}
	\Sa(\lambda)
	&=
	\frac{\prod_{z\in\Za}(\lambda-z)}{\prod_{z\in\Zat}(\lambda-z)}
	(1-\gamma e^{-\lambda^2T/2})^{\mu_{\Sa}}
	\cdot e^{\varphi(\lambda)},
\\
	\label{e.1-1.Sat}
	\Sat(\lambda)
	&=
	\frac{\prod_{z\in\Zat}(\lambda-z)}{\prod_{z\in\Za}(\lambda-z)}
	(1-\gamma e^{-\lambda^2T/2})^{\mu_{\Sat}}
	\cdot
	e^{-\varphi(\lambda)},
\end{align}
where $ (\mu_{\Sa},\mu_{\Sat}):=(0,1) $, $ (1/2,1/2) $, and $ (1,0) $ when $ \lambda\in D_\up $, $ \lambda\in C_0  $, and $ \lambda\in D_\lw $ respectively, and
\begin{align}
	\label{e.phi}
	\varphi(\lambda)
	:=
	\int_{C_0} \frac{\d\eta}{2\pi\img} \frac{\log(1-\gamma e^{-\eta^2 T/ 2})}{\eta-\lambda},
\end{align}
where the integral is interpreted in the Cauchy sense when $ \lambda\in C_0 $ and the logarithm is interpreted as in Figure~\ref{f.C0} when $ \gamma>1 $.
To verify \ref{sRH.analytic}, first note that the expressions in \eqref{e.1-1.Sa}--\eqref{e.1-1.Sat} are analytic on $ \C\setminus C_0 $ because $ \varphi $ is analytic on $ \C\setminus C_0 $ and because the poles in the fractions are canceled by the factors $ (1-\gamma e^{-\lambda^2T/2})^{\mu_{\Sa}} $ and $ (1-\gamma e^{-\lambda^2T/2})^{\mu_{\Sat}} $.
With the aid of the Sokhotski–Plemelj theorem, it is not hard to show that, for every $ \lambda\in C_0 $,
\begin{align*}
	\lim_{\e\to 0} e^{\pm\varphi(\lambda\pm\img|\e|)} = (1-\gamma e^{-\gamma\lambda^2T/2})^{1/2} \, e^{\pm\varphi(\lambda)},
	\qquad
	\lim_{\e\to 0} (1-\gamma e^{-\gamma\lambda^2T/2})^{1/2} \, e^{\pm\varphi(\lambda\mp\img|\e|)} = e^{\pm\varphi(\lambda)}.
\end{align*}
This implies that the expressions in \eqref{e.1-1.Sa} are continuous on $ \C $ and hence entire.
Next, the condition \ref{sRH.to1} is satisfied thanks to \eqref{e.|Za|=|Zat|} and the property that $ \lim_{|\lambda|\to\infty}\varphi(\lambda) = 0 $. 
The condition \ref{sRH.jump} indeed holds.
Finally, the sets of zeros of \eqref{e.1-1.Sa}--\eqref{e.1-1.Sat} in $ D_\up $ and $ D_\lw $ are indeed $ \Za $ and $ \Zat $ respectively and the zeros are indeed simple.
When $ \gamma=1 $, the conditions~\ref{sRH.analytic} and \ref{sRH.jump} together force \eqref{e.1-1.Sa}--\eqref{e.1-1.Sat} to each have a simple zero at $ 0 $.
This verifies the condition~\ref{sRH.zeros}.


Let us summarize our progress so far and outline the rest of the proof.
So far, we have found $ (\Sb(\lambda),\Sbt(\lambda)) $ in \eqref{e.1-1.Sb} and identified the candidates for $ (\Sa(\lambda),\Sat(\lambda)) $ in \eqref{e.1-1.Sa}--\eqref{e.phi}.
Among these candidates, few are the `physical' ones seen in \cite{krajenbrink21}.
The goal is hence to rule out all `non-physical' candidates.
Our proof here rules out some but not all, and we impose Assumption~\ref{assu.poles} to exclude those non-physical candidates we have not ruled out.
Our argument for ruling out non-physical candidates depends on the conserved quantities.
In Section~\ref{s.1-1.conserved}, we will evaluate the conserved quantities for all candidates.
In Section~\ref{s.1-1.physical}, we will introduce the physical candidates and explain how they yield Theorem~\ref{c.1-1}.
In Sections~\ref{s.1-1.defocusing}--\ref{s.1-1.alpha.gamma}, we will rule out those `non-physical' candidates not covered by Assumption~\ref{assu.poles}.

\subsection{Evaluating the conserved quantities}
\label{s.1-1.conserved}

Recall from Property~\ref{propert.conserved} in Section~\ref{s.solving.scatter} that the conserved quantities can be calculated from the finite-term Taylor expansion of $ (\log\Sa) $ in $ 1/\lambda $ along $ \lambda\in\R $.
Set $ \lambda\in\R $ in \eqref{e.1-1.Sa}--\eqref{e.1-1.Sat} and take the logarithm. 
The term $ \frac12 \log(1-\gamma e^{-\lambda^2T/2}) $ decays super-polynomially fast as $ |\lambda|\to\infty $, so will not contribute to the expansion.
The expansion of $ \varphi $ is carried out in Lemma~\ref{l.phi.expansion}, and the result reads
\begin{align}
	\label{e.phi.expansion}
	\varphi(\lambda)
	=
	\sum_{k=0}^{n} \frac{1}{\lambda^{2k+1}}\Big(\frac{2}{T}\Big)^{\frac{2k+1}{2}}  
	\Big(
		-\int_{\R}  \frac{\d\eta}{2\pi\img} \, \eta^{2k} \log\big|1-\gamma e^{-\eta^2}\big|
		+
		\frac{z_0^{2k+1}}{2k+1} \ind_{\{\gamma>1\}}
	\Big)
	+
	O(|\lambda|^{-2n-1}).
\end{align}
Next, the fractions in \eqref{e.1-1.Sa}--\eqref{e.1-1.Sat} contribute
$
	\sum_{k=1}^{n} \frac{-1}{k\lambda^k} ( \sum_{\Za} - \sum_{\Zat} ) (2/T)^{k/2}z^k + O(|\lambda|^{-n-1}).
$
Altogether, assuming the candidate in \eqref{e.1-1.Sa}--\eqref{e.1-1.Sat} gives a solution of the NLS equations, we have
\begin{subequations}
\label{e.1-1.conserved}
\begin{align}
\label{e.1-1.conserved.1}
	\conserved_k 
	=
	\Big(\frac{2}{T}\Big)^{\frac{k}{2}}  
	\Big(
		&(-1)^{\frac{k-1}{2}} \Big( -
			\int_{\R}  \frac{\d\eta}{2\pi} \, \eta^{k} \log\big|1-\gamma e^{-\eta^2}\big|
			+
			\frac{\img z_0^{k}}{k} \ind_{\{\gamma>1\}}
		\Big)\ind_{\{\frac{k-1}{2}\in\Z_{\geq 0}\}}
\\	
\label{e.1-1.conserved.2}
		&-
		\frac{\img^k}{k}\Big( \sum_{\Za} - \sum_{\Zat} \Big) z^k
	\Big).
\end{align}
\end{subequations}

Among the conserved quantities, $ \conserved_1 $ and $ \conserved_3 $ are of particular relevance.
For any $ (\p,\q) $ as in Theorem~\ref{t.NS} with the 1-to-1 initial-terminal condition, the corresponding $ \conserved_1 $ and $ \conserved_3 $ satisfy the relations:
\begin{align}
	\label{e.1-1.matching.conserved}
	\conserved_1 &= \gamma e^{\alpha},
\\
	\label{e.1-1.rate.conserved}
	\tfrac{1}{2} \norm{\w}^2_{2;[0,T]\times\R} & = \conserved_1 + T \conserved_3. 
\end{align}
The relation~\eqref{e.1-1.matching.conserved} follows by taking the limit $ t\to T $ in $ \conserved_1 = \int_{\R} \d x \, (\p\q)(t,x) $ and using $ \q(T,0)=e^{\alpha} $ and $ \p(T,\Cdot) = \gamma \delta_0 $.
The relation~\eqref{e.1-1.rate.conserved} is verified in Lemma~\ref{l.1-1.L2}.

\subsection{The physical candidates}
\label{s.1-1.physical}

We begin by describing the physical candidates.
They are
\begin{description}[leftmargin=20pt]
\item[(Non-solitonic candidate)] $ \gamma\leq 1 $ and $ \Za = \Zat = \emptyset $.
\item[(Solitonic candidate)] $ \gamma \in (0,1) $ and $ \Za = \{z_0\} $ and $ \Zat = \{-z_0\} $.
\end{description}
We name these candidates non-soliton and soliton depending on whether $\Za \cup\Zat $ is empty.
More explicitly, the $ \Sa(\lambda) $ and $ \Sat(\lambda) $ of the non-solitonic candidate are given by 
\begin{align}
	\label{e.1-1.Sa.nonsoliton}
	&&\Sa(\lambda)
	=
	(1-\gamma e^{-\lambda^2T/2})^{\mu_{\Sa}}\cdot e^{\varphi(\lambda)},
	&&
	\Sat(\lambda)
	=
	(1-\gamma e^{-\lambda^2T/2})^{\mu_{\Sat}}\cdot e^{-\varphi(\lambda)},&&
\end{align}
where $ (\mu_{\Sa},\mu_{\Sat}):=(0,1) $, $ (1/2,1/2) $, and $ (1,0) $ when $ \im(\lambda)>0 $, $ \lambda\in \R $, and $ \im(\lambda)<0 $ respectively, and $ \varphi $ is given in \eqref{e.phi} with $ C_0=\R $.
Similarly, the $ \Sa(\lambda) $ and $ \Sat(\lambda) $ of the solitonic candidate are given by 
\begin{align}
	\label{e.1-1.Sa.soliton}
	&&\Sa(\lambda)
	=
	\frac{\lambda+\img\kappa}{\lambda-\img\kappa}
	(1-\gamma e^{-\lambda^2T/2})^{\mu_{\Sa}}
	\cdot e^{\varphi(\lambda)},
	&&
	\Sat(\lambda)
	=
	\frac{\lambda-\img\kappa}{\lambda+\img\kappa}
	(1-\gamma e^{-\lambda^2T/2})^{\mu_{\Sat}}
	\cdot e^{-\varphi(\lambda)},&&
\end{align}
where $ \kappa := \sqrt{2/T}z_0/\img = \sqrt{(2/T)\log(1/\gamma)} $.
Recall $ \ratenosoliton $ and $ \ratesoliton $ from \eqref{e.1-1.rate}.
Hereafter, we adopt the convention that $ \rateone := \ratenosoliton $ for the non-solitonic candidate, and $ \rateone := \ratesoliton $ for the solitonic candidate.
Specializing \eqref{e.1-1.conserved} to the physical candidates gives $ \conserved_1 = \sqrt{2/T}\rateone'(\gamma)\gamma $ and $ \conserved_3 = \sqrt{2/T^3}\rateone(\gamma) $. 

We can now fix the value of $ \gamma $ for the physical candidates.
Using~\eqref{e.1-1.matching.conserved} gives
\begin{align}
	\tag{\ref*{e.1-1.gamma}'}
	\label{e.1-1.gamma.}
	\sqrt{T/2}\,e^{\alpha} 
	= 
	\rateone'(\gamma).
\end{align}
Recall the sets $ \nosoliton $ and $ \soliton $ and the threshold $ c_\star $ from \eqref{e.soliton.range}.
As shown in Lemma~\ref{l.1-1.convex}, we have $ \ratenosoliton' \leq c_\star $ while $ \ratesoliton' > c_\star $.
Using these properties in \eqref{e.1-1.gamma.} shows that when $ \sqrt{T/2}\,e^{\alpha}\in\nosoliton $ (respectively $ \in\soliton $), only the non-solitonic candidate (respectively the solitonic candidate) is possible.
In either scenario, Equation~\eqref{e.1-1.gamma.} has a unique solution thanks to Lemma~\ref{l.1-1.convex}.

We next evaluate the squared $ \Lsp^2 $ norm of $ \w $.
Combine \eqref{e.1-1.rate.conserved} with \eqref{e.1-1.gamma.} and $ \conserved_3 = \sqrt{2/T^3}\rateone(\gamma) $ to get
\begin{align}
	\label{e.1-1.rate.result.}
	\tfrac12 (\norm{\w}_{2;[0,T]\times\R})^2
	=
	\big( \sigma e^{\alpha} - \tfrac{1}{\sqrt{T/2}} \rateone(\sigma) \big)\big|_{\sigma=\gamma}.
\end{align}
View the expression in $ (\ldots) $ on the right side of \eqref{e.1-1.rate.result.} as a function of $ \sigma $.
By \eqref{e.1-1.gamma.}, the point $ \sigma=\gamma $ is critical for this function.
Further, by Lemma~\ref{l.1-1.convex}, the function is strictly concave in $ \sigma\in(-\infty,1] $ when $ \sqrt{T/2}\,e^{\alpha}\in\nosoliton $, and is strictly convex in $ \sigma\in(0,1) $ when $ \sqrt{T/2}\,e^{\alpha}\in\soliton $.
Hence the right side of \eqref{e.1-1.rate.result.} is the respective maximum and minimum in \eqref{e.1-1.rate.result}.

We have shown that, for every $ \alpha\in\R $, there exists a unique physical candidate and the physical candidate gives the relation~\eqref{e.1-1.gamma} between $ \alpha $ and $ \gamma $ and gives the rate function~\eqref{e.1-1.rate.result}.
As soon as the non-physical candidates are ruled out or excluded, Theorem~\ref{c.1-1} will follow.

\subsection{Ruling out the non-physical candidates: the defocusing regime}
\label{s.1-1.defocusing}

Here we rule out all non-physical candidates with $ \gamma < 0 $.
This range of $ \gamma $ corresponds to the defocusing regime of the NLS equations.
First, the minimizer of the variational problem is unique.
This can be shown by the PDE argument in the mean-field game literature (see~\cite[Section~1.1.5]{gomes16} for example) with suitable adaptation for the initial-terminal condition.
The uniqueness implies that $ \gamma\q(T-t,x) = \p(t,x) $, otherwise $ (\q_1,\p_1)|_{(t,x)} := (\p/\gamma,\gamma\q)|_{(T-t,x)} $ would give another minimizer.
In particular, $ -|\gamma| \q(T/2,x) = \p(T/2,x) $.
Now, consider the auxiliary linear problem $ \partial_x \jost = \U \jost $ at $ t=T/2 $ and rewrite it via conjugation as
\begin{align*}
	\partial_x \Matrix{\ldots}
	=
	\Matrix{ -\img\lambda/2 & - |\gamma|^{-1/2}\p(T/2,x) \\ |\gamma|^{1/2}\q(T/2,x) & \img\lambda/2 } \Matrix{\ldots}.
\end{align*}
Substitute in $ |\gamma|^{1/2}\q(T/2,x) = -|\gamma|^{-1/2}\p(T/2,x) $ and note that $ \p $ is real.
The standard argument in the analysis of the defocusing NLS equation applies and yields that $ \Sa(\lambda) $ and $ \Sat(\lambda) $ have no zeros in the upper and lower half planes respectively, namely $ \Za =\emptyset $ and $ \Zat=\emptyset $; see~\cite[pp 48--49]{faddeev07} for example.
This rules out all non-physical candidates with $ \gamma < 0 $. 

\subsection{Ruling out the non-physical candidates: real conserved quantities}
\label{s.1-1.real.conserved}

Here, we derive a constraint on the sets $ \Za $ and $ \Zat $ to rule out some non-physical candidates.
The constraint is derived from the fact that the conserved quantities are real.
\begin{lem}
\label{l.1-1.symmetryofZaZat}
Assume that $ \Sa(\lambda) $ and $ \Sat(\lambda) $ are given by a minimizer of the variational problem.
For any $ n\neq 0 $, $ z_n \in \Za $ implies $ z_{-n} \in \Za $, and similarly $ -z_n \in \Zat $ implies $ -z_{-n} \in \Zat $.
\end{lem}
\begin{proof}
Set $ n_0 := \max\{ |n| : z_n \in \Za $ or $ -z_n \in \Zat \} $, with the convention $ \max\emptyset := -\infty $, and consider the case $ n_0>0 $.
Set $ r:= |z_{n_0}| $.
Since $ \conserved_k $ is real, the imaginary part of \eqref{e.1-1.conserved} is zero.
Take the imaginary part of \eqref{e.1-1.conserved}, multiply the result by $ kr^{-k} $, and send $ k\to\infty $.
In the limit, only the contribution of those $ z $s with $ |z| = r $ survives, whereby
\begin{align}
	\label{e.l.1-1.symmetryofZaZat.1}
	\lim_{k\to\infty} \im\Big( \img^k\Big( \sum\nolimits_{z\in\Za,|z|=r} - \sum\nolimits_{z\in\Zat,|z|=r} \Big) \frac{z^k}{r^k} \Big) = 0.
\end{align}
Write $ z_{n_0} = r e^{\img\theta} $, where $ \theta\in(0,\pi/2) $.
In \eqref{e.l.1-1.symmetryofZaZat.1}, potential elements in the sums are $ z_{n_0} = re^{\img\theta} $, $ z_{-n_0} = -re^{-\img\theta} $, $ -z_{n_0} = -re^{\img \theta} $, and $ -z_{-n_0} = r e^{-\img\theta} $, so there are $ 2^4-1 $ potential combinations.
Examining each of these combination shows that the possible ones are $ \{z_{n_0},z_{-n_0}\} $, $ \{-z_{n_0},-z_{-n_0}\} $, and $ \{z_{n_0},z_{-n_0},-z_{n_0},-z_{-n_0}\} $.
This shows that the desired property holds for $ n=n_0 $.
Further, in each of these combinations, we have $ \img^k(\sum_{z\in\Za,|z|=r} - \sum_{z\in\Zat,|z|=r})z^k \in \img \R $, for \emph{all} $ k\in\Z_{>0} $.
We can hence remove those $ z $s with $ |z|=r $ from \eqref{e.1-1.conserved} and repeat the preceding argument for the remaining $ z $s.
Proceeding this way completes the proof.
\end{proof}

With the aid of Lemma~\ref{l.1-1.symmetryofZaZat}, we can rule out all candidates with $ \gamma>1 $. 
Consider $ \conserved_1 $ and set $ k=1 $ in \eqref{e.1-1.conserved}. 
In \eqref{e.1-1.conserved.2}, separate the contribution of $ z_0 $ and $ -z_0 $ (if any) from the rest.
Given the property in Lemma~\ref{l.1-1.symmetryofZaZat}, the contribution of the rest is real.
This observation together with $ \conserved_1\in\R $ forces
$
	\img z_0 - \img ( \sum_{\Za\cap\{z_0\}} - \sum_{\Zat\cap\{-z_0\}} ) z
$
to be real.
On the other hand, Lemma~\ref{l.1-1.symmetryofZaZat} and \eqref{e.|Za|=|Zat|} together imply that the last expression is either $ \img z_0 - 0 $ or $ \img z_0 -\img(z_0-(-z_0)) $. Neither is real since $ z_0 = \sqrt{\log\gamma} >0 $.
Hence $ \gamma>1 $ is impossible.

\subsection{Ruling out the non-physical candidates: the $ \alpha $-$ \gamma $ relation}
\label{s.1-1.alpha.gamma}

Here we rule out some non-physical candidates by using \eqref{e.1-1.matching.conserved}, which we refer to as the $ \alpha $-$ \gamma $ relation.
The only remaining non-physical candidates all have $ \gamma \in (0,1] $.
Fix a non-physical candidate with $ \gamma \in (0,1] $.
By the definition of the non-physical candidates, the set $ (\Za\cup\Zat)\setminus\{z_0,-z_0\} $ is nonempty.
Using this property, Lemma~\ref{l.1-1.symmetryofZaZat}, and \eqref{e.|Za|=|Zat|} shows that the sets $ \Za\setminus\{z_0\} $ and $ \Zat\setminus\{-z_0\} $ must each contain at least two elements.
Using this information in \eqref{e.1-1.conserved} for $ k=1 $ gives
$
	\sqrt{T/2}\conserved_1 \geq (\Li_{3/2}(\gamma)/\sqrt{4\pi} + 4\,\im(z_2)).
$
This inequality, together with the $ \alpha $-$ \gamma $ relation~\eqref{e.1-1.matching.conserved}, rules out the non-physical candidate when $ \sqrt{T/2}\,e^{\alpha} <c_{\star,1} $, where
\begin{align}
	\label{e.thresdhol.assump}
	c_{\star,1} 
	:= 
	\inf_{\gamma\in(0,1]} \,\tfrac{1}{\gamma}\,\Big\{ \tfrac{1}{\sqrt{4\pi}} \Li_{3/2}(\gamma) + 4 \,\im\big(\sqrt{-\log(1/\gamma) + 2\pi\img}\,\big) \Big\}
	=
	9.4296\ldots,
\end{align}
where the square root is taken on the upper half plane.
This threshold $ c_{\star,1} $ is larger than the non-soliton-to-soliton threshold $ c_{\star} := \Li_{5/2}'(1)/\sqrt{4\pi} = 0.7369\ldots $ in \eqref{e.soliton.range}.

The remaining non-physical candidates are excluded by Assumption~\ref{assu.poles}.

\section{The 1-to-1 initial-terminal condition: Proof of Corollary~\ref{c.1-1.large}}
\label{s.1-1.large}

We begin with some reductions.
First, note that the symmetries $ \p(t,x) = \gamma\q(T-t,-x) $, $ \w(t,x)=\w(t,-x) $, and $ (\p,\q)|_{(t,x)} = (\p,\q)|_{(t,-x)} $ hold.
The first symmetry is straightforwardly verified from the explicit expressions of $ (\p,\q) $ given by \eqref{e.t.formula.p}--\eqref{e.t.formula.q}, \eqref{e.Srho}--\eqref{e.Srhot}, \eqref{e.1-1.Sb}, and \eqref{e.1-1.Sa.nonsoliton}--\eqref{e.1-1.Sa.soliton}.
The second symmetry is straightforwardly verified from the explicit expression of $ \w $, with the aid of the identity $ \det(\oid-\mathbf{u}\mathbf{v}) = \det(\oid-\mathbf{v}\mathbf{u}) $.
The third symmetry follows from the second through Definition~\ref{d.duhamelsense}.
Given the second and third symmetries, \emph{we will only consider $ x \geq 0 $.}
Combining the first and third symmetries gives $ \p(t,x) = \gamma\q(T-t,x) $.
Given this relation, the estimates~\eqref{e.c.1-1.large.3}--\eqref{e.c.1-1.large.5} of $ \p $ and $ \w=\p\q $ follow from the estimate~\eqref{e.c.1-1.large.1}--\eqref{e.c.1-1.large.2} of $ \q $.
Hence, \emph{we will only consider $ \q $.}

We begin by setting up the scaling.
Recall that $ N\to\infty $ denotes the scaling parameter.
Set $ T=2N $, $ \alpha = N \salpha $, and $ \gamma = e^{-N\sgamma} $.
As $ N\to\infty $, the parameter $ \alpha =  N \salpha $ eventually belongs to the solitonic range $ \soliton $ (see \eqref{e.soliton.range}), whence $ \sgamma $ solves the second equation in \eqref{e.1-1.gamma}.
Under the scaling, the equation reads
\begin{align}
	\label{e.1-1.gamma.scaling}
	\tfrac{1}{\sqrt{4\pi N}} \, \Li_{5/2}'(e^{-N\sgamma}) + 2\sqrt{\sgamma} e^{N\sgamma} = e^{N\salpha}.
\end{align}

Let us prepare the notation.
Let $ \ip{ f , g } := \int_{\R} \d x \, \bar{f}(x) \, g(x) $ denote the inner product on $ \Lsp^2(\R) $.
For a unit vector $ f \in \Lsp^2(\R) $, let $ \projection{f} $ denote the projection onto $ f $, more explicitly $ \projection{f}(\phi) := f\, \ip{f,\phi} $.
Let $ \reflection $ denote the reflection operator, namely $ (\reflection \phi)(s) := \phi(-s) $.
We adopt the notation in Section~\ref{s.solving.proof}: often omitting the dependence on $ (t,x) $, for example $ \orho_{t,x} = \orho $, and writing $ \oindm(\ldots)\oindp = {}_{\smallm}(\ldots)_{\smallp} $.

The idea of the proof is to use the formula~\eqref{e.t.formula.q} for $ \q $, more precisely the first expression on the right side of the formula.
Doing so requires estimating the operators $ {}_{\smallp}\,\orho\,{}_{\smallm} $ and $ {}_{\smallm}\,\orhot\,{}_{\smallp} $.
We will decompose these operators into their leading parts and remaining parts, use the decomposition to produce an approximation of $ (\oid - {}_{\smallp}\orho{}_{\smallm}\orhot{}_{\smallp})^{-1} $, and use this approximation to (approximately) evaluate $ \q $ through \eqref{e.t.formula.q}.

As said, the first step is to decompose $ {}_{\smallp}\,\orho\,{}_{\smallm} $ and $ {}_{\smallm}\,\orhot\,{}_{\smallp} $.
The kernels of these operators admit explicit expressions through \eqref{e.Srho}, \eqref{e.1-1.Sb}, and \eqref{e.1-1.Sa.soliton}.
A careful analysis applied to the expressions produces detailed estimates of the operators.
The analysis is performed in Lemmas~\ref{l.1-1.large.Srho}--\ref{l.1-1.large.orho}, and we state the result as follows.
\begin{align}
	\label{e.1-1.large.orho}
	&&&
	{}_{\smallp}\,\orho\,{}_{\smallm} = \orho_1 + \orho_2,
	&&
	{}_{\smallm}\,\orhot\,{}_{\smallp} = \orhot_1 + \orhot_2.&
\end{align}
The operators $ \orho_1 $ and $ \orhot_1 $ contribute the leading parts and are given by
\begin{align}
	\tag{\ref*{e.1-1.large.orho}b}
	\label{e.1-1.large.orho.2}
	\orho_1 
	:= 
	\ind_{\{x<\sqrt{\sgamma} t\}} \, e^{\sgamma t/2 - \sqrt{\sgamma} x} \, \projection{\rankone_{\sgamma}}\, \reflection\,,
	&&
	\orhot_1
	:=
	- \ind_{\{x<\sqrt{\sgamma} (2N-t)\}} \, e^{-\sgamma t/2 - \sqrt{\sgamma} x} \, \reflection \, \projection{\rankone_{\sgamma}},
\end{align}
where $ \rankone_{\sgamma}(s) := \sqrt{2\sqrt{\sgamma}} \exp(-\sqrt{\sgamma} s) \ind_{\{s>0\}} $.
The operators $ \orho_2 $ and $ \orhot_2 $ constitute the remaining parts.
Recall the definition of `having an almost continuous kernel' from before Theorem~\ref{t.formula}, and recall the notation $ {}_{0}[\mathbf{f}]_{0} $ from there. 
Similarly define $ {}_{0}[\mathbf{f} := f(0^\pm-\Cdot) = f(\Cdot) \in \Lsp^2(\R) $.
Let $ \norm{\ }_{\op} $ denote the operator norm for bounded operators on $ \Lsp^2(\R) $.
The operators $ \orho_2 $ and $ \orhot_2 $ have almost continuous kernels and
\begin{align}
	\tag{\ref*{e.1-1.large.orho}c}
	\label{e.1-1.large.orho.3}
	\Norm{\orho_2}_{\op} \ , \Norm{\,{}_{0}[\orho_2\,}_{2}
	&= 
	O(1)  e^{\sgamma t/2} ( e^{-\sgamma N} + (1+\sqrt{t})  e^{-\sgamma t/2} ),
\\
	\tag{\ref*{e.1-1.large.orho}d}
	\label{e.1-1.large.orho.4}
	\Norm{\orhot_2}_{\op} \ , \Norm{\,{}_{0}[\orhot_2\,}_{2}
	&= 
	O(1) e^{-\sgamma t/2} ( e^{-\sgamma N} + (1+\sqrt{2N-t})e^{-\sgamma(2N-t)/2} ).
\end{align}
Hereafter, the big $ O $ notation is understood with $ \sgamma\in(0,\infty) $ being \emph{fixed}, so $ \sgamma = O(1) $ for example.

Next, we proceed to find an approximate expression of $ (\oid - {}_{\smallp}\orho{}_{\smallm}\orhot{}_{\smallp})^{-1} $.
Let us first ignore the contribution of $ \orho_2 $ and $ \orhot_2 $ and evaluate the inverse $ (\oid - \orho_1 \orhot_1)^{-1} $.
With the aid of $ \projection{\rankone_{\sgamma}}^2 = \projection{\rankone_{\sgamma}} $, we have
\begin{align}
	\label{e.1.1.inverse}
	(\oid - \orho_1 \orhot_1)^{-1}
	=
	\oid - \ind_{\{x<\sqrt{\sgamma} \tau\}} (e^{2\sqrt{\sgamma} x} +1)^{-1} \, \projection{\rankone_{\sgamma}},
\end{align}
where $ \tau := \min\{t,2N-t\} $.
To take into account the remaining parts, we use the identity
$
	(\oid - \mathbf{u} - \mathbf{v})^{-1}
	=
	(\oid - \mathbf{u})^{-1} \sum_{n=0}^\infty ( \mathbf{v} (\oid - \mathbf{u})^{-1} ),
$
which holds has long as $ \norm{(\oid-\mathbf{u})^{-1}}_{\op} \, \norm{\mathbf{v}}_{\op} < 1 $.
We seek to apply this identity with $ \mathbf{u} = \orho_1 \orhot_1 $ and $ \mathbf{v} = \orho_1 \orhot_2 + \orho_2 \orhot_1 + \orho_2 \orhot_2 $.
As is readily verified from \eqref{e.1-1.large.orho}--\eqref{e.1.1.inverse}, $ \norm{(1-\mathbf{u})^{-1}}_{\op} \leq 2 $, $ \norm{\mathbf{v}}_{\op} \leq c_0 \sqrt{1+\tau} e^{-\sgamma\tau/2} $, and $ \norm{\,{}_{0}[\mathbf{v}\,}_{2} \leq c_0 \sqrt{1+\tau} e^{-\sgamma\tau/2} $, where $ c_0 < \infty $ depends only on $ \sgamma $.
Hereafter, we assume $ \tau $ is large enough so that $ c_0 \sqrt{1+\tau} e^{-\sgamma\tau/2} < 1/3 $.
We have
\begin{align}
	\label{e.1.1.inverse.1}
	(\oid - {}_{\smallp}\orho{}_{\smallm}\orhot{}_{\smallp})^{-1}
	=
	\oid - (e^{2\sqrt{\sgamma} x} +1)^{-1} \, \projection{\rankone_{\sgamma}}
	+
	\oremainder,
\end{align}
where $ \oremainder $ satisfies the bound
$
	\Norm{\, {}_{0}[\oremainder \, }_2
	=
	O(\sqrt{1+\tau} \,e^{-\sgamma\tau/2}).
$

Equipped with \eqref{e.1.1.inverse.1}, we now evaluate $ \q $ through the formula~\eqref{e.t.formula.q}.
Doing so amounts to right multiplying \eqref{e.1.1.inverse.1} with $ {}_{\smallp} \orho ]_0 $ and applying $ {}_0[ $ to the result.
For $ {}_{\smallp} \orho ]_0 $, we turn to the estimate of $ \Srho $ in Lemma~\ref{l.1-1.large.Srho}.
Multiply the expressions in \eqref{e.l.1-1.large.Srho1}--\eqref{e.l.1-1.large.Srho2} by $ \ind_{\{s>0\}} $, and let $ \frhoI $ and $ \frhoII $ denote the respective results.
View $ \frhoI = \frhoI(s) $ and $ \frhoII = \frhoII(s) $ as elements in $ \Lsp^2(\R) $ so that $ {}_{\smallp} \orho ]_0 = \frhoI + \frhoII $.
Right multiply \eqref{e.1.1.inverse.1} with $ \frhoI $ and apply $ {}_0[ $ to the result.
Straightforward calculations give
\begin{subequations}
\label{e.1-1.large.qI}
\begin{align}
	\label{e.1-1.large.qI.1}
	\prescript{}{0}{\big[}  (\oid - {}_{\smallp}\orho{}_{\smallm}\orhot{}_{\smallp})^{-1} \, \frhoI
	=\,&
	e^{\sgamma t/2} \sqrt{\sgamma} \, 
	\Big(
		(1+e^{-2\sqrt{\sgamma}t})\sech(\sqrt{\sgamma}x) \, \ind_{\{x<\sqrt{\sgamma}(2N-t)\}} 
\\
	\label{e.1-1.large.qI.2}
		&\qquad
		+ 2e^{-\sqrt{\sgamma}|x|} \, \ind_{\{x\geq\sqrt{\sgamma}(2N-t)\}} 
	\Big)
\\
	\label{e.1-1.large.qI.3}
	&\cdot (1+O(e^{-\sgamma N} + \sqrt{1+\tau} e^{-\sgamma\tau/2}))) \ind_{\{x<\sqrt{\sgamma}t\}}.
\end{align}
\end{subequations}
In~\eqref{e.1-1.large.qI.1}, absorb the contribution of the term $ +e^{-2\sqrt{\sgamma}t} $ into the $ O(\ldots) $ in \eqref{e.1-1.large.qI.3}.
In~\eqref{e.1-1.large.qI.2}, approximate the term $ 2e^{-\sqrt{\sgamma}x} \, \ind_{\{x>\sqrt{\sgamma}(2N-t)\}} $ by $ \sech(\sqrt{\sgamma}x) \, \ind_{\{x>\sqrt{\sgamma}(2N-t)\}} $, and note that the error can be absorbed into the $ O(\ldots) $ in \eqref{e.1-1.large.qI.3}.
In~\eqref{e.1-1.large.qI.3}, note that the leading term in the $ O(\ldots) $  is $ \sqrt{1+\tau} e^{-\sgamma\tau/2} $.
We have shown  that \eqref{e.1-1.large.qI} gives the expression in \eqref{e.c.1-1.large.1}.
Moving on, we right multiply \eqref{e.1.1.inverse.1} with $ \frhoII $ and apply $ {}_0[ $ to the result.
Straightforward calculations show that the result gives the expression in \eqref{e.c.1-1.large.2} plus $ O(1)(e^{2\sqrt{\sgamma} x} +1)^{-1}e^{\sgamma t/2} \sqrt{1+\tau}e^{-\sgamma\tau/2} \ind_{\{x<\sqrt{\sgamma}\tau\}} $.
Absorbing the last expression into~\eqref{e.c.1-1.large.1} completes the proof.

\appendix

\section{The existence of a minimizer}

\begin{lem}
\label{l.minimization.exist}
The set in the variational problem~\eqref{e.minimization} is nonempty and has a minimizer.
\end{lem}

\begin{proof}
To show the nonemptiness, consider first $ \q_0(t,x) := \int_\R \d y \ \hk(t,x-y) \qic(y) $, which solves the heat equation with the initial condition $ \qic $.
Add a function $ g \in \Ccsp^\infty([0,T]\times\R) $ to $ \q_0 $ to get $ \q_1 := \q_0 + g $.
We require $ g(0,\Cdot) \equiv 0 $, $ \q_1(T,\xi_{\ii}) = e^{\alpha_\ii} $ for $ \ii=1,\ldots,\mm $, and $ \q_1 > 0 $ everywhere on $ (0,T]\times \R $.
These conditions are indeed achievable.
We would like to realize $ \q_1 $ as $ \qfn[\omega] $ for some $ \omega \in \Lsp^2 $.
To this end, set $ \omega := (\partial_t g - \frac12 \partial_{xx} g)/\q_1 $.
It is not hard to check that $ \q_1 = \qfn[\omega] $.
The property $ g \in \Ccsp^\infty $ and the positivity of $ \q_1 $ give $ \omega \in \Lsp^2([0,T]\times\R) := \Lsp^2 $.
The nonemptiness follows.

Turning to the existence, we begin by reformulating the variational problem \eqref{e.minimization}.
Fix a complete orthonormal basis $ \{e_n\}_{n\in\N} $ for $ \Lsp^2 $, where $ \N := \Z_{>0} $.
Let $ u<\infty $ denote the infimum in \eqref{e.minimization}.
To analyze this infimum, instead of the full $ \Lsp^2 $ space, it suffices to consider the subset
\begin{align*}
	\til{\Lsp} := \big\{ \omega = \sum\nolimits_{n\in\N} u_n e_n \in\Lsp^2 : |u_n| \leq u+1, \forall n \in \N \big\} \subset \Lsp^2.
\end{align*}
Consider the function $ f: \til{\Lsp} \to (0,\infty)^m $, $ f(\omega):= (\qfn[\omega](T,\xi_1),\ldots,\qfn[\omega](T,\xi_\mm)) $ that reads out the terminal values of $ \qfn[\omega] $.
We are concerned with minimizing $ \frac12 \norm{w}^2_{2} $ over the set $ C := f^{-1}(e^{\alpha_1},\ldots,e^{\alpha_\mm}) $, or more precisely proving that a minimizer exists.

To prove the existence, identify $ C $ and $ \til{\Lsp} $ as subsets of $ [-u-1,u+1]^{\N} $ and equip them with the product topology.
It is straightforward to verify that the function $ \til{\Lsp} \to [0,\infty): $ $ w \mapsto \frac12 \norm{w}^2_{2} $ is lower semicontinuous, and, by Tychonoff's theorem, $ \til{\Lsp} $ is compact.
Hence it suffices to show that $ C $ is closed.
By the argument in \cite[Lemma~3.7]{lin21}, the function $ f $ is continuous.
Hence $ C := f^{-1}(e^{\alpha_1},\ldots,e^{\alpha_\mm}) $ is closed.
\end{proof}

\section{Properties of $ \qfn[w] $}
\label{s.a.porperties.of.q}

\begin{lem}\label{l.bd.iter}
The inequality~\eqref{e.bd.hkiter} holds.
\end{lem}
\begin{proof}
Given that $ k(s,y)|_{s<0} := 0 $, in \eqref{e.hkiter}, we replace the time domain with $ \Delta_n(s,t) := \{s<s_1<\ldots<s_n:=t\} $. 
Apply the Cauchy--Schwarz inequality to bound $ |\hkiter{n}{|\w|}| $ by the product of an integral that involves only $ \w $ and an integral that involves only $ \hk $, and then evaluate both integrals.
To evaluate the latter integral, use $ \hk^2(s,y) = \frac{1}{2\sqrt{\pi s}} \hk(\frac{s}{2},y) $ and $ \int_{\Delta_n(s,t)} \prod_{i=1}^{n-1} \d s_i\, (s_{i}-s_{i-1})^{-1/2} \cdot (t-s_{n-1})^{-1/2} = (t-s)^{\frac{n}{2}-1}\pi^{\frac{n}{2}}/\Gamma(\frac{n}{2}) $.
\end{proof}

\begin{lem}
\label{l.q.bd}
There exists $ c=c(T,\norm{w}_2) $ such that
$
	\qfn[\w](t,x) \geq \frac{1}{c} \,(\hk(t) * q_\ic)(x).
$
\end{lem}
\begin{proof}
Without loss of generality, we consider $ \qic = \delta_0 $ and prove $ \qfn[\w](t,x) \geq \frac{1}{c} \hk(t,x) $.
Once this bound is proven, the result for a general $ \qic \geq 0 $ follows by convolving the bound with $ \qic $.

The first step is to develop a modified Feynman--Kac formula that bounds $ \qfn[\w] $ from below.
For $ r>0 $ to be specified later, consider $ D := [r,T]\times[-r^{-1},r^{-1}] $.
Within the Feynman--Kac formula~\eqref{e.FKq}, forgo the contribution from when the Brownian motion ever visits $ D $.
More precisely, letting $ \sigma := \max\{ s\leq t : B(t-s) \in D \} $ denote the first time (going backward) when the Brownian motion hits $ D $, we write
\begin{align}
	\label{e.FKq.killed}
	\qfn[\w](t,x) \geq \E_{x} \Big[ \exp\Big( \int_0^t \d s \, \w(s,B(t-s)) \Big) \delta_0(B(t)) \, \ind_{\{\sigma>t\}} \Big].
\end{align}
Let $\hkkilled(s,s',y,y') $ be the kernel of the Brownian motion annihilated upon visiting $ D $, namely $ \int_{\Omega} \d y' \,\hkkilled(s,s',y,y') = \P[ \{B(t-s')\in \Omega\} \cap \bigcap_{\sigma\in[s',s]}\{ B(t-\sigma) \notin D \} \, | \, B(t-s) =y ] $, for all Borel $ \Omega\subset \R $.
Taylor expanding the exponential in \eqref{e.FKq.killed} and exchanging the sum with the expectation give
\begin{align}
	\label{e.lwbd.q.1}
	\text{(right side of \eqref{e.FKq.killed})} = \hkkilled(0,t,x,0) + \sum_{n=2}^\infty \int_{\R} \d y \,\hkkilled^{n(*\w)}(0,t,x,0).
\end{align}
where $\hkkilled^{n(*\w)}$ is defined by replacing $ \hk $ with $\hkkilled$ in \eqref{e.hkiter}.

We proceed to estimate the right side of \eqref{e.lwbd.q.1}, and in doing so we will obtain the desired bound off $ [r,T]\times[-L,L] $, for some $ L\in[r,\infty) $.
First, by definition, $ \hkkilled(s,s',y,y') \leq \hk(s-s',y-y') $.
This property gives $ |\hkkilled^{n(*\w)}(0,t,0,y)| \leq |\hkiter{|\w|\ind_{D^\cc}}{n}(0,t,0,y)| $, where $ D^\cc := ([0,T]\times\R) \setminus D $.
The last expression can be bounded by \eqref{e.bd.hkiter}.
Summing the bound over $ n \geq 2 $ gives $ \sum_{n \geq 2}|\int_{\R} \d y \, \hkkilled^{n(*\w)}(0,t,0,x)| \leq  \norm{\w}_{2;D^\cc} \, c(T,\norm{\w}_{2;D^\cc}) \, \hk(t,x) $, where $ c $ dependents on $ \norm{\w}_{2;D^\cc} $ in such a way that $ c $ stays bounded as $ \norm{\w}_{2;D^\cc} \to 0 $.
Since $ w \in \Lsp^2 $, we have $ \norm{w}_{2;D^\cc} \to 0 $ as $ r\to 0 $.
Fix an $ r>0 $ small enough so that $ \norm{\w}_{2;D^\cc} \,c(T,\norm{\w}_{2;D^\cc}) \leq \frac12 $.
We have
\begin{align}
	\label{e.qfn.lwbd}
	\big(\qfn[w]\big)(t,x) \geq \hkkilled(0,t,x,0) - \tfrac12 \hk(t,x).
\end{align}
It is not hard to verify that, with $ r>0 $ being fixed, 
\begin{align*}
	\hkkilled(0,t,x,0)/\hk(t,x) 
	\begin{cases}
		\text{converges to } 1 \text{ as } |x|\to\infty, \text{ uniformly in } t\in[r,T],\\
		\text{is equal to } 1 \text{ for all } (t,x)\in(0,r)\in\R.
	\end{cases}
\end{align*}
Using these properties in \eqref{e.qfn.lwbd} gives the desired lower bound off $ [r,T]\times [-L,L] $, for some $ L\in[r,\infty) $.

Finally, with $ r $ and $ L $ being fixed, the desired lower bound within $ [r,T]\times [-L,L] $ follows since $ \qfn[w] $ is everywhere positive and continuous within $ [r,T]\times [-L,L] $.
Recall that $ \qfn[\w]|_{(0,T]\times\R} $ is continuous by definition and is positive by the Feynman--Kac formula~\eqref{e.FKq}.
\end{proof}

\section{A linear algebra tool}
\label{s.a.la}
\begin{lem}\label{l.la}
Let $ \vec{\eta}^{\,1}(v) $ and $ \vec{\eta}^{\,2}(v) \in \R^{\mm} $ be vectors parameterized by $ v>0 $, and assume they converge to $ \vec{\eta} \in \R^\mm $, as $ v \to 0 $.
Let $ M^1(v) $ and $ M^2(v) $ be $ \mm\times\mm $ matrices over $ \R $, and assume they converge to $ M $ as $ v \to 0 $, for some invertible $ M $.
If $ \vec{\gamma}(v) \in \R^{\mm} $ satisfies the inequalities $ \eta^1_{\ii}(v) \leq (M^1(v) \vec{\gamma}(v))_{\ii} $ and $ (M^2(v) \vec{\gamma}(v))_{\ii} \leq \eta^{2}_{\ii}(v) $, for $ \ii=1,\ldots,\mm $, then $ \vec{\gamma}(v) \to M^{-1}\vec{\eta} =: \vec{\gamma} $, as $ v\to 0 $.
\end{lem}
\begin{proof}
We write $ o(1) $ for a generic scalar, vector, or matrix that converges to zero as $ v\to 0 $.
By definition, $ \eta_{\ii} = (M \vec{\gamma})_{\ii} $.
Take the difference of the first given inequality and this equality, and use $ \vec{\eta}^{\,1}_{\ii}(v) - \vec{\eta}_{\ii}= o(1) $, $ M^1(v)-M = o(1) $ to simplify the result.
Doing so gives $ 0 \leq (M (\vec{\gamma}(v)-\vec{\gamma}))_{\ii} + o(1) $.
The same procedure applied to the second given inequality gives $ (M (\vec{\gamma}(v)-\vec{\gamma}))_{\ii} \leq o(1) $.
Combining these two bounds and taking the squared sum over $ \ii=1,\ldots,\mm $ yield $ |M(\vec{\gamma}(v)-\vec{\gamma})| = o(1) $, where $ |\ | $ denotes the Euclidean norm.
Since $ M $ is invertible, this implies $ |\vec{\gamma}(v)-\vec{\gamma}| = o(1) $.
\end{proof}

\section{Evaluating the conserved quantities}
\label{s.a.conserved}

\begin{lem}\label{l.phi.expansion}
Recall $ \varphi $ from \eqref{e.phi}.
For any $ n\in\Z_{>0} $ and for $ \lambda\in\R $, the expansion~\eqref{e.phi.expansion} holds.
\end{lem}
\begin{proof}
Divide the integral in \eqref{e.phi} into two: one on $ C_0\cap\{|\eta|< |\lambda|/2\} $ and the other on $ C_0\cap\{|\eta|>|\lambda|/2\} $.
The second integral decays super-polynomially fast in $ |\lambda| $.
In the first integral, write $ \frac{1}{\eta-\lambda} = -\sum_{k=0}^{2n} \lambda^{-k-1} \eta^{k} + O(|\eta|^{2n}|\lambda|^{-n-1}) $.
The result gives
\begin{align*}
	\varphi(\lambda)
	=
	-\sum_{k=0}^{2n} \lambda^{-k-1} \int_{C_0\cap\{|\eta|<|\lambda|/2\}}  \frac{\d\eta}{2\pi\img} \, \eta^{k} \log\big(1-\gamma e^{-\eta^2T/2}\big)
	+
	O(|\lambda|^{-2n-1}).
\end{align*}
Next, extend the range of integration from $ C_0\cap \{|\eta|<|\lambda|/2\} $ to $ \eta\in C_0 $. 
Doing so only costs an error that decays super-polynomially fast.
Since $ \log(1-\gamma e^{-\eta^2T/2})|_{C_0} $ remains unchanged upon $ \eta\mapsto -\eta $ (including when $ \gamma>1 $, see Figure~\ref{f.C0}), the resulting integral vanishes for odd $ k $s.
Reindex $ k \mapsto 2k $ and perform a change of variables $ \sqrt{T/2}\eta \mapsto \eta $.
This gives the desired expansion~\eqref{e.phi.expansion} when $ \gamma \leq 1 $.
When $ \gamma>1 $, deforming the contour to $ \R $ gives the desired expansion~\eqref{e.phi.expansion}.
\end{proof}

\begin{lem}\label{l.1-1.convex}
Recall $ \ratenosoliton $ and $ \ratesoliton $ from \eqref{e.1-1.rate}. Recall $ c_\star $ from \eqref{e.soliton.range}.
\begin{itemize}[leftmargin=20pt]
\item As $ \gamma $ increases in $ (-\infty,1] $, the function $ \ratenosoliton' $ strictly increases in $ (-\infty, c_{\star} ] $.
\item As $ \gamma $ increases in $ (0,1) $, the function $ \ratesoliton' $ strictly decreases in $ (c_{\star},\infty) $.
\end{itemize}
\end{lem}
\begin{proof}
It suffices to show $ \ratenosoliton''>0 $ and $ \ratesoliton''<0 $.
The first step is to find $ \Li_{5/2}''(\gamma)/\sqrt{4\pi} $.
Use $ \Li_{5/2}'(\gamma) = \Li_{3/2}(\gamma)/\gamma $, express $ \Li_{3/2}(\gamma) $ in the Bose--Einstein integral representation, and differentiate in $ \gamma $ to get
\begin{align}
	\label{e.Li5/2''.bd}
	\frac{1}{\sqrt{4\pi}}\Li_{5/2}''(\gamma) 
	= 
	\frac{1}{\pi} \int_0^\infty \d s \, \frac{s^{1/2}}{(e^{s}-\gamma)^2}.
\end{align}
This is positive for all $ \gamma \leq 1 $, so $ \ratenosoliton''>0 $ follows.
Next, to show $ \ratesoliton''<0 $ amounts to showing
\begin{align}
	\label{e.ratesoliton.goal}
	&&&\tfrac{1}{\sqrt{4\pi}}\Li_{5/2}''(\gamma) < \tfrac43 ((\log\tfrac{1}{\gamma})^{3/2})'',
	&&
	\gamma\in(0,1).&
\end{align}
To bound the left side, use $ e^{s} \geq s+1 $ and $ s^{1/2} \leq (s+(1-\gamma))^{1/2} $ in \eqref{e.Li5/2''.bd} and evaluate the resulting integral.
Doing so gives
$
	\frac{1}{\sqrt{4\pi}}\Li_{5/2}''(\gamma) 
	\leq
	\frac{2}{\pi} (1-\gamma)^{-1/2}.
$
As for the right side, calculate
$
	\frac43 ((\log(1/\gamma))^{3/2})''
	=
	(1+2\log(1/\gamma))\,\gamma^{-2} (\log(1/\gamma))^{-1/2}
	>
	\gamma^{-2} (\log(1/\gamma))^{-1/2}.
$
Consider the ratio of the bounds just obtained:
$
	f(\gamma)
	:=
	\frac{ \gamma^{-2} (\log(1/\gamma))^{-1/2} }{ \frac{2}{\pi} (1-\gamma)^{-1/2} }.
$
Differentiate it to get
\begin{align*}
	f'(\gamma)
	=
	\frac{\pi }{4(1-\gamma)^{1/2}\gamma^3(\log(1/\gamma))^{3/2}}
	\big(  1-\gamma + (4-3\gamma)\log\gamma) \big).
\end{align*}
In the last parenthesis, using the convexity inequality $ \log\gamma \leq \gamma-1 $ gives $ f'|_{(0,1)}<0 $.
Further, $ f(1^-) = \frac{\pi}{2}>1 $.
Hence $ f $ is always larger than $ 1 $, which gives the desired result~\eqref{e.ratesoliton.goal}.
\end{proof}

\begin{lem}\label{l.1-1.L2}
Fix any $ (\p,\q) $ as in Theorems~\ref{t.NS}, with the 1-to-1 initial-terminal condition.
We have $ \frac12 \norm{\p\q}^2_{2;[0,T]\times\R} = \frac12 \norm{\w}^2_{2;[0,T]\times\R} = \conserved_1 + T \conserved_3 $.
\end{lem}
\begin{proof}
The proof follows the calculations in \cite[Section~K, Supplementary Material]{krajenbrink21}.
First, straightforward differentiation and using the NLS equations \eqref{e.NS.q}--\eqref{e.NS.p} give
$
	\partial_t(\p\,\partial_x\q)
	=
	\frac12 \partial_x ( \p\,\partial_{xx}\q - \partial_x\p \cdot \partial_x \q + \p^2\q^2 ).
$
Multiply both sides by $ x $ and integrate the result over $ [\e,T-\e]\times\R $. 
On the right side, integrate by parts in $ x $ to move the `outermost' derivative $ \partial_x $ to $ x $.
Then, perform integration by parts to the term $ (- \partial_x\p \cdot \partial_x \q) $ to move the derivative on $ \p $ to $ \q $.
Sending $ \e\to 0 $ gives
\begin{align*}
	\lim_{\e\to 0}
	\Big( \int_{\R} \d x\, x\, (\p\,\partial_x\q) \Big) \Big|^{t=T-\e}_{t=\e}
	=
	-\int_{0}^{T} \d t \int_{\R} \d x \, \big( \p\,\partial_{xx}\q + \p^2 \q^2 \big)
	+
	\frac12 \int_{0}^{T} \d t \int_{\R} \d x \, \p^2 \q^2.
\end{align*}
On the right side, the first integral is recognized as $ T \conserved_3 $; see Property~\ref{propert.conserved} in Section~\ref{s.solving.scatter}.
The last term is exactly the squared $ \Lsp^2 $ norm of $ \w = \p\q $ that we are after.
On the left side, using $ \p(T,\Cdot) = \ptc = \gamma \delta_0 $ gives $ (\int_{\R} \d x\, x \, (\p\,\partial_x\q))|_{t=T-\e} \to 0 $.
For the contribution from $ t=\e $, integrate by parts to get $ -\int_{\R} \d x\, x \, (\p\,\partial_x\q) = \int_{\R} \d x \, x \, (\partial_x\p\cdot \q) + \int_{\R} \d x \,(\p\q) $.
For the first integral, set $ t=\e $, send $ \e\to 0 $, and use $ \q(0,\Cdot) = \qic = \delta_0 $.
We see that the integral converges to $ 0 $.
The second integral is $ \conserved_1 $; see~\eqref{e.1-1.conserved}.
This completes the proof.
\end{proof}

\section{Large scale asymptotics}
\label{s.a.large}

\begin{lem}\label{l.1-1.large.Srho}
Notation as in Section~\ref{s.1-1.large}, in particular $ T\mapsto 2N $.
For any fixed $ \sgamma\in(0,\infty) $ and for all $ (t,x) \in (0,2N) \times [0,\infty) $, $ s>0 $, and $ N \geq 1 $,
\begin{subequations}
\label{e.l.1-1.large}
\begin{align}
	\label{e.l.1-1.large.Srho1}
	\Srho(s;t,x)
	=&\,
	2\sqrt{\sgamma} \, e^{\sgamma t/2 - \sqrt{\sgamma}(s+x)} \, (1+O(1) e^{-\sgamma N} ) \ind_{\{s<\sqrt{\sgamma} t - x\}}
\\
	\label{e.l.1-1.large.Srho2}
	&
	+ \hk(t,s+x) \big( 1 + O(1) e^{-\sgamma N} + O(1) / \max\{ |\tfrac{s+x}{t}-\sqrt{\sgamma}|, \tfrac{1}{\sqrt{t}} \} \big),
\\
	\label{e.l.1-1.large.Srhot1}
	\Srhot(-s;t,x)
	=&\,
	-2\sqrt{\sgamma}\, e^{-\sgamma t/2 - \sqrt{\sgamma}(s+x)} \, (1+O(1) e^{-\sgamma N} ) \ind_{\{s<\sqrt{\sgamma} (2N-t) - x\}}
\\
	\label{e.l.1-1.large.Srhot2}
	+ e^{-\sgamma N}\, &\hk(2N-t,s+x) \big( -1 + O(1) e^{-\sgamma N} + O(1) / \max\{ |\tfrac{s+x}{2N-t}-\sqrt{\sgamma}|, \tfrac{1}{\sqrt{2N-t}} \} \big).
\end{align}
\end{subequations}
\end{lem}
\begin{proof}
We will prove the estimate of $ \Srho $, and the proof for $ \Srhot $ is similar.
Recall that $ \Srho(s;t,x) = \Srho(s+x;t,0) $ (see~\eqref{e.Srho}), so without loss of generality we consider $ x=0 $ only.
To simplify notation, \emph{throughout this proof we rename $ \sgamma \mapsto \gamma $.}

We begin by deriving an expression for $ \Srho(s;t,0) $ that is amenable for analysis.
Recall that $ \Sr := \Sb/\Sa $, insert the expressions \eqref{e.1-1.Sb} and \eqref{e.1-1.Sa.soliton} for $ \Sb(\lambda) $ and $ \Sa(\lambda) $ into the definition \eqref{e.Srho} of $ \Srho(s;t,0) $, and substitute $ \gamma \mapsto e^{-\gamma N} $ and $ T\mapsto 2N $.
We have
\begin{align*}
	\Srho(s;t,0)
	=
	\int_{\R+\img v_0} \frac{\d\lambda}{2\pi} \, e^{-\lambda^2 t/2 + \img\lambda s} \frac{\lambda+\img\sqrt{\gamma}}{\lambda-\img\sqrt{\gamma}} \, e^{-\varphi(\lambda)},
	\qquad
	\varphi(\lambda)
	=
	\int_{\R} \frac{\d\eta}{2\pi\img} \, \frac{\log(1-e^{-N(\gamma+\eta^2)})}{\eta-\lambda},
\end{align*}
where $ v_0 \in (\sqrt{\gamma},\infty) $.
In the last expression of $ \Srho $, write the fraction as $ 1 + (2\img\sqrt{\gamma}/(\lambda-\img\sqrt{\gamma})) $ and decompose the result into two integrals accordingly.
For the first integral, shift the contour $ (\R+\img v_0) $ to $ (\R+\img s/t) $ and perform the change of variables $ \lambda \mapsto \lambda + \img s/t $.
For the second integral, setting $ v_1 := s/t $ when $ |s/t-\sqrt{\gamma}|> 1/\sqrt{t} $ and $ v_1 := s/t+1/\sqrt{t} $ when $ |s/t-\sqrt{\gamma}| \leq 1/\sqrt{t} $, we shift the contour $ (\R+\img v_0) $ to $ (\R+\img v_1) $ and perform the change of variables $ \lambda \mapsto \lambda + \img v_1 $.
In shifting the contour we picked up a pole at $ \lambda=\img\sqrt{\gamma} $ when $ s/t < \sqrt{\gamma} - 1/\sqrt{t} $.
Altogether,
\begin{subequations}
\label{e.1-1.large.Srho.ameneable}
\begin{align}
\label{e.1-1.large.Srho.ameneable.1}
	\Srho(s;t,0)
	=\,
	&e^{-\frac{s^2}{2t}}
	\int_{\R} \frac{\d\lambda}{2\pi} \, e^{-\frac{t}{2}\lambda^2} \, e^{-\varphi(\lambda)}
\\
\label{e.1-1.large.Srho.ameneable.2}
	&+
	e^{-\frac{s^2}{2t}+\frac{t(v_1-s/t)^2}{2}}
	\int_{\R} \frac{\d\lambda}{2\pi} \, e^{-\frac{t}{2}\lambda^2} e^{\img(s-tv_1)\lambda}  \frac{2\img\sqrt{\gamma}}{\lambda-\img(\sqrt{\gamma}-v_1)} \, e^{-\varphi(\lambda)}
\\
\label{e.1-1.large.Srho.ameneable.3}
	&+
	2 \sqrt{\gamma} \, e^{\gamma t/2 -\sqrt{\gamma} s-\varphi(\img\sqrt{\gamma})} \, \ind_{\{ s/t < \sqrt{\gamma} - 1/\sqrt{t} \}}.
\end{align}
\end{subequations}

The next step is to bound $ \varphi $ for $ \im(\lambda) \geq 0 $.
In the preceding expression of $ \varphi $, deform the contour $ \R \mapsto \R -\img/\sqrt{N} $ and perform the change of variables $ \eta \mapsto \eta-\img/\sqrt{N} $.
In the resulting integral, the exponential is bounded, and the denominator is at least $ |1/\sqrt{N}| $ in absolute value.
Bound the entire integrand by $ \frac{c(\gamma)}{1/\sqrt{N}} |\exp(-N(\gamma+(\eta-\img/\sqrt{N})^2)| = O(1)\sqrt{N}\exp(-N(\gamma+\eta^2)) $, where $ \eta\in\R $, and evaluate the resulting integral.
Doing so gives
\begin{align}
	\label{e.phi.large.bd}
	|\varphi(\lambda)| = O(1) e^{-\gamma N},
	\qquad
	\im(\lambda) \geq 0.
\end{align}

We are now ready to estimate $ \Srho $.
Inserting \eqref{e.phi.large.bd} into the right side of \eqref{e.1-1.large.Srho.ameneable.1} and evaluating the integral give $ \hk(t,s)(1+O(1)e^{-\gamma N}) $.
Next, the integrand in \eqref{e.1-1.large.Srho.ameneable.2} is bounded by $ e^{-\lambda^2 t/2} O(1)/\max\{ 1/\sqrt{t}, |s/t-v_1|\} $ in absolute value.
Evaluate the integral of this bound and use $ |s/t-v_1| \leq 1/\sqrt{t} $ to bound the exponential factor in front of the integral. 
The result gives $ \eqref{e.1-1.large.Srho.ameneable.2} = O(1)\hk(t,s) / \max\{ |s/t-\sqrt{\gamma}|, 1/\sqrt{t} \} $.
For \eqref{e.1-1.large.Srho.ameneable.3}, use \eqref{e.phi.large.bd} to write $ e^{\varphi(\img\sqrt{\gamma})} = 1 + O(1) e^{-\gamma N} $.
Finally, note that when $ \{|s/t-\sqrt{\gamma}| \leq 1/\sqrt{t}\} $, we have $ 2 \sqrt{\gamma} \, e^{\gamma t/2 -\sqrt{\gamma} s} = O(1) \sqrt{t} \hk(t,s) $.
Hence we can replace the indicator in \eqref{e.1-1.large.Srho.ameneable.3} with $ \ind_{\{ s/t<\sqrt{\gamma} \}} $ without changing the result.
\end{proof}

\begin{lem}\label{l.1-1.large.orho}
The expressions in \eqref{e.1-1.large.orho} hold.
\end{lem}
\begin{proof}
We will prove the statement for $ {}_{\smallp}\,\orho\,{}_{\smallm} $, and the proof for $ {}_{\smallm}\,\orhot\,{}_{\smallp} $ is similar.
The proof uses Lemma~\ref{l.1-1.large.Srho}.
\emph{We follow the convention in the proof of Lemma~\ref{l.1-1.large.Srho} to rename $ \sgamma \mapsto \gamma $.}
In \eqref{e.l.1-1.large.Srho1}--\eqref{e.l.1-1.large.Srho2}, set $ s \mapsto (s-s') $ and multiply each term by $ \ind_{\{s>0\}} $ and $ \ind_{\{s'<0\}} $.
The indicators are relevant because we are interested in $ {}_{\smallp}\,\orho\,{}_{\smallm} $.
In the result, the first term carries the indicators $ \ind_{\{s>0\}}\ind_{\{s'<0\}}\ind_{\{s -s' \leq \sqrt{\gamma} t - x\}} $, and these indicators together imply $ x < \sqrt{\gamma} t $.
Write 
\begin{align*}
	\ind_{\{s>0\}}\ind_{\{s'<0\}}\ind_{\{s -s' \leq \sqrt{\gamma} t - x\}}
	=
	\ind_{\{s>0\}}\ind_{\{s'<0\}}\ind_{\{x<\sqrt{\gamma} t \}} 
	- 
	\ind_{\{s>0\}}\ind_{\{s'<0\}}\ind_{\{x<\sqrt{\gamma} t \}}\ind_{\{s-s'>\sqrt{\gamma} t -x \}}.
\end{align*}
We now recognized the term $ e^{\gamma t/2 - \sqrt{\gamma} x} \, 2\sqrt{\gamma} e^{-\sqrt{\gamma}(s-s')} \ind_{\{s>0\}} \, \ind_{\{s'<0\}}\ind_{\{x<\sqrt{\gamma} t \}} $ as the kernel of $ \orho_1 $; see \eqref{e.1-1.large.orho.2}.
Collecting the remaining terms gives
\begin{align}
	\label{e.1-1.large.orho.remaining}
	\ind_{\{s>0\}}
	\Big(
		e^{\gamma t/2 - \sqrt{\gamma} x} \, 2\sqrt{\gamma} e^{-\sqrt{\gamma}(s-s')} 
		\, 
		O\big( 
			\ind_{\{s-s'>\sqrt{\gamma} t -x \}}
			+
			e^{-\gamma N}
		\big)
		\,
		\ind_{\{x<\sqrt{\gamma} t \}}
		+
		\eqref{e.l.1-1.large.Srho2}
	\Big)
	\ind_{\{s'<0\}}\,.
\end{align}
Let $ \orho_2 $ denote the operator whose kernel is \eqref{e.1-1.large.orho.remaining}.
We proceed to bound the operator norm of $ \orho_2 $ by bounding its Hilbert--Schmidt norm; indeed, $ \norm{\ }_{\op} \leq \norm{\ }_{\HS} $.
With the aid of the identity $ \norm{ \mathbf{f} }^2_{\HS} = \int_{\R^2} \d s \d s' \, |f(s,s')|^2 $, we can bound $ \norm{ \orho_2 }_{\HS} $ by evaluating the corresponding integral of~\eqref{e.1-1.large.orho.remaining}, and the result gives the operator-norm bound in \eqref{e.1-1.large.orho.3}.
The $ \Lsp^2 $ bound in \eqref{e.1-1.large.orho.3} follows similarly from~\eqref{e.1-1.large.orho.remaining}.
\end{proof}

\section{The NLS equations through the lens of classical field theory}
\label{s.a.classical.field.theory}

Here we give a \emph{physics} derivation of the NLS equations from the variational problem.
This derivation is not used elsewhere in this paper; the purpose is just to highlight the fact that the NLS equations can be understood as certain Hamilton equations in the framework of classical field theory \cite[Chapter~13]{goldstein02}.
To begin, rearranging terms as $ \w = (\partial_t \q - \frac12 \partial_{xx} \q)/\q $ and putting $ \partial_t \q = \dot{\q} $, we recognize the objective function $ \frac12 \norm{w}^2_{2;[0,T]\times\R} $ in \eqref{e.minimization} as a Lagrangian
\begin{align*}
	L(\q,\dot{\q}) := \int_{\R} \d x \ \frac{1}{2 \q^2 } \Big(\dot{\q} - \frac12 \partial_{xx} \q \Big)^2.
\end{align*}
A minimizer of a Lagrangian should satisfy the corresponding Euler--Lagrange equation, which is second-order in time, and the second-order equation can be converted to a system of first-order equations: the Hamilton equations.
The conversion starts with defining the canonical momentum $ \p := \delta L/\delta \dot{\q} $ and the Hamiltonian $ H(\q,\p) := \int_{\R} \d x \, \p \dot{\q} - L $.
The Euler--Lagrange equation is formally equivalent to the Hamilton equations $ \partial_t \q = \delta H/\delta\p $ and $ \partial_t \p = - \delta H/\delta\q $.
For the specific Lagrangian considered here,
\begin{align*}
	\p := \frac{\delta L}{\delta \dot{\q}} = \frac{1}{\q^2} \Big(\dot{\q} - \frac12 \partial_{xx} \q\Big),
	\qquad
	H := \int_{\R} \d x \, \p \dot{\q} - L  = \int_{\R} \d x \, \frac{1}{2}  \Big( \p \, \partial_{xx} \q + \p^2 \q^2 \Big),
\end{align*}
and the corresponding Hamilton equations read
\begin{align*}
	&\partial_t q = \frac{\delta H}{\delta p} = \frac12 \partial_{xx} q + p q^2,
	&&
	\partial_t p = -\frac{\delta H}{\delta q} = -\frac12 \partial_{xx} p - q p^2,&
\end{align*}
which are exactly the NLS equations.
Also, note that, with $ \dot{q} - \frac12 \partial_{xx} q = wq $, we have $ p = w/q $.

\bibliographystyle{alpha}
\bibliography{wnt-kpz}

\end{document}